\documentclass[reqno]{amsart}
\usepackage[english]{babel}
\usepackage[utf8]{inputenc}
\usepackage[T1]{fontenc}
\usepackage{amsmath,amssymb,amsthm,mathtools,enumitem, dynkin-diagrams}
\usepackage{mathrsfs} 
\usepackage{lmodern,subcaption,bbm}
\usepackage[all,cmtip]{xy}
\usepackage[linktocpage=true, pdfusetitle,pagebackref]{hyperref}
\renewcommand*{\backref}[1]{}
\usepackage{array}
\usepackage{csquotes}
\usepackage{multirow}

\calclayout
\newcolumntype{C}{>{$}c<{$}}

\usepackage{xcolor}
\hypersetup{
    colorlinks,
    linkcolor={red!80!black},
    citecolor={green!50!black},
    urlcolor={blue!80!black}
}
\makeatletter
\def\th@plain{
  \thm@notefont{}
  \itshape 
}
\def\th@definition{
  \thm@notefont{}
  \normalfont 
}
\makeatother
\newtheorem{theorem}{Theorem}[section]

\newtheorem{corollary}[theorem]{Corollary}

\newtheorem{proposition}[theorem]{Proposition}
\newtheorem{lemma}[theorem]{Lemma}

\theoremstyle{definition}
\newtheorem{definition}[theorem]{Definition}

\newtheorem{example}[theorem]{Example}
\newtheorem{remark}[theorem]{Remark}
\newtheorem{notation}[theorem]{Notation}

\DeclarePairedDelimiter{\abs}\lvert\rvert
\DeclarePairedDelimiter{\norm}\lVert\rVert

\renewcommand{\d}[1][t]{\ensuremath{\left.\frac{d}{d#1}\right|_{#1=0}}}
\newcommand\restr[2]{{
  \left.\kern-\nulldelimiterspace 
  #1 
  \vphantom{\big|} 
  \right|_{#2} 
  }}
  
\makeatletter
\newcommand{\sn}[3]{\ensuremath{\norm{#3}_{H^{#1}(#2)}}}
\newcommand{\fourier}[1]{\ensuremath{\sum_{\tau\in\hat K_M}\pi_{Y_\tau}(\pi_{Y_\tau}^*(#1))}}
\newcommand*\bigcdot{\mathpalette\bigcdot@{.5}}
\newcommand*\bigcdott{\mathpalette\bigcdot@{.7}}
\newcommand*\bigcdot@[2]{\mathbin{\vcenter{\hbox{\scalebox{#2}{$\m@th#1\bullet$}}}}}
\g@addto@macro\bfseries{\boldmath}
\makeatother
\newcommand{\lra}{\ensuremath{\leftrightarrow}}
\newcommand{\lraomega}{\ensuremath{\xleftrightarrow\omega}}

\newcommand{\notleftrightomega}{\mathrel{\ooalign{$\xleftrightarrow\omega$\cr\hidewidth$/$\hidewidth}}}
\newcommand{\Ad}{\ensuremath{\on{Ad}}}

\newcommand{\hyp}[4]{\ensuremath{F\left(#1,#2,#3,#4\right)}}
\newcommand{\mb}[1]{\ensuremath{\mathbb{#1}}}
\newcommand{\mc}[1]{\ensuremath{\mathcal{#1}}}
\newcommand{\mf}[1]{\ensuremath{\mathfrak{#1}}}
\newcommand{\on}[1]{\ensuremath{\operatorname{#1}}}

\newcommand{\Hom}[3]{\ensuremath{\on{Hom}_{#1}(#2,#3)}}
\newcommand{\rhoa}{\ensuremath{\rho}}
\newcommand{\rhoc}{\ensuremath{\rho_c}}

\newcommand{\R}{\ensuremath{\mathbb{R}}}
\newcommand{\N}{\ensuremath{\mathbb{N}}}
\newcommand{\Z}{\ensuremath{\mathbb{Z}}}
\newcommand{\C}{\ensuremath{\mathbb{C}}}
\renewcommand{\H}{\ensuremath{\mathbb{H}}}
\newcommand{\ps}[2]{\ensuremath{H_{#1,#2}}}
\newcommand{\pss}[2]{\ensuremath{H_{#1,#2}^\infty}}
\newcommand{\psd}[2]{\ensuremath{H_{#1,#2}^{-\infty}}}
\newcommand{\psc}[2]{\ensuremath{H_{#1,#2}^{\mathrm{cpt}}}}

\newcommand{\pst}[1]{\ensuremath{H_{#1}}}

\newcommand{\pstd}[1]{\ensuremath{H_{#1}^{-\infty}}}

\newcommand{\I}{\ensuremath{\mathbf{I}}}
\newcommand{\Ck}[1]{\ensuremath{k_C(#1)}}
\newcommand{\Cp}[1]{\ensuremath{p_C(#1)}}
\newcommand{\Ik}[1]{\ensuremath{k_I(#1)}}
\newcommand{\In}[1]{\ensuremath{n_I(#1)}}

\newcommand{\SO}[1]{\ensuremath{\mathrm{SO}(#1)}}
\newcommand{\GSO}[1]{\ensuremath{\mathrm{SO}_0(#1,1)}}

\newcommand{\GSU}[1]{\ensuremath{\mathrm{SU}(#1,1)}}

\newcommand{\GSp}[1]{\ensuremath{\mathrm{Sp}(#1,1)}}
\newcommand{\GF}{\ensuremath{\mathrm{F}_4}}

\newcommand{\intd}{\ensuremath{\,\mathrm{d}}}
\newcommand{\ov}[1]{\ensuremath{\overline{#1}}}
\date{\today}

\begin{document}
\pagenumbering{arabic}
\title[Spectral Correspondences for Rank One Locally Symmetric Spaces]{Spectral Correspondences for Rank One Locally Symmetric Spaces - The Case of Exceptional Parameters}
\author[Arends]{Christian Arends}
\address{Institut f\"ur Mathematik, Universit\"at Paderborn, Warburger Str. 100,
        33098 Paderborn, Germany}
        \email{arendsc@math.upb.de}
\author[Hilgert]{Joachim Hilgert}
\address{Institut f\"ur Mathematik, Universit\"at Paderborn, Warburger Str. 100,
        33098 Paderborn, Germany}
        \email{hilgert@math.upb.de}
\maketitle 

\begin{abstract}
In this paper we complete the program of relating the Laplace spectrum for rank one compact locally symmetric spaces with the first band Ruelle-Pollicott resonances of the geodesic flow on its sphere bundle. This program was started in \cite{FF03} by Flaminio and Forni for hyperbolic surfaces, continued in \cite{DFG} for real hyperbolic spaces and in \cite{GHWb} for general rank one spaces. Except for the case of hyperbolic surfaces (see also \cite{GHWa}) a countable set of exceptional spectral parameters always left untreated since the corresponding Poisson transforms are neither injective nor surjective. We use vector valued Poisson transforms to treat also the exceptional spectral parameters. For surfaces the exceptional spectral parameters lead to discrete series representations of $\mathrm{SL}(2,\mathbb R)$ (see \cite{FF03, GHWa}). In higher dimensions the situation is more complicated, but can be described completely. 
\end{abstract}

\section{Introduction}
Dynamical systems with additional symmetry are surprisingly rigid. One manifestation of this observation is the close connection between geodesic flows on locally symmetric spaces and their quantizations, the Laplace-Beltrami wave kernels. This was first observed for tori in the form of the Poisson summation formula and its non-commutative analog, the Selberg trace formula, where the length spectrum of closed geodesics and the spectrum of the Laplacian enter. In specific cases correspondences on the level of eigenfunctions were established about twenty years ago \cite{LZ01,FF03,DH05,M06,Poh12}.

In \cite{DFG} Dyatlov, Faure and Guillarmou showed that the spectrum of the geodesic flow on compact hyperbolic manifolds essentially decomposes into bands, the first of which is in one to one correspondence with the Laplace spectrum. For these spectral values they also constructed linear isomorphisms between the corresponding eigenspaces.  In this context  \emph{essentially} means that there is a countable set of explicitly known spectral values for which the methods do not apply. 

In \cite{GHWa} the very explicit information available for hyperbolic surfaces was used to establish spectral correspondences also for the exceptional spectral values. In these cases the quantum side turns out to be related to the discrete series representations of $\mathrm{SL}(2,\R)$, whereas the regular spectral values were related to irreducible unitary spherical principal series representations. 

The theory of quantum-classical spectral correspondences with spherical principal series representations on the quantum side was extended to all rank one compact locally symmetric spaces in \cite{GHWb}. In this paper we complete the program for these spaces by establishing quantum-classical spectral correspondences on the level of eigenvectors for all exceptional spectral values.

We describe the setting in a little more detail. Let $G$ be a non-compact simple Lie group of real rank one and $\Gamma$ be a co-compact discrete subgroup of $G$. For simplicity we assume that $G$ has finite center and $\Gamma$ is torsion free. We fix a maximal compact subgroup $K$ and observe that the locally symmetric space $\Gamma\backslash G/K$ is a compact Riemannian manifold. Therefore its (elliptic) Laplace-Beltrami operator has discrete spectrum on $L^2(\Gamma\backslash G/K)$ with smooth eigenfunctions lifting to $\Gamma$-invariant eigenfunctions on $G/K$. Note that on $G/K$ the Laplace-Beltrami operator comes from a Casimir element and generates the algebra of $G$-invariant differential operators. For generic spectral parameters $\mu$ the eigenfunctions generate an irreducible $G$-representation which is equivalent to a spherical principal series representation $H_\mu$. The corresponding intertwiner is the Poisson transform $P_\mu$. So, generically the Laplace-Beltrami eigenspaces $^\Gamma E_{-\mu}$ can be identified with the $\Gamma$-invariant distribution vectors $^\Gamma \pstd{\mu}$ in the corresponding spherical principal series representation, where the normalization of the spectral parameters is taken from \cite{GHWb}. 

The word \emph{generic} in the previous paragraph can be given a precise meaning. Let  $\mathfrak g_0$ be the Lie algebra of $G$ and $\mathfrak g$ the complexification of $\mathfrak g_0$ (we use the analogous convention for all subspaces of $\mathfrak g_0$). The eigenvalues of the Laplace-Beltrami operator on $G/K$ are parameterized by elements of $\mathfrak a^*$ via the Harish-Chandra isomorphism, where $\mathfrak g=\mathfrak k+\mathfrak p$ is the Cartan decomposition of the Lie algebra $\mathfrak g$ fixed by the choice of $K$ and $\mathfrak a_0$ is a maximal abelian subspace of $\mathfrak p_0$. The parameters are unique up to the action of the Weyl group $W=N_K(\mathfrak a)/Z_K(\mathfrak a)$. A spectral parameter $\mu$ is \emph{generic} if and only if it is \emph{not} a zero of the Harish-Chandra $\mathbf e$-function which in turn is equivalent to the bijectivity of the intertwining Poisson transform $P_\mu$. Thus the exceptional parameters alluded to in the title of the paper are the zeros of the $\mathbf e$-function.

In the case of compact hyperbolic surfaces (see \cite{GHWa}) the exceptional spectral parameters are related to discrete series representations, which can be realized as smooth (in fact, holomorphic or anti-holomorphic) sections of certain $G$-homogeneous vector bundles over $G/K$. In these spaces of sections one has the action of a suitable \emph{Bochner-Laplace operator} (see \cite[Lemma~2.2]{Ol}). 
While these representations are no longer completely determined by the action of the Bochner-Laplacian, they are still irreducible unitary representations of $G$ obtained by a suitable vector valued Poisson transform. This part can be generalized and we view the $\Gamma$-invariant sections, which descend to the locally symmetric space, as part of the quantization of the cotangent bundle of this space.   

The cotangent bundle $T^*(\Gamma\backslash G/K)=\Gamma\backslash G\times_K\mathfrak p_0^*$ of $\Gamma\backslash G/K$ is foliated into the cosphere bundles $\Gamma\backslash G/Z_K(\mathfrak a)\times \{r\}$ with $r\in \mathfrak a_0^*\equiv \mathbb R$ determining the radius and the zero section $\Gamma\backslash G/K$. Each leaf of the foliation is invariant under the geodesic flow. On the zero section it is trivial, whereas on the cosphere bundles it is given by the right action $\Gamma\backslash G/M\times A\to \Gamma\backslash G/M,\  (gM,a)\mapsto gaM$, where we use the standard abbreviation $M$ for the centralizer $Z_K(\mathfrak a)$ and set $A=\exp(\mathfrak a_0)$. This decomposition reduces the spectral analysis of the geodesic flow to the $A$-action on $\Gamma\backslash G/M$. This action is Anosov as one sees from the Bruhat decomposition $T(\Gamma\backslash G/M)=G\times_M(\mathfrak n_0^++\mathfrak a_0+\mathfrak n_0^-)$, where $\mathfrak g_0=\mathfrak k_0+\mathfrak a_0+\mathfrak n_0^\pm$ are the two Iwasawa decompositions of $\mathfrak g_0$ associated with the two possible orderings of the set $\Sigma$ of restricted roots in $\mathfrak a_0^*$. The approach to Ruelle-Pollicott resonances for the geodesic flow used in \cite{GHWb} makes use of the set $\mathcal D_+'(\Gamma\backslash G/M)$ consisting of the distributions $u\in \mathcal D'(\Gamma\backslash G/M)$ whose wavefront set  $\mathrm{WF}(u)$ is contained in the annihilator $\Gamma\backslash G\times_M (\mathfrak n_0^++\mathfrak a_0)^\perp\subseteq T^*(\Gamma\backslash G/M)$. Then the set of \emph{resonant states} for the spectral parameter $\mu\in\mathfrak a^*$ is defined as
$$\mathrm{Res}(\mu):= \{u\in \mathcal D_+'(\Gamma\backslash G/M)\mid \forall H\in\mathfrak a_0:\ H\cdot u+ \mu(H)u=0\},$$
where $H$ acts as a left-invariant vector field on $G/M$ descending to $\Gamma\backslash G/M$. A spectral parameter $\mu\in\mathfrak a^*$ is called a \emph{Ruelle-Pollicott resonance} if $\mathrm{Res}(\mu)\not=0$. The Ruelle-Pollicott resonances form a discrete set and the corresponding spaces of resonant states are finite dimensional. A \emph{first band resonant state} is a resonant state $u$ which satisfies $X\cdot u =0$, where $X$ is any vector field on $\Gamma\backslash G/M$ which is a section of the subbundle $G\times_M \mathfrak n_0^-\subseteq T(\Gamma\backslash G/M)$. We denote the space of first band resonant states for the spectral parameter $\mu\in\mathfrak a^*$ by $\mathrm{Res}^0(\mu)$. In the case of generic spectral parameters the \emph{quantum-classical spectral correspondence} says that the push-forward of the canonical projection $\mathrm{pr}: \Gamma\backslash G/M\to \Gamma\backslash G/K$ is a linear isomorphism $\pi_*: \mathrm{Res}^0(\mu-\rho) \to {}^\Gamma E_\mu$, where $\rho\in\mathfrak a_0^*$ is the usual half-sum of positive restricted roots counted with multiplicity (see \cite[Thm.\@ 4.5]{GHWb}). 

The strategy for our extension of the quantum-classical correspondence to exceptional spectral parameters is as follows. As in the generic case (see \cite[\S~3.2]{GHWb}) we start by lifting the first band Ruelle-Pollicott resonances to $\Gamma$-invariant distributions on the global symmetric space. The lifted spaces can be interpreted in terms of spherical principal series (that part works for all spectral parameters, see \cite[Prop.~3.8]{GHWb}) and the first band resonant states $\mathrm{Res}^0(-\mu-\rho)$ correspond to the space $^\Gamma H_\mu^{-\infty}$ of $\Gamma$-invariant distribution vectors of the corresponding principal series. For an exceptional spectral parameter $\mu$ the corresponding principal series $H_\mu$ is no longer irreducible. But it has a managable composition series and it turns out that the $\Gamma$-invariant distribution vectors are all contained in the socle (i.e. the sum of all irreducible subrepresentations) of the representation, see Thm.~\ref{thm:Gamma_inv_ps}. In each of the rank one cases except $\mathrm{SO}_0(2,1)$ (the case of surfaces, see \cite{GHWa}) the socle turns out to be irreducible with a unique minimal $K$-type $\tau_\mu$ (see Thm.~\ref{thm:socle}) and we can show that the vector valued Poisson transform associated with this $K$-type (sum of $K$-types in the case of surfaces) is injective, see Prop.~\ref{prop:Poisson_sum}. The image consists of spaces of $\Gamma$-invariant sections of vector bundles over $\Gamma\backslash G/K$ and we have a quantum-classical correspondence as soon as we have characterized the image of this Poisson transform.

We achieve the characterization of the image of the minimal $K$-type Poisson transform via Fourier expansions of $M$-invariant functions with respect to $M$-spherical $K$-representations. More precisely, we determine necessary and sufficient conditions for a Fourier series to represent a distribution vector of the reducible spherical principal series $H_\mu$, see Thm.~\ref{thm:fourier_char}, where the conditions are given in terms of generalized gradients (see \cite{BransonOlafsson}). In each of the cases it is possible to  determine a $G$-invariant system of differential equations on the sections of the homogeneous bundle $G\times_K V_{\tau_\mu}$ given by the minimal $K$-type $(\tau_\mu,V_{\tau_\mu})$ of the socle such that on the space of $\Gamma$-invariant solutions we can write down an explicit boundary value on $K/M$ in terms of Fourier coefficients, see Thms.~\ref{thm:spectral_corres_SUn2}, \ref{thm:spectral_corres_Spn} and \ref{thm:spectral_corres_F4}. Then our Fourier characterization of $H_\mu^{-\infty}$ allows us to show that the boundary values are contained in $^\Gamma H_\mu^{-\infty}$. In the case of $\mathrm{SO}_0(n,1)$ and for most exceptional spectral parameters in the case of $\mathrm{SU}(n,1)$ we have an alternative (and simpler) characterization of the vector valued Poisson transform, which is based on techniques developed in \cite{M89} to study Cauchy-Szegö maps for $\mathrm{SU}(n,1)$, see Thms.~\ref{thm:spectral_corres_SOn} and \ref{thm:spectral_corres_SUn1}.   

We can explicitly determine the socle of all reducible spherical principal series representations in rank one (see Theorem \ref{thm:socle}), and we see that the surface case, which so far was the only one known, is quite untypical. Not only is it the only case where the socle is not irreducible, it is also one of the very few cases in which the representation generated by the resonant states belongs to the discrete series. This is only the case for $\mathrm{SO}_0(2,1)$ (surfaces), $\mathrm{SU}(2,1)$, $\mathrm{Sp}(2,1)$ and $F_4$, see Theorem \ref{thm:langlands}. On the other hand it turns out that all of these representations are unitarizable, see Theorem \ref{thm:socle}. We can determine the Langlands parameters (see Theorem \ref{thm:langlands}), and in some cases geometric realizations, e.g. as solution spaces of differential equations are well-known (see \cite{Ol,G88}). But for most cases we did not find such descriptions in the literature, so that our geometric realization as solution spaces of differential equations describing the  images of minimal $K$-type Poisson transforms might actually be new.

As mentioned above, our results complete the picture of first band quantum-classical correspondences for compact locally symmetric spaces of rank one. In higher rank an analogous  quantum-classical correspondence for generic spectral parameters has been established in \cite{HWW21}. Extending that result to exceptional spectral parameters will be substantially harder as the information available on composition series of spherical principal series is much less explicit in higher rank. Moreover, some of the multiplicity one results we use (Props.\@ \ref{prop:multone}, \ref{prop:soc_multone_rank_one}, \ref{prop:tensor_decomp_eq_rk}) or prove (Prop.\@ \ref{prop:mult_one}) here are not always available in higher rank. As far as non-compact locally symmetric spaces are concerned, one has to replace the (discrete) spectrum of the algebra of invariant differential operators by a suitable concept of quantum resonances. So far one only has quantum-classical correspondences for convex co-compact real hyperbolic spaces and, for dimensions larger than two, only generic spectral parameters \cite{GHWa,Ha20}. For locally symmetric spaces with cusps the results on record are either very special (e.\@g.\@ \cite{LZ01,M06}) or else give only very rough information (e.\@g.\@ \cite{DH05}). In view of \cite{GW17,Poh12}, however, a quantum-classical correspondence for surfaces seems to be within reach. Finally, we mention \cite{KW21}, where quantum-classical correspondences for lifts of geodesic flows on compact locally symmetric spaces of rank one are treated for generic spectral parameters. That exceptional spectral parameters occur also in such situations can be seen from \cite{KW20}, where the authors have to leave out the case of three dimensional hyperbolic spaces because the Gaillard Poisson transform they use is not bijective.

We conclude this introduction with a brief description of the way the paper is structured. In Section \ref{sec:red_princ_series} we collect the information on principal series representations and their $K$-types. In Section \ref{sec:PT} we recall the scalar Poisson transforms for symmetric spaces and introduce the minimal $K$-type Poisson transforms. In Section~\ref{sec:Gamma_inv} we show that $\Gamma$-invariant distribution vectors in principal series representations have to be contained in the socle of the representation. Moreover, we determine the socles and their minimal $K$-types in all cases. In Section \ref{sec:Fourier} we study Fourier expansions of $M$-invariant  functions with respect to $M$-spherical  $K$-representations. Apart from convergence issues we deal with the technicalities needed to characterize the spherical principal series representations in terms of Fourier expansions. In Section \ref{sec:spec_corres} we complete the determination of the spectral correspondences by describing the $\Gamma$-invariant vectors in the image of the minimal $K$-type Poisson transform. Appendix~\ref{app:comp_scalars} is devoted to the case by case calculations we could not avoid in proving the technical results of Sections~\ref{sec:Gamma_inv} and \ref{sec:Fourier}.

\section{Reducible principal series}\label{sec:red_princ_series}
In this section we recall the main facts about principal series representations we use in this paper. 
\subsection{Basic notation}
Let $G$ be a noncompact, connected, real, semisimple Lie group with finite center and $\Gamma\leq G$ a co-compact, torsion free lattice. We denote the Iwasawa decomposition of G by $G=KAN$. The $K$-, $A$-, or $N$-component in the Iwasawa decomposition is denoted by $k_I$, $a_I$, or $n_I$, respectively. Let $M\coloneqq Z_K(A)$ denote the centralizer of $A$ in $K$. The corresponding Lie algebras will be denoted by $\mf{g}_0,\mf{k}_0,\mf{a}_0,\mf{n}_0,\mf{m}_0$ with complexifications $\mf g,\mf k,\mf a,\mf n,\mf m$. Moreover, let $\mf g=\mf k\oplus\mf p$ be the Cartan decomposition and denote the corresponding Cartan involution by $\theta$. Associated with the $\mf a_0$-action we define the restricted root spaces $\mf g^\alpha$ corresponding to the restricted roots $\Sigma\subset\mf a_0^*$. Furthermore, we have the Bruhat decomposition given by $\mf g_0=\mf a_0\oplus\mf m_0\oplus\bigoplus_{\alpha\in\Sigma}\mf g^\alpha$. The Iwasawa decomposition determines a positive system $\Sigma^+\subset\Sigma$. The half-sum of positive roots is denoted by $\rhoa\coloneqq\frac{1}{2}\sum_{\alpha\in\Sigma^+}m_\alpha\alpha$ with the multiplicities $m_\alpha\coloneqq\dim_\R\mf g^\alpha$. If $\log:A\to\mf a_0$ denotes the logarithm on $A$ and $\mu\in\mf a^*$ we define $a^\mu\coloneqq e^{\mu(\log a)}$. By $\hat K$ (resp.\@ $\hat G,\, \hat M$) we denote the equivalence classes of irreducible unitary representations of $K$ (resp.\@ $G,\, M$). The Weyl group of $(\mf g_0,\mf a_0)$ is denoted by $W$. Let $\kappa$ denote the Killing form of $\mf g$ and $\mc U(\mf g)$ denote the universal enveloping algebra of $\mf g$. For $H\in\{K,M\}$ and a finite-dimensional representation $(\tau,Y)$ of $H$ we define the associated vector bundle $G\times_HY$ as the quotient $(G\times Y)/\sim$, where
\begin{gather*}
\forall g\in G,\, h\in H,\, v\in Y\colon (g,v)\sim(gh,\tau(h^{-1})v).
\end{gather*}
We always identify the space of smooth sections of this bundle with
\begin{gather*}
C^\infty(G\times_HY)\coloneqq\{f\in C^\infty(G,Y)\mid \forall g\in G,\, h\in H\colon f(gh)=\tau(h^{-1})f(g)\}.
\end{gather*}

\subsection{Realizations of principal series representations}
The principal series representations can be realized in different ways (``pictures'') all of which have their advantages. Let $(\sigma,V_\sigma)\in\hat M$ with inner product $\langle\cdot,\cdot\rangle_\sigma$ and $\mu\in\mf a^*$. Denote by $L^2(K,V_\sigma)$ the space of $V_\sigma$-valued functions which are $L^2$ with respect to the normalized Haar measure $\intd k$ on $K$. 

In the \emph{induced picture} the representation space $\ps{\sigma}{\mu}$ is given by all measurable functions $f:G\to V_\sigma$ such that
\begin{enumerate}
\item $f(gman)=a^{\mu-\rhoa}\sigma(m^{-1})f(g)$ for all $g\in G,m\in M,a\in A,n\in N$,
\item $\restr{f}{K}\in L^2(K,V_\sigma)$.
\end{enumerate}
The representation is given by
\begin{gather*}
(\pi_{\sigma,\mu}(g)f)(x)\coloneqq f(g^{-1}x),\qquad g,x\in G,\ f\in \ps{\sigma}{\mu}.
\end{gather*}
Endowed with the norm $\norm{f}^2\coloneqq\int_K\norm{f(k)}^2_\sigma\intd k$ this realization is a Hilbert space representation. The parametrization is chosen such that $\ps{\sigma}{\mu}$ is unitary if $\mu\in i\mf a_0^*$ is imaginary. 

The Iwasawa decomposition shows that a function in $\ps{\sigma}{\mu}$ is completely determined by its restriction to $K$. Thus, the surjective isometry
\begin{gather}\label{eq:cpt_induced}
\ps{\sigma}{\mu}\cong\psc{\sigma}{\mu},
\end{gather}
where $\psc{\sigma}{\mu}$ denotes the Hilbert space of all functions $f$ in $L^2(K,V_\sigma)$ such that $f(km)=\sigma(m^{-1})f(k)$ for all $k\in K,\ m\in M$, endowed with the same norm as above, gives another realization of the principal series representation. This realization is called the \emph{compact picture}. Note that the representation space does not depend on $\mu$. However, in this picture the $G$-action is more complicated compared to the induced picture. It is induced by the action $\pi_{\sigma,\mu}$ via the isometry above and given by
\begin{gather*}
(\pi_{\sigma,\mu}^{\mathrm{cpt}}(g)f)(k)\coloneqq a_{I}(g^{-1}k)^{\mu-\rhoa}f(k_{I}(g^{-1}k)),
\end{gather*}
where $k\in K, g\in G$ and $f\in\psc{\sigma}{\mu}$. In the following, we will simply write $\ps{\sigma}{\mu}$ for both realizations for the sake of simplicity. If $\sigma$ is the trivial representation we write $(\pi_\mu,\pst{\mu})$ and refer to these representations as the \emph{spherical principal series}. The representation spaces in the spherical case naturally factor through the quotient $G/M$, respectively $K/M$, in the induced or compact picture and we will use these realizations in that case.
\subsection{Globalizations and infinitesimal character}
Let $(\pi,\mc H)$ denote a Hilbert space realization of a (subrepresentation of a) principal series representation. In this paragraph we define smooth and distribution vectors in $(\pi,\mc H)$. We call a vector $v\in\mc H$ a \emph{smooth} or $C^\infty$-vector for $\pi$ if 
\begin{gather*}
G\to\mc H,\qquad g\mapsto\pi(g)v
\end{gather*}
is smooth. Let $\mc H^\infty\subseteq\mc H$ denote the vector space of all smooth vectors in $\mc H$. For $\pi=\pi_{\sigma,\mu}^{\mathrm{cpt}}$ the smooth vectors are actually smooth functions (see e.g.\@ \cite[Equation (5.15)(a)]{VoganUnitaryRep}):
\begin{gather*}
\mc H^\infty=\{f:K\to V_\sigma\ \text{smooth}\mid\forall k\in K,\ m\in M \colon f(km)=\sigma(m^{-1})f(k)\}.
\end{gather*}
The \emph{distributional vectors} $\mc H^{-\infty}$ are given by the elements of the dual representation of the smooth vectors in the dual representation of $(\pi,\mc H)$. We give an alternative description which often is more convenient. Let $\tilde\sigma$ denote the dual representation of $\sigma$. Then, using \cite[ch.\@ I, §5.3.\@, Eq.\@ (25)]{GAGA}, we see that
\begin{gather*}
\langle\cdot,\cdot\rangle_{\sigma,\mu}\colon\ps{\sigma}{\mu}\times\ps{\tilde\sigma}{-\mu}\to\C,\quad \langle f_1,f_2\rangle_{\sigma,\mu}\coloneqq\int_Kf_2(k)(f_1(k))\intd k
\end{gather*}
is a nondegenerate, bilinear, and $G$-invariant pairing between $\ps{\sigma}{\mu}$ and $\ps{\tilde\sigma}{-\mu}$ for each $\sigma\in\hat M$ and $\mu\in\mf a^*$. By this pairing, we see that the distributional vectors $\psd{\sigma}{\mu}$ of the principal series representation $\ps{\sigma}{\mu}$ are given by the contragredient representation of $\pss{\tilde\sigma}{-\mu}$. If $\sigma$ is trivial, the distributional vectors can be realized on $\mc D'(K/M)\coloneqq C^\infty(K/M)'$, the space of distributions on $K/M$. We also define the space of distributions $\mc D'(G/M)\coloneqq C_c^\infty(G/M)'$ on $G/M$.

In rank one we have a unique $H\in\mf a_0\label{def:H}$ such that $\alpha(H)=1$ for the unique simple positive restricted root $\alpha$ of $(\mf g,\mf a)$. In this case, the distributional vectors in the induced picture of $\pst{\mu}$ are given by
\begin{gather*}\label{eq:distr_elem_ps}
\mc R(\mu-\rhoa)\coloneqq\{u\in\mc D'(G/M)\mid (H-\mu(H)+\rhoa(H))u=0,\forall\, U\in C^\infty(G\times_M\mf n_0)\colon Uu=0\},
\end{gather*}
equipped with the left regular representation, and there is a topological isomorphism
\begin{gather}\label{eq:end_point}
\mc Q_\mu\colon(\mc D'(K/M),\pi_{\mu}^{\mathrm{cpt}})\to\mc R(\mu-\rhoa)
\end{gather}
which intertwines the $G$-actions and extends \eqref{eq:cpt_induced} (cf.\@ \cite[Prop.\@ 3.4]{GHWb}).

Note that we always have the linear embedding 
\begin{gather*}
\iota_{\sigma,\mu}\colon\ps{\sigma}{\mu}\hookrightarrow\psd{\sigma}{\mu},\quad\iota_{\sigma,\mu}(f_1)(f_2)\coloneqq\langle f_1,f_2\rangle_{\sigma,\mu}.
\end{gather*}
For each subrepresentation $V\leq\ps{\sigma}{\mu}$ we have the restricted pairing
\begin{gather*}
V\times(\ps{\tilde\sigma}{-\mu}/V^{\perp_{\sigma,\mu}})\to\C,\quad \langle f_1,f_2+V^{\perp_{\sigma,\mu}}\rangle_{\sigma,\mu}\coloneqq\int_K f_2(k)(f_1(k))\intd k,
\end{gather*}
where
\begin{gather}\label{eq:perp_ps}
V^{\perp_{\sigma,\mu}}\coloneqq\{f_2\in\ps{\tilde\sigma}{-\mu}\mid\forall f_1\in V\colon\langle f_1,f_2\rangle_{\sigma,\mu}=0\}.
\end{gather}
This implies that $V^{-\infty}$ is the contragredient representation of $(\ps{\tilde\sigma}{-\mu}/V^{\perp_{\sigma,\mu}})^\infty$.

Any principal series representation has an infinitesimal character. In order to describe the infinitesimal character of $\ps{\sigma}{\mu}$ we first fix some notation. Let $\mf t\leq\mf m$ denote a $\theta$-stable Cartan subalgebra of $\mf m$, $\lambda_\sigma$ be the highest weight of $\sigma$ with respect to some ordering in $\mf t^*$ and $\rho_{\mf m}$ denote the half-sum of positive roots for $(\mf m,\mf t)$. Then $\ps{\sigma}{\mu}$ has infinitesimal character $\lambda_\sigma+\rho_{\mf m}-\mu$ relative to $\mf h\coloneqq\mf a\oplus\mf t$ (cf. \cite[Prop.\@ 8.22]{Kna86}). We recall the \emph{Casimir element} $\Omega_{\mf g}$, an important element of the center $\mc Z(\mc U(\mf g))$ of $\mc U(\mf g)$. Let $B$ be a fixed multiple of the Killing form $\kappa$. For a basis $X_1,\ldots,X_{\dim\mf g_0}$ of $\mf g_0$ let $(g^{ij})_{ij}$ denote the inverse matrix of $(B(X_i,X_j))_{i,j}$. Then the dual basis $(X^i)_i$ is given by $X^i=\sum g^{ij}X_j$ and the Casimir element is defined by
\begin{gather*}
\Omega_{\mf g}\coloneqq\sum\nolimits_{i}X^iX_i=\sum\nolimits_{i,j}g^{ij}X_jX_i\in\mc Z(\mc U(\mf g)).
\end{gather*}
Since $B$ is nondegenerate, there are unique elements $X_\varphi\in\mf g_0$ for each $\varphi\in\mf g_0^*$ such that $\varphi(X)=B(X,X_\varphi)$ for each $X\in\mf g_0$. We put $\langle\varphi,\psi\rangle\coloneqq B(X_\varphi,X_\psi)$ for $\varphi,\psi\in\mf g_0^*$ resp.\@ $\mf g^*$. Let us extend the ordering on $\mf a$ to $\mf h$ such that $\Sigma^+$ arises by restriction from the positive roots of $(\mf g,\mf h)$. By \cite[La.\@ 12.28]{Kna86}, the action of the Casimir element is then given by the scalar
\begin{align*}
\pi_{\sigma,\mu}(\Omega_{\mf g})&=\langle\lambda_\sigma+\rho_{\mf m},\lambda_\sigma+\rho_{\mf m}\rangle+\langle\mu,\mu\rangle-\langle\rhoa+\rho_{\mf m},\rhoa+\rho_{\mf m}\rangle\\
&=\langle\lambda_\sigma,\lambda_\sigma+2\rho_{\mf m}\rangle+\langle\mu,\mu\rangle-\langle\rhoa,\rhoa\rangle.
\end{align*}
\subsection{Reducibility}\label{subsec:red}
We are particularly interested in principal series representations which are not irreducible, i.e.\@ in the set
\begin{gather*}
\mc A'\coloneqq\{(\sigma,\mu)\in\hat M\times\mf a^*\mid\ps{\sigma}{\mu}\text{ reducible}\}.
\end{gather*}
In this subsection we introduce the representation theoretic tools we need to describe the structure of these reducible representations.
\subsubsection*{Composition series, minimal \texorpdfstring{$K$}{K}-types and socle}
In general, principal series representations are not completely reducible. However, they are all of \emph{finite length} (cf.\@ \cite{Kra78}). This means, there exists a finite \emph{composition series}, i.e.\@ a chain of subrepresentations of $\ps{\sigma}{\mu}$ of the form
\begin{gather*}
0\subsetneq V_1\subsetneq\ldots\subsetneq V_n=\ps{\sigma}{\mu}
\end{gather*}
such that the quotients $V_{i+1}/V_i$, the \emph{composition factors}, are irreducible. By the Jordan-Hölder theorem, any two composition series have the same length and the same composition factors up to permutation and isomorphism.

Let $\pi$ denote an admissible Hilbert representation of $G$ (i.e.\@ a continuous representation such that each $K$-isotypic component has finite dimension) and fix a Cartan subalgebra $\mf b_0$ of $\mf k_0$. With respect to some ordering, we define $\rho_{\mf k}$ as the half-sum of the positive roots of $(\mf k,\mf b)$. We say that $Y\in\hat K$ with highest weight $\lambda$ is a \emph{minimal $K$-type} of $\pi$ if $Y$ occurs in $\pi$ restricted to $K$ and
\begin{gather*}
\langle\lambda+2\rho_{\mf k},\lambda+2\rho_{\mf k}\rangle
\end{gather*}
is minimal with respect to this property. The set of minimal $K$-types is independent of the choice of the ordering and its cardinality is finite and at least one. For principal series representations $\pi_{\sigma,\mu}$ each minimal $K$-type of $\pi_{\sigma,\mu}$ occurs in $\restr{\pi_{\sigma,\mu}}{K}$ with multiplicity one (cf.\@ \cite[Thm.\@ 1.1]{Vo79}).

For any Hilbert representation $(\pi,\mc H)$ of $G$ we define $\on{soc}\pi$, the \emph{socle of $\pi$}, as the closure (in the sense of \cite[Thm.\@ 8.9]{Kna86}) of the sum of all completely reducible $(\mf g,K)$-submodules of the underlying $(\mf g,K)$-module of $(\pi,\mc H)$ (see \cite[p.\@ 538]{cohomo}).

\subsubsection*{Decomposition as \texorpdfstring{$K$}{K}-representation and \texorpdfstring{$M$}{M}-spherical functions} We begin with a brief discussion of the decomposition of $\restr{\pi_{\sigma,\mu}}{K}$ in general and then give some more precise results of this decomposition in the rank one case. For the decomposition as $K$-representation we consider the compact picture $\psc{\sigma}{\mu}$. As $K$-representation this coincides with the induced representation $\on{Ind}_M^K\sigma$ of $\sigma$ to $K$. By Frobenius reciprocity we thus obtain for each $Y\in\hat K$ that
\begin{gather*}
\Hom{K}{\psc{\sigma}{\mu}}{Y}=\Hom{K}{\on{Ind}_M^K\sigma}{Y}\cong\Hom{M}{V_\sigma}{Y}.
\end{gather*}
Let us denote the multiplicity of $V_\sigma$ in $Y$ (and analogously for other groups and representations) by 
\begin{gather*}
\on{mult}_M(V_\sigma,Y)\coloneqq\dim_\C\Hom{M}{V_\sigma}{Y}.
\end{gather*}
Then, writing
\begin{gather*}
\hat K_\sigma\coloneqq\{Y\in\hat K\colon\on{mult}_M(V_\sigma,Y)\neq0\},
\end{gather*}
we have that, denoting equivalence as $K$-representations by $\cong_K$ and the Hilbert space direct sum by $\widehat\bigoplus$,
\begin{gather*}
\psc{\sigma}{\mu}\cong_K\widehat\bigoplus_{Y\in\hat K_\sigma}\on{mult}_M(V_\sigma,Y)Y.
\end{gather*}
In the spherical case we will abbreviate $\hat K_M\coloneqq\hat K_{\on{triv}_M}$. If not stated otherwise, we will always realize $Y\in\hat K_M$ as a subrepresentation of $\psc{\sigma}{\mu}=L^2(K/M)$. Note that $L^2(K/M)$ carries the left regular representation $L$. We denote the derived representation of $L$ by $\ell$.

Let us now assume that $G$ has real rank one and consider the spherical principal series. In this case some more precise results can be achieved. Most importantly, $(K,M)$ is a Gelfand pair in this case (cf.\@ \cite[ch.\@ II, §6, Cor.\@ 6.8]{GASS}). This implies the following  

\begin{proposition}\label{prop:multone}
Let $\C$ denote the trivial $M$-representation. Then
\begin{gather}\label{eq:multone}
\forall\, Y\in\hat K_M\colon\on{mult}_K(Y,\pst{\mu})=\on{mult}_M(\C,Y)=\dim_\C Y^M=1,
\end{gather}
where 
\begin{gather*}
Y^M\coloneqq\{v\in Y\mid\forall m\in M\colon m.v=v\}
\end{gather*}
denotes the space of $M$-invariant elements in $Y$. In particular, the decomposition
\begin{gather*}
\pst{\mu}\cong_K\widehat\bigoplus_{Y\in\hat K_M}Y
\end{gather*}
is multiplicity free.  
\end{proposition}
\begin{proof}
 The first equality follows from Frobenius reciprocity and the second equality follows from \cite[ch.\@ V, Thm.\@ 3.5 (iv)]{GAGA}.
\end{proof}

Note that $\on{Ind}_M^K(\on{triv}_M)$ is given by the left regular representation of $K$ on $L^2(K/M)$ resp.\@ $L^2(K)^M$, the $M$-invariant elements of $L^2(K)$ with respect to the right regular representation. The following proposition describes the $M$-spherical elements $Y^M$ for each $Y\in\hat K_M$ and is well-known to specialists. We give a proof for the convenience of the reader.
\begin{proposition}[cf. {\cite[Introduction, Proposition 3.2 (ii)]{GAGA}}]\label{prop:helg_general}
Let $0\neq(\tau,Y)\leq L^2(K)^M$ be an irreducible representation. Then
\begin{enumerate}
\item there exists a unique $\phi_Y\in Y^M$ such that $\phi_Y(e)=1$ and $Y^M=\C\phi_Y$,\label{prop:helg_general_i}
\item $\varphi(k)\langle\phi_Y,\phi_Y\rangle_{L^2(K)}=\langle\varphi,\tau(k)\phi_Y\rangle_{L^2(K)}$ for $k\in K,\ \varphi\in Y$\label{prop:helg_general_ii},
\item $\langle\phi_Y,\phi_Y\rangle_{L^2(K)}=\frac{1}{\on{dim}Y},\ \phi_Y(k^{-1})=\overline{\phi_Y(k)},\ \abs{\phi_Y(k)}\leq1$ for $k\in K$\label{prop:helg_general_iii}.
\end{enumerate}
\end{proposition}

\begin{proof}
\ref{prop:helg_general_i} By Equation \eqref{eq:multone} we have $\dim_\C Y^M=1$. Let $0\neq\psi\in Y$ and choose some $k\in K$ such that $\psi(k)\neq0$. Replacing $\psi$ by $\tau(k^{-1})\psi$ we may assume that $\psi(e)\neq0$. The function
\begin{gather*}
\Psi:K\to\C,\quad k\mapsto\int_M\tau(m)\psi(k)\intd m
\end{gather*}
is contained in $Y^M$ with $\Psi(e)=\psi(e)\neq0$. This proves the first part.\par
\ref{prop:helg_general_ii} For each $m\in M$ we have by the $K$-invariance of the Haar measure 
\begin{align*}
\langle\varphi,\phi_Y\rangle_{L^2(K)}&=\int_{K}\varphi(k)\ov{\phi_Y(k)}\intd k
=\int_{K}\varphi(k)\ov{\phi_Y(m^{-1}k)}\intd k\\
&=\int_{K}\varphi(mk)\ov{\phi_Y(k)}\intd k
=\int_{K}\ov{\phi_Y(k)}\int_M\varphi(mk)\intd m\intd k.
\end{align*}
Note that the map
\begin{gather*}
\theta:K\to\C,\quad k\mapsto\int_M\varphi(mk)\intd m=\int_M\tau(m^{-1})\varphi(k)\intd m
\end{gather*}
is contained in $V^M=\C\phi_Y$. We infer that $\theta=\theta(e)\phi_Y=\varphi(e)\phi_Y$ and thus
\begin{gather*}
\langle\varphi,\phi_Y\rangle_{L^2(K)}=\varphi(e)\int_{K}\ov{\phi_Y(k)}\phi_Y(k)\intd k=\varphi(e)\langle\phi_Y,\phi_Y\rangle_{L^2(K)}.
\end{gather*}
Replacing $\varphi$ by $\tau(k^{-1})\varphi$ we obtain \ref{prop:helg_general_ii}.\par
\ref{prop:helg_general_iii} By the Schur orthogonality relations we have
\begin{align*}
\frac{1}{\on{dim}Y}\langle\varphi,\varphi\rangle_{L^2(K)}\langle\phi_Y,\phi_Y\rangle_{L^2(K)}&=\int_{K}\langle\tau(k)\phi_Y,\varphi\rangle_{L^2(K)}\ov{\langle\tau(k)\phi_Y,\varphi\rangle_{L^2(K)}}\intd k\\
&\stackrel{\ref{prop:helg_general_ii}}{=}\int_{K}\ov{\varphi(k)\langle\phi_Y,\phi_Y\rangle_{L^2(K)}}\varphi(k)\langle\phi_Y,\phi_Y\rangle_{L^2(K)}\intd k\\
&=\langle\phi_Y,\phi_Y\rangle_{L^2(K)}^2\int_{K}\ov{\varphi(k)}\varphi(k)\intd k\\
&=\langle\phi_Y,\phi_Y\rangle_{L^2(K)}^2\langle\varphi,\varphi\rangle_{L^2(K)}.
\end{align*}
This proves $\langle\phi_Y,\phi_Y\rangle_{L^2(K)}=\frac{1}{\dim Y}$.
By \ref{prop:helg_general_ii} we deduce
\begin{gather*}
\phi_Y(k)=\dim Y\langle\phi_Y,\tau(k)\phi_Y\rangle_{L^2(K)}=\dim Y\ov{\langle\phi_Y,\tau(k^{-1})\phi_Y\rangle_{L^2(K)}}=\ov{\phi_Y(k^{-1})}
\end{gather*}
and, using the Cauchy-Schwarz inequality,
\begin{gather*}
\abs{\phi_Y(k)}=\dim Y\abs{\langle\phi_Y,\tau(k)\phi_Y\rangle_{L^2(K)}}\leq\dim Y\langle\phi_Y,\phi_Y\rangle_{L^2(K)}=1.\qedhere
\end{gather*}
\end{proof}

\subsubsection*{Intertwiner}
Finally, we describe a procedure to obtain $G$-equi\-vari\-ant maps between sections of associated vector bundles. These \emph{generalized gradients} generalize the classical raising and lowering operators of $\mathrm{PSL}(2,\R)$. 

The following fact allows the definition of generalized gradients.
\begin{proposition}[{cf. \cite[Proposition 3.1]{oersted}}]\label{prop:oersted_nabla}
    Let $K$ act on $\mf p^*$ by the coadjoint representation. The following map is defined for every $(\tau, Y)\in\hat K$:
    \begin{gather*}
        \nabla: C^\infty(G\times_KY)\to C^\infty(G\times_K(Y\otimes\mf p^*)),\\
        (\nabla f)(g)\in\on{Hom}(\mf p, Y)\cong Y\otimes\mf p^*,\ (\nabla f)(g)(X)\coloneqq\d f(g\exp tX).
    \end{gather*}
    Moreover, it defines a $G$-equivariant covariant derivative with zero torsion. 
\end{proposition}
\begin{definition}\label{def:gen_grad}
    Let $(\tau_i, Y_{\tau_i})\in\hat K,\ i\in\{1,2\},$ with $Y_{\tau_2}\leq Y_{\tau_1}\otimes\mf p^*$ and let $T\in\Hom{K}{Y_{\tau_1}\otimes\mf p^*}{Y_{\tau_2}}$.
 Then we define the \emph{generalized gradient}
    \begin{gather*}
        T\circ\;\nabla: C^\infty(G\times_KY_{\tau_1})\to C^\infty(G\times_KY_{\tau_2}).
    \end{gather*}
    If not stated otherwise, we choose $T=\on{pr}_{\tau_2}$, the orthogonal projection onto $Y_{\tau_2}$.
\end{definition}
\section{Poisson transforms}\label{sec:PT}
In this section we connect principal series representations with joint eigenspaces of differential operators on vector bundles over $G/K$.
 We will first recall the standard scalar Poisson transform and show how it is related to the exceptional parameters of \cite{DFG} and \cite{GHWb}.
 Then we introduce vector valued generalizations based on \cite{Ol}, discuss some mapping properties and relate them to specific generalized gradients.
\subsection{Invariant differential operators and eigenspaces}
Let $(\tau,Y)\in\hat K$. A differential operator $D$ on $C^\infty(G\times_KY)$ is called \emph{invariant} if it commutes with the left regular representation $L$ on $C^\infty(G\times_KY)$.
Let $\mathbb{D}(G,\tau)$ denote the algebra of all invariant differential operators on $C^\infty(G\times_KY)$ and abbreviate $\mathbb{D}(G/K)\coloneqq\mathbb{D}(G,\on{triv})$. Then $\mathbb{D}(G,\tau)$ is isomorphic to $\mc U(\mf g)^K/(\mc U(\mf g)I_{\tilde\tau})^K$ via the right regular representation $r$, where $I_{\tilde\tau}\coloneqq\ker\tilde\tau\subset\mf k$ denotes the kernel of $\tilde\tau$, the dual representation of $\tau$ (see \cite[Satz 2.4]{Ol}).

For the trivial bundle the \emph{Harish-Chandra homomorphism} $\chi:\mathbb{D}(G/K)\to S(\mf a_0)^W$ allows us to identify $\mathbb{D}(G/K)$ with the $W$-invariants $S(\mf a_0)^W$ of the symmetric algebra $S(\mf a_0)$ of $\mf a_0$ (see \cite[ch.\@ II, Thm.\@ 4.3, 5.18]{GAGA}). Moreover, every character of $\mathbb{D}(G/K)$ is of the form 
\begin{gather*}
\chi_\mu\colon\mathbb{D}(G/K)\to\C,\quad\chi_\mu(D)\coloneqq\chi(D)(\mu)
\end{gather*}
for some $\mu\in\mf a^*$ and $\chi_\nu=\chi_\mu$ if and only if $\nu\in W\mu$ (cf.\@ \cite[ch.\@ III, Lemma 3.11]{GAGA}). Let us denote the space of joint eigenfunctions of $\mathbb{D}(G/K)$ by
\begin{gather*}
\mc E_\mu\coloneqq\{f\in C^\infty(G/K)\mid \forall D\in\mathbb{D}(G/K)\colon Df=\chi_\mu(D)f\},
\end{gather*}
and, with the Riemannian distance function $d_{G/K}$ on $G/K$, for each $r\geq0$
\begin{gather}\label{eq:E_distr}
\mc E_{\mu,r}(G/K)\coloneqq\{f\in\mc E_\mu\mid\sup_{g\in G}\abs{e^{-rd_{G/K}(eK,gK)}f(g)}<\infty\}.
\end{gather}
We put $\mc E_{\mu,\infty}(G/K)\coloneqq\bigcup_{r\geq0}\mc E_{\mu,r}(G/K)$, equipped with the direct limit topology. 

For arbitrary $(\tau,Y)\in\hat K$ we define a representation $\chi_{\sigma,\mu}$ of $\mathbb{D}(G,\tau)$ for each $\mu\in\mf a^*$ and $(\sigma,V)\in\hat M$ with $\on{mult}_M(V,Y)\neq0$ by
\begin{gather*}
\chi_{\sigma,\mu}:\mathbb{D}(G,\tau)\to\on{End}(\Hom{K}{\ps{\sigma}{\mu}}{Y}),\quad\chi_{\sigma,\mu}(r(u))(T)\coloneqq T\circ\pi_{\sigma,\mu}(\on{opp}u),
\end{gather*}
where $u\in\mc U(\mf g)^K$ and $\on{opp}:\mc U(\mf g)\to\mc U(\mf g)$ is defined by $\on{opp}(X)\coloneqq -X$ for $X\in\mf g$ (see \cite[Def.\@ 2.10]{Ol}). If $\on{mult}_M(V,Y)=1$ these representations are one dimensional and we can define the space of joint eigensections
\begin{gather*}
E_{\sigma,\mu}\coloneqq\{f\in C^\infty(G\times_KY)\mid\forall D\in\mathbb{D}(G,\tau)\colon Df=\chi_{\sigma,\mu}(D)f\},
\end{gather*}
where we identified $\on{End}(\Hom{K}{\ps{\sigma}{\mu}}{Y})$ with $\C$. Each $E_{\sigma,\mu}$ has an infinitesimal character and it coincides with that of $\ps{\sigma}{\mu}$ (see \cite[Folg.\@ 2.15]{Ol}).

\subsection{Mapping properties of scalar Poisson transforms}
The asymptotics of joint eigenfunctions in $\mc E_\mu$ can be described by a specific meromorphic function on $\mf a^*$, the \emph{Harish-Chandra $c$-function} $\mathbf{c}(\mu)$. We define its ``denominator'', the meromorphic function $\mathbf{e}(\mu)^{-1}$, by ($\mu\in\mf a^*$)
\begin{gather*}
\mathbf{e}(\mu)^{-1}\coloneqq\prod_{\alpha\in\Sigma^+}\Gamma\left(\frac{1}{2}\left(\frac{1}{2}m_\alpha+1+\frac{\langle\mu,\alpha\rangle}{\langle\alpha,\alpha\rangle}\right)\right)\Gamma\left(\frac{1}{2}\left(\frac{1}{2}m_\alpha+m_{2\alpha}+\frac{\langle\mu,\alpha\rangle}{\langle\alpha,\alpha\rangle}\right)\right),
\end{gather*}
see e.g.\@ \cite[Eq.\@ (5.17)]{Sc84}. Then $\mathbf{e}$ is an entire function on $\mf a^*$ without zeros on the closure of the positive Weyl chamber. 
\begin{definition}[{cf. \cite[Thm.\@ 10.6, 12.2]{vdBS87}}]\label{def:sc_PT}
For $\mu\in\mf a^*$ define the \emph{scalar Poisson transform}
\begin{gather*}
P_\mu:\mc D'(K/M)\to\mc E_{\mu,\infty}(G/K),\quad P_\mu(T)(gK)\coloneqq T(kM\mapsto a_I(g^{-1}k)^{-(\mu+\rhoa)}).
\end{gather*}
Then $P_\mu$ is a topological isomorphism if and only if $\mathbf{e}(\mu)\neq 0$. If $\mathbf{e}(\mu)=0$, then $P_\mu$ is neither injective nor surjective.
\end{definition}

\begin{definition}
We call
\begin{gather*}
\mathbf{Ex}\coloneqq\{\mu\in\mf a^*\mid \mathbf{e}(\mu)=0\}
\end{gather*}
the set of \emph{exceptional parameters}.
\end{definition}

\begin{example}\label{ex:Ex_rank_one}
The exceptional parameters are exactly the parameters which were excluded in \cite{DFG} and \cite{GHWb}. Indeed, let $G$ be of real rank one. Then the $\mathbf{e}$-function is zero if and only if one of the Gamma functions has a pole which is the case if and only if
\begin{gather*}
\mu\in\left(-\frac{1}{2}m_\alpha-1-2\N_0\right)\alpha\cup\left(-\frac{1}{2}m_\alpha-m_{2\alpha}-2\N_0\right)\alpha,
\end{gather*}
where $\alpha$ denotes the unique simple positive real root. Moreover, (see \cite[ch.\@ IV, Thm.\@ 1.1]{Hel70})
\begin{gather*}
\pst{\mu}\text{ irreducible}\quad\Leftrightarrow\quad\mathbf{e}(\mu)\mathbf{e}(-\mu)\neq0.
\end{gather*}
Therefore, irreducibility of $\pst{\mu}$ is sufficient but not necessary for the bijectivity of $P_\mu$.
\end{example}

\subsection{Vector valued Poisson transforms}
In this subsection we describe generalized Poisson transforms based on \cite{Ol}, which will serve as a substitute for the scalar Poisson transform for the exceptional parameters.

\begin{definition}[cf. {\cite[Definition 3.2/Satz 3.4]{Ol}}]\label{def:PoissonTransformation}
  Let $\tau\in\hat K,\ \sigma\in\hat M$ and $\mu\in\mf a^*$. Then we define the \emph{(vector valued) Poisson transform} by
  \begin{gather}
    P_{\sigma, \mu}^\tau: \Hom{K}{\ps{\sigma}{\mu}}{V_\tau}\otimes\psd{\sigma}{\mu}\to C^\infty(G\times_K V_\tau),\quad P_{\sigma,\mu}^\tau(T\otimes f)(g)=T(\pi_{\sigma,\mu}(g^{-1})f).\label{eq:olbrich_Poisson_G}
  \end{gather}
If $F:\Hom{K}{\ps{\sigma}{\mu}}{V_\tau}\cong\Hom{M}{V_\sigma}{V_\tau}$ denotes the Frobenius isomorphism we have
\begin{align*}
P_{\sigma,\mu}^\tau(T\otimes f)(g)&=\int_K\tau(k)F(T)(f(gk))\intd k\\
&=\int_K a_{I}(g^{-1}k)^{-(\mu+\rhoa)}\tau(k_{I}(g^{-1}k))F(T)(f(k))\intd k
\end{align*}
for $T\in\Hom{K}{\ps{\sigma}{\mu}}{V_\tau},\, f\in\ps{\sigma}{\mu}$ and $g\in G$. The image of $P_{\sigma, \mu}^\tau$ is contained in $E_{\sigma,\mu}$ and $P_{\sigma,\mu}^\tau$ is $\mathbb{D}(G,\tau)\times G$-equivariant, where $\mathbb{D}(G,\tau)$ acts on $\Hom{K}{\ps{\sigma}{\mu}}{V_\tau}$ by $\chi_{\sigma,\mu}$. We abbreviate $P_{\sigma,\mu}^\tau$ by $P_\mu^\tau$ if $\sigma$ is the trivial representation of $M$.
\end{definition}

\begin{remark}[Scalar vs.\@ vector valued]
If $\tau=\on{triv}_K$ is the trivial $K$-representation we have $\Hom{K}{\pst{\mu}}{V_\tau}\cong\Hom{M}{\C}{\C}\cong\C$. Let $t\in\Hom{M}{\C}{\C}$ be the identity and $T\coloneqq F^{-1}(t)$. Then
\begin{gather*}
P_{\mu}^\tau(T\otimes f)(g)=\int_K a_{I}(g^{-1}k)^{-(\mu+\rhoa)}f(k)\intd k=P_\mu(f)(gK).
\end{gather*}
\end{remark}
The following lemma illustrates the naturality of Olbrich's Poisson transforms.
\begin{lemma}[cf. {\cite[Remark after Lemma 3.3]{Ol}}]\label{la:Poisson_natural}
    Let $\Psi: \ps{\sigma}{\mu}\to C^\infty(G\times_KV_\tau)$ be a $G$-equivariant map. Then 
    \begin{gather*}
        \Psi = P_{\sigma,\mu}^\tau(T\otimes\bigcdot)
    \end{gather*}
    where $T\in\Hom{K}{\ps{\sigma}{\mu}}{V_\tau}$ is defined by $T(f)\coloneqq\Psi(f)(e)$.
\end{lemma}
\begin{proof}
  For every $k\in K$ we have 
  \begin{gather*}
    T(\pi_{\sigma,\mu}(k)f)=\Psi(\pi_{\sigma,\mu}(k)f)(e)=\Psi(f)(k^{-1})=\tau(k)\Psi(f)(e)=\tau(k)T(f)
  \end{gather*}
  and thus $T\in\Hom{K}{\ps{\sigma}{\mu}}{V_\tau}$. Moreover we have for every $g\in G$ and $f\in \ps{\sigma}{\mu}$
  \begin{gather*}
    P_{\sigma,\mu}^\tau(T\otimes f)(g)=\Psi(\pi_{\sigma,\mu}(g^{-1})f)(e)=\Psi(f)(g).\qedhere
  \end{gather*}
\end{proof}
This lemma admits the following important implications.

\begin{corollary}\label{cor:no_poisson_transforms}
    Let $\Psi: \ps{\sigma}{\mu}\to C^\infty(G\times_KV_\tau)$ be a $G$-equivariant map where $V_\tau$ does not contain the $M$-representation $V_\sigma$. Then $\Psi=0$.
\end{corollary}

\begin{proof}
 By Lemma \ref{la:Poisson_natural} there exists $T\in\Hom{K}{\ps{\sigma}{\mu}}{V_\tau}$ such that $\Psi = P_{\sigma,\mu}^\tau(T\otimes\bigcdot)$. But $\Hom{K}{\ps{\sigma}{\mu}}{V_\tau}\cong\Hom{M}{V_\sigma}{V_\tau}=0$ by Frobenius reciprocity.
\end{proof}

\begin{corollary}\label{cor:constant_Poisson_transforms}
    Let $(\tau_i,V_{\tau_i})\in\hat K,\ i\in\{1,2\},$ be such that 
    \begin{gather*}
        \on{mult}_K(V_{\tau_i}, \ps{\sigma}{\mu})=\dim_\C\Hom{K}{\ps{\sigma}{\mu}}{V_{\tau_i}}=1
    \end{gather*}
    and let $\Phi:C^\infty(G\times_KV_{\tau_1})\to C^\infty(G\times_KV_{\tau_2})$ be a $G$-equivariant map. By choosing $0\neq T_i\in\Hom{K}{\ps{\sigma}{\mu}}{V_{\tau_i}}$ we consider the Poisson transforms $P_{\sigma,\mu}^{\tau_i}$ as maps from $\ps{\sigma}{\mu}$ to $C^\infty(G\times_KV_{\tau_i})$. Then there exists some $c\in\C$ such that 
    \begin{gather*}
        \Phi\circ P_{\sigma,\mu}^{\tau_1}=c\cdot P_{\sigma,\mu}^{\tau_2}.
    \end{gather*}
\end{corollary}
\begin{proof}
    Since 
    \begin{gather*}
        \Phi\circ P_{\sigma,\mu}^{\tau_1}(T_1\otimes\bigcdot):\ps{\sigma}{\mu}\to C^\infty(G\times_KV_{\tau_2})
    \end{gather*}
    is a $G$-equivariant map there exists some $T\in\Hom{K}{\ps{\sigma}{\mu}}{V_{\tau_2}}$ such that 
    \begin{gather*}
        \Phi\circ P_{\sigma,\mu}^{\tau_1}(T_1\otimes\bigcdot)=P_{\sigma,\mu}^{\tau_2}(T\otimes\bigcdot)
    \end{gather*}
     by Lemma \ref{la:Poisson_natural}. Since $\dim_\C\Hom{K}{\ps{\sigma}{\mu}}{V_{\tau_i}}=1$ there exists some $c\in\C$ with $T=c\cdot T_2$.
\end{proof}

\subsection{Injectivity of vector valued Poisson transforms}\label{subsec:inj_PT}
In this subsection we investigate specific vector valued Poisson transforms. We will see that if we pick a minimal $K$-type for each irreducible subspace of the representation, the direct sum of the associated Poisson transforms is injective. By our rank one assumption each spherical principal series representation $\pst{\mu}$ decomposes multiplicity-freely as a $K$-representation (see Proposition \ref{prop:multone}). Therefore we have the following

\begin{lemma}\label{la:orth_proj_scalar}
Let $0\neq(\tau,V)\leq\pst{\mu}$ be an irreducible $K$-representation and $t\in\Hom{M}{\C}{V}$. Then
\begin{gather*}
P_\mu^\tau(F^{-1}(t)\otimes\bigcdot)=\frac{t(1)(e)}{\dim V}P_\mu^\tau(\on{pr}_V\otimes\;\bigcdot),
\end{gather*}
where as above $F$ denotes the Frobenius isomorphism.
\end{lemma}

\begin{proof}
By Equation \eqref{eq:olbrich_Poisson_G} we have for each $f\in\pstd{\mu}$ and $g\in G$ that
\begin{gather*}
P_\mu^\tau(F^{-1}(t)\otimes f)(g)=F^{-1}(t)(\pi_\mu(g)^{-1}f)\text{ and }
P_\mu^\tau(\on{pr}_V\otimes\; f)(g)=\on{pr}_V(\pi_\mu(g)^{-1}f).
\end{gather*}
By Proposition \ref{prop:multone} we obtain that $\Hom{K}{\pst{\mu}}{V}=\C\on{pr}_V$ is one-dimensional since $(\tau,V)\in\hat K_M$. This proves that there exists some $c\in\C$ such that
\begin{gather*}
P_\mu^\tau(F^{-1}(t)\otimes\bigcdot)=cP_\mu^\tau(\on{pr}_V\otimes\;\bigcdot).
\end{gather*}
Recall the $M$-invariant function $\phi_V$ from Proposition \ref{prop:helg_general}. By Definition \ref{def:PoissonTransformation} we have
\begin{align*}
P_\mu^\tau(F^{-1}(t)\otimes\phi_V)(e)&=\int_K\tau(k)t(\phi_V(k))\intd k=\int_K\phi_V(k)\tau(k)t(1)\intd k\\
&=\int_K\phi_V(k)\tau(k)t(1)(e)\phi_V\intd k,
\end{align*}
where we recall that $V$ is realized in $L^2(K)$ so that $t(1)(e)$ makes sense, and we used Proposition \ref{prop:helg_general}\,\ref{prop:helg_general_i} for the last equality. Using Proposition \ref{prop:helg_general}\,\ref{prop:helg_general_iii} we infer
\begin{align*}
P_\mu^\tau(F^{-1}(t)\otimes \phi_V)(e)(e)&=\int_K\phi_V(k)t(1)(e)\phi_V(k^{-1})\intd k=\int_K\phi_V(k)t(1)(e)\ov{\phi_V(k)}\intd k\\
&=t(1)(e)\langle\phi_V,\phi_V\rangle_{L^2(K)}=\frac{t(1)(e)}{\dim V}.
\end{align*}
On the other hand \eqref{eq:olbrich_Poisson_G} yields
\begin{gather*}
P_\mu^\tau(\on{pr}_V\otimes\;\phi_V)(e)=\on{pr}_V(\phi_V)=\phi_V.
\end{gather*}
Thus,
\begin{gather*}
c=\frac{P_\mu^\tau(F^{-1}(t)\otimes \phi_V)(e)(e)}{P_\mu^\tau(\on{pr}_V\otimes\;\phi_V)(e)(e)}=\frac{t(1)(e)}{\phi_V(e)\dim V}=\frac{t(1)(e)}{\dim V}.\qedhere
\end{gather*}
\end{proof}
From now on we choose $t\in\Hom{M}{\C}{V}$ for each $(\tau,V)\in\hat K_M$ by $t(1)\coloneqq\phi_V$ and define
\begin{gather*}
P_{\mu}^\tau: \pstd{\mu}\to C^\infty(G\times_K V),\qquad P_{\mu}^\tau(f)\coloneqq P_{\mu}^\tau(F^{-1}(t)\otimes f).
\end{gather*}
Note that, by Lemma \ref{la:orth_proj_scalar}, we have for each $f\in\pstd{\mu}$ and $g\in G$
\begin{gather}\label{eq:Poisson_proj}
P_{\mu}^\tau(f)(g)=\frac{1}{\dim V}\on{pr}_V(\pi_\mu(g)^{-1}f).
\end{gather}

\begin{proposition}\label{prop:PoissonInjective}
    Let $[(\tau, V_\tau)]\in\hat K_M$ and $\mu\in\mf a^*$. Then the Poisson transform 
    \begin{align*}
        P_{\mu}^\tau: \pstd{\mu}\to C^\infty(G\times_K V)
    \end{align*}
     is injective if and only if every non-trivial $G$-invariant subspace of $\pstd{\mu}$ contains $\tau$. Moreover, the kernel is given by the distributional elements in the closure of the sum of all $G$-invariant subspaces $V\leq\pst{\mu}$ with $\on{mult}_K(\tau, V)=0$.
\end{proposition}

\begin{proof}
Since $P_{\mu}^\tau$ is $G$-equivariant, the kernel $\ker P_{\mu}^\tau$ is $G$-invariant. We claim that it equals the closure of the sum of all invariant subspaces of $\pst{\mu}$ which do not contain the $K$-representation $(\tau,V_\tau)$:\\
   If $\{0\}\neq W\leq\pst{\mu}$ is an invariant subspace of $\pst{\mu}$ which does not contain the $K$-representation $\tau$, by \eqref{eq:Poisson_proj} we have
   \begin{gather*}
        P_{\mu}^\tau(f)(g)=\frac{1}{\dim V_\tau}\on{pr}_{V_\tau}(\pi_\mu(g^{-1})f)=0
   \end{gather*}
   for every $f\in W$ and $g\in G$ since $\pi_\mu(g^{-1})f\in W$. Thus, $f\in\ker P_{\mu}^\tau$. This proves the first inclusion because the kernel is closed.\\
   Conversely, let $f\in\ker P_{\mu}^\tau$. Since the kernel is invariant, the distributional elements in the $G$-cyclic space $W_f$ of $f$ are also contained in the kernel of $P_{\mu}^\tau$. Therefore, $f$ is contained in an invariant space which does not contain $\tau$ (if $W_f$ contains $\tau$ we can choose $g=e$ to get a contradiction to $W_f\subseteq\ker P_{\mu}^\tau$).
\end{proof}

\begin{proposition}\label{prop:Poisson_sum}
 Let $\mu\in\mf a^*$ and $\mathbf{Irr}(\mu)$ be the set of all non-zero irreducible subrepresentations of $\pst{\mu}$. Then, if $(\tau_U, V_{\tau_U})$ is any non-zero $K$-type of $U$ for $U\in\mathbf{Irr}(\mu)$, the direct sum of the corresponding Poisson transforms
  \begin{gather*}
    \oplus_{U\in\mathbf{Irr}(\mu)} P_\mu^{\tau_U}: \pstd{\mu}\to \bigoplus_{U\in\mathbf{Irr}(\mu)}C^\infty(G\times_KV_{\tau_U}) 
  \end{gather*}
  is injective. A natural choice of $(\tau_U,V_{\tau_U})$ is given by a minimal $K$-type of $U$.
\end{proposition}
\begin{proof}
    Since the kernel of the direct sum $\oplus_{U\in\mathbf{Irr}(\mu)} P_\mu^{\tau_U}$ is the intersection of the kernels of  $P_\mu^{\tau_U},\ U\in\mathbf{Irr}(\mu),$ we can apply Proposition \ref{prop:PoissonInjective} to deduce
    \begin{gather*}
        \oplus_{U\in\mathbf{Irr}(\mu)} P_\mu^{\tau_U}\text{ injective }\Leftrightarrow\forall\, \{0\}\neq V\leq \pst{\mu}\ \exists\, U\in\mathbf{Irr}(\mu)\colon\on{mult}_K(\tau_U,V)\neq 0.
    \end{gather*}
    Let $\{0\}\neq V\leq \pst{\mu}$ be a non-trivial (closed) $G$-invariant subspace. We claim that there exists some $U\in\mathbf{Irr}(\mu)$ such that $\on{mult}_K(\tau_U, V)\neq0$. In fact, since $\pst{\mu}$ has a composition series, $V$ also has a composition series by \cite[p. 815]{cohomo}. In particular, there exists an irreducible subrepresentation $\{0\}\neq I\leq V$. But $I\in\mathbf{Irr}(\mu)$ by the definition of $\mathbf{Irr}(\mu)$ and $\on{mult}_K(\tau_I,V)\neq 0$ since $I\leq V$. 
\end{proof}

\subsection{The role of generalized gradients}
In this subsection we use generalized gradients to connect different Poisson transforms associated with inequivalent $K$-representations. We first introduce some notation.

\begin{notation}\label{not:inn_prod_dual_basis}
We define the inner product
\begin{gather*}
\langle\cdot,\cdot\rangle\coloneqq-\frac{\kappa(\cdot,\theta\cdot)}{\kappa(H,H)}
\end{gather*}
and identify
\begin{gather*}
\I:\mf p\to\mf p^*,\quad X\mapsto \langle X,\cdot\rangle.
\end{gather*}
If $X_1,\ldots,X_{\dim\mf p}$ is a basis of $\mf p$ we denote the dual basis with respect to $\langle\cdot,\cdot\rangle$ by $\tilde X_1,\ldots,\tilde X_{\dim\mf p}$, i.e.
\begin{gather*}
\I(\tilde X_i)(X_j)=\langle\tilde X_i,X_j\rangle=\delta_{ij}.
\end{gather*}
\end{notation}

\begin{lemma}\label{la:gen_diffops}
For $Y\in\hat K_M$ let $\mathrm{d}_V^Y\coloneqq T_V^Y\circ\nabla$ with $T_V^Y\in\Hom{K}{Y\otimes\mf p^*}{V}$, where $V\leq L^2(K)$ denotes an irreducible subrepresentation of $Y\otimes\mf p^*$, be a generalized gradient and $\mu\in\mf a^*$. Choose a basis $X_1,\ldots,X_{\dim\mf p}$ of $\mf p_0$ such that $X_1\in\mf a$ and $X_j\in\mf k\oplus\mf n$ (e.g.\@ an orthonormal basis of $\mf p$ with $X_1\in\mf a$). Let
\begin{gather*}
p_{Y,\mu}\coloneqq(\mu+\rhoa)(X_1)\phi_Y\otimes\I(\tilde X_1)-\sum_{j=2}^{\dim\mf p}\ell(\Ik{X_j})\phi_Y\otimes\I(\tilde X_j)\in Y\otimes\mf p^*,
\end{gather*}
where $\Ik{X_j}\in\mf k$ denotes the $\mf k$-component in the $\mf k\oplus\mf a\oplus\mf n$-decomposition of $X_j$.
Then
\begin{enumerate}
\item\label{it:gen_diffops1}$p_{Y,\mu}$ is independent of the basis and $M$-invariant, 
\item\label{it:gen_diffops2} $\mathrm{d}_V^Y\circ P_\mu^Y=T_V^Y(p_{Y,\mu})(e)P_\mu^V$ if $V$ is $M$-spherical, i.e.\@ $V\leq L^2(K)^M$,
\item\label{it:gen_diffops3}$\mathrm{d}_V^Y\circ P_\mu^Y=0$ if $V$ is not $M$-spherical, i.e.\@ $V^M=0$.
\end{enumerate}
\end{lemma}

\begin{proof}
\ref{it:gen_diffops1} Identifying
\begin{gather*}
Y\otimes\mf p^*\cong\Hom{}{\mf p}{Y},\quad f\otimes\lambda\mapsto(X\mapsto\lambda(X)f),
\end{gather*}
the tensor $p_{Y,\mu}$ corresponds to the homomorphism given by
\begin{alignat*}{2}
p_{Y,\mu}(X)&=(\mu+\rhoa)(X)\phi_Y\quad &&\forall X\in\mf a,\\
p_{Y,\mu}(X)&=\ell(\Ik{X})\phi_Y\ &&\forall X\in\mf p\cap(\mf k\oplus\mf n), 
\end{alignat*}
which is independent of the basis. For the $M$-invariance note first that the $K$-action on $\Hom{}{\mf p}{Y}$ is given by
\begin{gather*}
(k.\Phi)(X)=k.\Phi(k^{-1}.X)=L(k)\Phi(\on{Ad}(k^{-1})X),\quad X\in\mf p,\ \Phi\in\Hom{}{\mf p}{Y}.
\end{gather*}
Since $M$ stabilizes $\mf a$ and $\phi_Y$ is $M$-invariant we have for each  $X\in\mf a$,
\begin{align*}
(m.p_{Y,\mu})(X)&=L(m)p_{Y,\mu}(\on{Ad}(m^{-1})X)=L(m)p_{Y,\mu}(X)\\
&=(\mu+\rhoa)(X)L(m)\phi_Y=(\mu+\rhoa)(X)\phi_Y=p_{Y,\mu}(X).
\end{align*}
Moreover, since $M$ leaves $\mf k,\, \mf a$ and $\mf n$ invariant, we have for each $X\in\mf p\cap(\mf k\oplus\mf n)$,
\begin{align*}
(m.p_{Y,\mu})(X)&=L(m)p_{Y,\mu}(\on{Ad}(m^{-1})X)=L(m)\ell(\Ik{\on{Ad}(m^{-1})X})\phi_Y\\
&=L(m)\ell(\on{Ad}(m^{-1})\Ik{X})\phi_Y\\
&=L(m)L(m^{-1})\ell(\Ik{X})L(m)\phi_Y\\
&=\ell(\Ik{X})\phi_Y=p_{Y,\mu}(X).
\end{align*}
This proves the first part.\par
\ref{it:gen_diffops2}, \ref{it:gen_diffops3} Let $\delta_{eM}$ denote the Delta distribution at $eM$ on $K/M$. Then 
\begin{gather}\label{eq:PT_at_delta}
P_\mu^{Y}(\delta_{eM})(g)=a_{I}(g^{-1})^{-(\mu+\rho)}\tau(k_{I}(g^{-1}))\phi_Y\in C^\infty(G\times_KY).
\end{gather}
We first obtain 
\begin{align*}
(\nabla\circ P_\mu^{Y}(\delta_{eM}))(e)(X_1)&=\d P_\mu^{Y}(\delta_{eM})(\exp tX_1)\\
&=\d a_{I}(\exp-tX_1)^{-(\mu+\rho)}\phi_Y\\
&=\d e^{t(\mu+\rhoa)(X_1)}\phi_Y=(\mu+\rhoa)(X_1)\phi_Y.
\end{align*}
For $j\in\{2,\ldots,\dim\mf p\}$ we write $X_j=\Ik{X_j}+\In{X_j}\in\mf k_0\oplus\mf n_0$ and obtain
\begin{align*}
(\nabla\circ P_\mu^{Y}(\delta_{eM}))(e)(X_j)&=(\nabla\circ P_\mu^{Y}(\delta_{eM}))(e)(\Ik{X_j})
+(\nabla\circ P_\mu^{Y}(\delta_{eM}))(e)(\In{X_j})\\
&=(\nabla\circ P_\mu^{Y}(\delta_{eM}))(e)(\Ik{X_j})\\
&=\d\tau(\exp-t\Ik{X_j})\phi_Y=-\ell(\Ik{X_j})\phi_Y,
\end{align*}
where we used in the second step that $P_\mu^{Y}(\delta_{eM})(n)=\phi_Y$ for $n\in N$ by \eqref{eq:PT_at_delta}. Thus,
\begin{gather*}
(\nabla\circ P_\mu^{Y}(\delta_{eM}))(e)=(\mu+\rhoa)(X_1)\phi_Y\otimes\I(\tilde X_1)-\sum_{j=2}^{\dim\mf p}\ell(\Ik{X_j})\phi_Y\otimes\I(\tilde X_j)
\end{gather*}
and therefore
\begin{align*}
(\mathrm{d}_V^Y\circ P_\mu^{Y}(\delta_{eM}))(e)&=T_V^Y((\nabla\circ P_\mu^{Y}(\delta_{eM}))(e))\\
&=T_V^Y\left((\mu+\rhoa)(X_1)\phi_Y\otimes\I(\tilde X_1)-\sum_{j=2}^{\dim\mf p}\ell(\Ik{X_j})\phi_Y\otimes\I(\tilde X_j)
\right).
\end{align*}
By Corollary \ref{cor:no_poisson_transforms} and \ref{cor:constant_Poisson_transforms}, $\mathrm{d}_V^Y\circ P_\mu^{Y}$ has to be a multiple of $P_\mu^{V}$ if $V$ is $M$-spherical and $0$ otherwise. In particular, we deduce that
\begin{gather*}
T_V^Y\left((\mu+\rhoa)(X_1)\phi_Y\otimes\I(\tilde X_1)-\sum_{j=2}^{\dim\mf p}\ell(\Ik{X_j})\phi_Y\otimes\I(\tilde X_j)\right)
\end{gather*}
is a multiple of $P_\mu^{V}(\delta_{eM})(e)=\phi_V$. Since $\phi_V(e)=1$ this multiple is given by
\begin{gather*}
T_V^Y\left((\mu+\rhoa)(X_1)\phi_Y\otimes\I(\tilde X_1)-\sum_{j=2}^{\dim\mf p}\ell(\Ik{X_j})\phi_Y\otimes\I(\tilde X_j)\right)(e).\qedhere
\end{gather*}
\end{proof}

\section{\texorpdfstring{$\Gamma$-invariant elements}{Gamma-invariant elements}}\label{sec:Gamma_inv}
 In this section we investigate which principal series representations admit $\Gamma$-invariant distributional elements and, if the representation is reducible, in which composition factors they can occur. We do not have to assume that the co-compact lattice $\Gamma\leq G$ is torsion free in this section.

\begin{theorem}[Location of $\Gamma$-invariant elements]\label{thm:Gamma_inv_ps}
Let $\sigma\in\hat M$ and $\mu\in\mf a^*$. Assume that the socle of $\ps{\sigma}{\mu}$ decomposes multiplicity-freely. Then
\begin{gather*}
{}^\Gamma\psd{\sigma}{\mu}\cong{}^\Gamma(\on{soc}\ps{\sigma}{\mu})^{-\infty}=\bigoplus_{V\leq\ps{\sigma}{\mu}\text{ irred.}}{}^\Gamma V^{-\infty},
\end{gather*}
where the sum on the right hand side is finite. Moreover, for each irreducible $V\leq\ps{\sigma}{\mu}$, the existence of $\Gamma$-invariant distributional elements in $V$ implies that $V$ is infinitesimally unitary.
\end{theorem}

\begin{proof}
Note first that $\ps{\sigma}{\mu}$ has finitely many irreducible subrepresentations by the finite length of $\ps{\sigma}{\mu}$ and our multiplicity one assumption. We claim that the dual principal series representation $\ps{\tilde\sigma}{-\mu}$ has finitely many irreducible quotients. Indeed, let $\ps{\tilde\sigma}{-\mu}/V$, for some subrepresentation $V\leq\ps{\tilde\sigma}{-\mu}$, denote an irreducible quotient of $\ps{\tilde\sigma}{-\mu}$. Then we have that $V^{\perp_{\tilde\sigma,-\mu}}\leq\ps{\sigma}{\mu}$ is a subrepresentation (see Equation \eqref{eq:perp_ps} for the notation). Moreover, $V^{\perp_{\tilde\sigma,-\mu}}\leq\ps{\sigma}{\mu}$ is the dual representation of $\ps{\tilde\sigma}{-\mu}/V$ and therefore irreducible. If $\ps{\tilde\sigma}{-\mu}/V_1\neq\ps{\tilde\sigma}{-\mu}/V_2$ are two different irreducible quotients, we obtain two different irreducible subrepresentations $V_1^{\perp_{\tilde\sigma,-\mu}}\neq V_2^{\perp_{\tilde\sigma,-\mu}}\leq\ps{\sigma}{\mu}$ by the non-degeneracy of $\langle\cdot,\cdot\rangle_{\tilde\sigma,-\mu}$. Since there are only finitely many of the latter, $\ps{\tilde\sigma}{-\mu}$ resp.\@ $\pss{\tilde\sigma}{-\mu}$ has finitely many irreducible quotients $\ps{\tilde\sigma}{-\mu}/V_j,\ j=1,\ldots,n$ resp.\@ $\pss{\tilde\sigma}{-\mu}/V_j^\infty,\ j=1,\ldots,n$.

By definition we have that $\psd{\sigma}{\mu}=\Hom{\C}{\pss{\tilde\sigma}{-\mu}}{\C}$ is the space of continuous linear maps from $\pss{\tilde\sigma}{-\mu}$ to $\C$, equipped with the dual representation of $\pss{\tilde\sigma}{-\mu}$. This implies that
\begin{gather}\label{eq:pf_Gamma_inv}
{}^\Gamma\psd{\sigma}{\mu}={}^\Gamma\Hom{\C}{\pss{\tilde\sigma}{-\mu}}{\C}=\Hom{\Gamma}{\pss{\tilde\sigma}{-\mu}}{\C}.
\end{gather}
Note that $\pss{\tilde\sigma}{-\mu}$ is a nuclear Fréchet space (consider the compact picture and see e.g.\@ \cite[§2]{CHM00}) and a differentiable $G$-module. Moreover, $\C$ is a differentiable nuclear $\Gamma$-module. Therefore we may use Frobenius reciprocity to obtain (see \cite[La.\@ 1.3]{Z78})
\begin{gather*}
{}^\Gamma\psd{\sigma}{\mu}=\Hom{\Gamma}{\pss{\tilde\sigma}{-\mu}}{\C}\cong\Hom{G}{\pss{\tilde\sigma}{-\mu}}{\on{Ind}_\Gamma^{G,\infty}(\C)},
\end{gather*}
where $\on{Ind}_\Gamma^{G,\infty}(\C)\cong C^\infty(\Gamma\backslash G)$ denotes the representation smoothly induced by the trivial representation of $\Gamma$. 
By \cite[Thm.\@, ch.\@ 1, §2.3]{GGPS69}, there exists a countable subset $\hat G_\Gamma\subset\hat G$ such that $\on{Ind}_\Gamma^{G}(\C)$ decomposes as a direct sum
\begin{gather*}
\on{Ind}_\Gamma^{G}(\C)\cong\widehat\bigoplus_{\pi\in\hat G_\Gamma}m_\Gamma(\pi)\pi,
\end{gather*}
where each multiplicity $m_\Gamma(\pi)\geq 1$ is finite. Therefore, if $0\neq\varphi\in{}^\Gamma\psd{\sigma}{\mu}$ with corresponding $\varphi_F\in\Hom{G}{\pss{\tilde\sigma}{-\mu}}{\on{Ind}_\Gamma^{G,\infty}(\C)}$, there exists some $\pi\in\hat G_\Gamma$ such that $\on{pr}_\pi\circ\,\varphi_F\neq 0$, where $\on{pr}_\pi$ denotes the orthogonal projection onto one copy of $\pi$ in $\on{Ind}_\Gamma^{G}(\C)$. Since $\varphi_F$ and $\on{pr}_\pi$ are continuous and linear they are smooth. Therefore, $\on{pr}_\pi\circ\,\varphi_F$ maps $\pss{\tilde\sigma}{-\mu}$ into $\pi^\infty$. By \cite[§4.4, p.\@ 253]{W72}, $\pss{\tilde\sigma}{-\mu}$ and $\pi^\infty$ are smooth Fréchet representations. Therefore, the image of $\on{pr}_\pi\circ\,\varphi_F$ is closed and a topological summand of $\pi^\infty$ \cite[La.\@ 11.5.1, Thm.\@ 11.6.7(2)]{W92}. Since $\pi$ is irreducible, $\pi^\infty$ is irreducible (see e.g.\@ \cite[p.\@ 254]{W72}) and therefore $\on{pr}_\pi\circ\,\varphi_F$ is surjective. Now \cite[Thm.\@ 12.16.8]{D70} implies that the canonical factorization $\pss{\tilde\sigma}{-\mu}/\ker(\on{pr}_\pi\circ\,\varphi_F)\to\pi^\infty$ is a topological isomorphism. Since $\pi^\infty$ is irreducible, $\pss{\tilde\sigma}{-\mu}/\ker(\on{pr}_\pi\circ\,\varphi_F)$ is irreducible. It follows that $\ker(\on{pr}_\pi\circ\varphi_F)=V_j^\infty$ for some $j\in\{1,\ldots,n\}$.
Thus we proved that if $\on{pr}_\pi\circ\varphi_F\neq0$, then it factors through an irreducible quotient of $\pss{\tilde\sigma}{-\mu}$. 

Consider the finite set
\begin{gather*}
F\coloneqq\{\pi\in\hat G_{\Gamma}\mid\exists\:j\in\{1,\ldots,n\}\colon \pi^\infty\cong\pss{\tilde\sigma}{-\mu}/V_j^\infty\}.
\end{gather*}
For $\pi\in F$ with $\pi^\infty\cong\pss{\tilde\sigma}{-\mu}/V_j^\infty$ we set $j(\pi)\coloneqq j$. Moreover, let
\begin{gather*}
 I_\Gamma\coloneqq\{j\in\{1,\ldots,n\}\mid\exists\: \pi_j\coloneqq\pi\in F\colon j(\pi)=j\}.
\end{gather*}
Then
\begin{align*}
\Hom{G}{\pss{\tilde\sigma}{-\mu}}{\on{Ind}_\Gamma^{G,\infty}(\C)}&=\Hom{G}{\pss{\tilde\sigma}{-\mu}}{\bigoplus_{\pi\in F}m_{\Gamma}(\pi)\pi}\\
&\cong\bigoplus_{\pi\in F}\bigoplus_{k=1}^{m_{\Gamma}(\pi)}\Hom{G}{\pss{\tilde\sigma}{-\mu}}{\pi}\\
&\cong\bigoplus_{\pi\in F}\bigoplus_{k=1}^{m_{\Gamma}(\pi)}\Hom{G}{\ps{\tilde\sigma}{-\mu}^\infty/V_{j(\pi)}^\infty}{\pi}\\
&\cong\bigoplus_{\pi\in F}\Hom{G}{\ps{\tilde\sigma}{-\mu}^\infty/V_{j(\pi)}^\infty}{m_{\Gamma}(\pi)\pi}\\
&\cong\bigoplus_{j\in I_\Gamma}\Hom{G}{\ps{\tilde\sigma}{-\mu}^\infty/V_{j}^\infty}{m_{\Gamma}(\pi_j)\pi_j}\\
&\cong\bigoplus_{j\in I_\Gamma}\Hom{G}{\ps{\tilde\sigma}{-\mu}^\infty/V_{j}^\infty}{\on{Ind}_\Gamma^{G,\infty}(\C)}\\
&\cong\bigoplus_{j\in I_\Gamma}\Hom{\Gamma}{\ps{\tilde\sigma}{-\mu}^\infty/V_{j}^\infty}{\C}\\
&\cong\bigoplus_{j\in I_\Gamma}\Hom{\Gamma}{(\ps{\tilde\sigma}{-\mu}/V_{j})^\infty}{\C}.
\end{align*}
Note that the dual representation of $\ps{\tilde\sigma}{-\mu}/V_j$ is given by $W_j\coloneqq V_j^{\perp_{\tilde\sigma,-\mu}}\leq\ps{\sigma}{\mu}$. Therefore, as in \eqref{eq:pf_Gamma_inv},
\begin{gather*}
\bigoplus_{j\in I_\Gamma}\Hom{\Gamma}{(\ps{\tilde\sigma}{-\mu}/V_{j})^\infty}{\C}=\bigoplus_{j\in I_\Gamma}{}^\Gamma W_j^{-\infty}.
\end{gather*}
This proves the first part. We now prove the second part concerning the infinitesimal unitarity. Let $\varphi_F$ and $\pi$ as above. Then, denoting the $K$-finite elements by $\cdot_K$, we have (cf.\@ \cite[Cor.\@ 11.6.8]{W92})
\begin{gather*}
\left(\pss{\tilde\sigma}{-\mu}/\ker(\on{pr}_\pi\circ\varphi_F)\right)_K\cong\pi_K
\end{gather*}
as $(\mf g,K)$-modules. Since $\pi$ is unitary we infer that $\ps{\tilde\sigma}{-\mu}/\ker(\on{pr}_\pi\circ \varphi_F)$ is infinitesimally unitary.
\end{proof}

Note that Theorem \ref{thm:Gamma_inv_ps} applies if $\ps{\sigma}{\mu}$ is irreducible. The following proposition shows that the hypotheses of Theorem \ref{thm:Gamma_inv_ps} are in particular satisfied in the rank one case.

\begin{proposition}\label{prop:soc_multone_rank_one}
Let $G$ be of real rank one. Then the socle of $\ps{\sigma}{\mu}$ decomposes multiplicity-freely for each $\sigma\in\hat M$ and $\mu\in\mf a^*$.
\end{proposition}

\begin{proof}
See \cite[Theorem (6.1.3)]{C85}.
\end{proof}

\begin{example}
Figure \ref{fig:unit_sub} describes the spherical principal series representations which can possibly contain $\Gamma$-invariant elements for $G=\GSO{n},\ n\geq 2$, and $G=\GSp{n},\ n\geq 2$. The unitary principal series is given by $\mu\in i\mf a_0^*$ in both cases and the complementary series consists of the parameters $\mu$ with $\mu(H)\in]-\rhoa(H),\rhoa(H)[$ resp.\@ $\mu(H)\in]-\rhoa(H)+2,\rhoa(H)-2[$, where $H\in\mf a_0$ as before denotes the unique element with $\alpha(H)=1$ for the unique simple positive real root $\alpha$. Moreover, $\pst{\mu}$ is reducible if and only if $\mu\in\pm(\rhoa+\N_0\alpha)$ resp.\@ $\mu\in\pm(\rhoa+(2\N_0-2)\alpha)$ and $\mu$ is exceptional if and only if $\pst{\mu}$ is reducible and has a unitarizable subrepresentation. In each case, the constant functions form an irreducible subspace of $\pst{\rhoa}$ and thus ${}^\Gamma(\on{soc}\pst{\rhoa})^{-\infty}\neq\{0\}$.
\end{example}
\begin{figure}
\begin{center}
\begin{minipage}{0.45\textwidth}
\begin{tikzpicture}[scale=1.5,extended line/.style={shorten >=-#1,shorten <=-#1},]
\draw [->](0,-1.2)--(0,1.2) node[right]{$i\mf a_0^*$};
\draw [->](-1.7,0)--(1.7,0) node[right]{$\mf a_0^*$};
\foreach \x/\xtext in {-1.5/{}, -1/{}, -0.5/-\rhoa, 0.5/\phantom{-}\hspace{-0.25cm}\rhoa, 1/{}, 1.5/{}}
{\draw (\x cm,1pt ) -- (\x cm,-1pt ) node[anchor=north] {$\scriptstyle{\xtext}$};}
\foreach \y/\ytext in {-1/{}, -0.5/{},0.5/{}, 1/{}}
{\draw (0pt,\y cm) -- (-0pt ,\y cm) node[anchor=east] {$\ytext$};}
\draw [red,extended line=0cm,line width=0.5mm] (0,-1.2)--(0,1.18);  
\draw [red,extended line=0cm,line width=0.5mm] (-.5,0)--(.5,0); 
\fill [red](-.5,0) circle(1pt);
\fill [red](-1,0) circle(1pt);
\fill [red](-1.5,0) circle(1pt);
\fill (.5,0) circle(1pt);
\fill (1,0) circle(1pt);
\fill (1.5,0) circle(1pt);
\end{tikzpicture}
\end{minipage}
\begin{minipage}{0.45\textwidth}
\begin{tikzpicture}[scale=1.5,extended line/.style={shorten >=-#1,shorten <=-#1},]
\draw [->](0,-1.2)--(0,1.2) node[right]{$i\mf a_0^*$};
\draw [->](-1.7,0)--(1.7,0) node[right]{$\mf a_0^*$};
\foreach \x/\xtext in {-1.5/{}, -1/{-\rhoa}, -0.5/-\rhoa+2\alpha, 0.5/\phantom{-}\rhoa-2\alpha, 1/{\hspace{-0.25cm}\phantom{-}\rhoa}, 1.5/{}}
{\draw (\x cm,1pt ) -- (\x cm,-1pt ) node[anchor=north] {$\scriptstyle{\xtext}$};}
\foreach \y/\ytext in {-1/{}, -0.5/{},0.5/{}, 1/{}}
{\draw (0pt,\y cm) -- (-0pt ,\y cm) node[anchor=east] {$\ytext$};}
\draw [red,extended line=0cm,line width=0.5mm] (0,-1.2)--(0,1.18);  
\draw [red,extended line=0cm,line width=0.5mm] (-.5,0)--(.5,0);  
\fill [red](-.5,0) circle(1pt);
\fill [red](-1,0) circle(1pt);
\fill [red](-1.5,0) circle(1pt);
\fill (.5,0) circle(1pt);
\fill (1,0) circle(1pt);
\fill (1.5,0) circle(1pt);
\end{tikzpicture}
\end{minipage}
\end{center}
\caption{Parameters $\mu$ for which $\pst{\mu}$ has a unitarizable subrepresentation (red) resp.\@ is reducible (dots) for $G=\GSO{n},\ n\geq 2,$ (left) resp.\@ $G=\GSp{n},\ n\geq2,$ (right). The exceptional set is given by the red dots.}
\label{fig:unit_sub}
\end{figure}
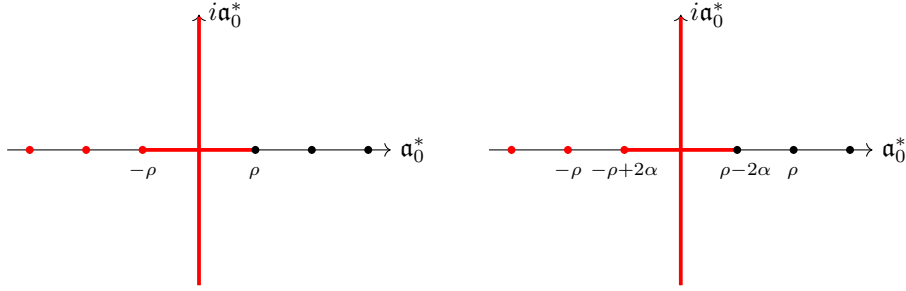

\begin{remark}\label{rem:poisson_qc}
Recall from Theorem \ref{thm:Gamma_inv_ps} that
\begin{gather*}
{}^\Gamma\pstd{\mu}=\bigoplus_{U\in\mathbf{Irr}(\mu)}{}^\Gamma U^{-\infty}.
\end{gather*}
Choosing $(\tau_U,V_{\tau_U})$ as in Proposition \ref{prop:Poisson_sum} (e.g.\@ a minimal $K$-type of $U$) we have by Proposition \ref{prop:PoissonInjective} that each $\restr{P_\mu^{\tau_U}}{U^{-\infty}}$ is injective and therefore
\begin{gather*}
{}^\Gamma\pstd{\mu}\cong\bigoplus_{U\in\mathbf{Irr}(\mu)}{}^\Gamma P_\mu^{\tau_U}(U^{-\infty})\subseteq\bigoplus_{U\in\mathbf{Irr}(\mu)}{}^\Gamma C^\infty(G\times_KV_{\tau_U}).
\end{gather*}
\end{remark}

We describe the socle in more detail.
\begin{theorem}\label{thm:socle}
Denoting the set of minimal $K$-types by $\tau_{\on{min}}$ and the Harish-Chandra module of $\on{soc}(\pst{\mu})$ by $\on{soc}(\pst{\mu})_K$ we have (see Appendix \ref{app:comp_scalars} for the notation)
\begin{center}
\def\arraystretch{1.2}
\setlength{\tabcolsep}{10pt}
\begin{tabular}{l|lll}
$G$&$\mathbf{Ex}=\{\mu_\ell\mid\ell\in\N_0\}$& $\on{soc}(\pst{\mu_\ell})_K$&$\tau_{\on{min}}(\on{soc}(\pst{\mu_\ell}))$\\\hline
$\GSO{2}$&$\mu_\ell=-\rhoa-\ell\alpha$&$\bigoplus_{k\geq\ell+1} Y_{k}\oplus Y_{-k}$&$\{Y_{-(\ell+1)}, Y_{(\ell+1)}\}$\\
$\GSO{n},\ n\geq3$&$\mu_\ell=-\rhoa-\ell\alpha$&$\bigoplus_{k=\ell+1}^\infty Y_k$&$\{Y_{\ell+1}\}$\\
$\GSU{n},\ n\geq2$&$\mu_\ell=-\rhoa-2\ell\alpha$&$\bigoplus_{p,q=\ell+1}^\infty Y_{p,q}$&$\{Y_{\ell+1,\ell+1}\}$\\
$\GSp{n},\ n\geq2$&$\mu_\ell=-\rhoa-(2\ell-2)\alpha$&$\bigoplus_{a\geq b\geq\ell+1} V_{a,b}$&$\{V_{\ell+1,\ell+1}\}$\\
\GF&$\mu_\ell=-\rhoa-(2\ell-6)\alpha$&$\bigoplus_{\substack{m-k\geq 2\ell+2\\m\equiv k\on{ mod }2}} V_{m,k}$&$\{V_{2\ell+2,0}\}$
\end{tabular}
\end{center}
In each case, every irreducible subrepresentation of $\on{soc}(\pst{\mu})$ is unitarizable and has a unique minimal $K$-type. For $G\neq\GSO{2}$ the socle is irreducible for all exceptional parameters. For $G=\GSO{2}$ the socle decomposes into two irreducible subrepresentations which are given by discrete series representations.
\def\arraystretch{1}
\end{theorem}

\begin{proof}
The exceptional parameters can be computed by Example \ref{ex:Ex_rank_one} and Table \ref{table:structure_rank_one}. Using \cite[Thm.\@ 5.1 (2-4)]{JW} resp.\@ \cite[Thm.\@ 5.2 (2)]{JW2} with $\nu=(\rhoa-\mu_\ell)(H)$ we have that $\on{soc}(\pst{\mu_\ell})$ is irreducible and can determine its $K$-module structure for $G\neq\GSO{2}$. Moreover, \cite[Thm.\@ 6.3 (1-3)]{JW} resp.\@ \cite[Thm.\@ 5.3 (2)]{JW2} show that these socles are unitarizable. For $G=\GSO{2}$ the decomposition of the socle follows from \cite[p.\@ 38]{Kna86} with $n=2(\ell+1)$, where two (unitary, irreducible) discrete series representations $\mc D_{2(\ell+1)}^+$ and $\mc D_{2(\ell+1)}^-$ occur. The $K$-types of these representations are determined in \cite[p.\ 40]{Kna86}. The highest weights of the $K$-representations needed for the computation of the minimal $K$-types are determined in Appendix \ref{app:comp_scalars}.
\end{proof}

\begin{theorem}[Langlands parameters]\label{thm:langlands}
We have the following Langlands parameters for $\on{soc}(\pst{\mu_\ell}),\ \mu_\ell\in\mathbf{Ex}$ (see Theorem \ref{thm:socle}), in the notation of \cite[Thm.\@ 8.54]{Kna86}
\begin{center}
\def\arraystretch{1.2}
\setlength{\tabcolsep}{13pt}
\begin{tabular}{l|lll}
$G$&$S$& $\omega\in\hat M$&$\nu\in\mf a^*$\\\hline
\multirow{2}{*}{$\GSO{n},\ n\geq2$}&$G\text{ if }n=2$&$-$&$-$\\
&$P\text{ if }n\neq2$&$(\ell+1)e_1$&$(n-\frac{3}{2})\alpha$\\\hline
\multirow{2}{*}{$\GSU{n},\ n\geq2$}&$G$ if $n=2$&$-$&$-$\\
&$P$ if $n\neq2$&$(\ell+1)(\overline{\overline{\varepsilon}}_2-\overline{\overline{\varepsilon}}_n)$&$(n-2)\alpha$\\\hline
\multirow{2}{*}{$\GSp{n},\ n\geq2$}&$G\text{ if }n=2$&$-$&$-$\\
&$P$ if $n\neq2$&$(\ell+1)(\overline{\varepsilon}_2+\overline{\varepsilon}_3)$&$(2n-3)\alpha$\\\hline
\GF&$G$&$-$&$-$
\end{tabular}
\end{center}
Here, the highest weight of the $M$-representation $\omega$ is denoted as in \cite[Lem.\@ 4.3, 5.3]{B79} for $G\in\{\GSU{n},\GSp{n}\}$ and as in Appendix \ref{app:comp_scalars} for $G=\GSO{n}$ (then $M\cong\mathrm{SO}(n-1)$). By definition, if $S=G$, the socle $\on{soc}(\pst{\mu_\ell})$ is tempered. Moreover, in these cases, it is a discrete series representation if and only if $\mu_\ell(H)\leq-\rhoa(H)$. The Blattner parameter of the discrete series (see \cite[Terminology p.\@ 310]{Kna86}) is given by its minimal $K$-type. If $\mu_\ell(H)>-\rhoa(H)$, the socle is a limit of discrete series representation (this case only occurs for $G=\GSp{2}$ and $G=\GF$).
\end{theorem}

\begin{proof}
Using the branching rules described in \cite{B79} and \cite[Thm.\@ 9.16]{knapplie} we first try to find $\omega\in\hat M$ such that the minimal $K$-type of $\on{soc}(\pst{\mu_\ell})$ is also minimal for the induced representation $\on{Ind}_M^K(\omega)$. To determine $\nu\in\mf a^*$ we compare the infinitesimal character of the socle, which is the same as that of $\pst{\mu_\ell}$, with the infinitesimal character of the principal series representation corresponding to the pair $(\omega,\nu)$. They have to coincide up to the action of an element of the Weyl group and can be calculated using \cite[Prop.\@ 8.22]{Kna86}. If one of the two steps above does not work, we must have $S=G$, i.e.\@ the socle is tempered. In this case \cite[Thm.\@ 14.2]{KZ82} shows that it has to be a discrete series representation or a limit of discrete series representation depending on the infinitesimal character being regular or singular. The connection to the Blattner parameter follows from \cite[ch. XV,\S 1, Ex.\@ (1)]{Kna86}.
\end{proof}

\section{Fourier series}\label{sec:Fourier}
In this section we consider spherical principal series representations for exceptional parameters in the rank one case. Our aim is to find explicit realizations of the unitary irreducible subrepresentations occurring in Theorem \ref{thm:Gamma_inv_ps} in the space of smooth sections of a specific vector bundle. For this purpose we determine conditions the images of $\Gamma$-invariant elements under the injective vector valued Poisson transforms from Section \ref{subsec:inj_PT} have to satisfy (Lemma \ref{la:gen_diffops2}). We then prove that these conditions suffice to describe the image (Theorem \ref{thm:fourier_char}) and use this characterization to give explicit descriptions of the images in each of the cases listed in Section \ref{sec:spec_corres}.

\subsection{Fourier expansions}
In the following we describe a generalized Fourier series that is closely related to the Poisson transform and essentially gives that, properly interpreted, each $f\in\pst{\mu}$ is the sum of all its Poisson transform images.
\begin{definition}\label{def:pi_gamma}
For each $Y\in\hat K_M$ let
\begin{gather*}
\pi_Y:C^\infty(G\times_KY)\hookrightarrow C^\infty(G)^M,\quad\pi_Y(\varphi)(g)\coloneqq \varphi(g)(e),
\end{gather*}
where $C^\infty(G)^M$ denotes the right $M$-invariant elements in $C^\infty(G)$.
Moreover, let $\mc D'(G\times_K~Y)$ denote the dual of $C_c^\infty(G\times_K\tilde Y)$, where we realize the dual representation $\tilde Y$ of $Y$ as the complex conjugate representation of $Y$. We embed $C^\infty(G/M)$ into $\mc D'(G/M)$ by
\begin{gather*}
\iota_{G/M}\colon C^\infty(G/M)\hookrightarrow\mc D'(G/M),\quad\iota_{G/M}(f)(\varphi)\coloneqq\int_G f(gM)\varphi(gM)\intd g
\end{gather*}
and $C^\infty(G\times_KY)$ into $\mc D'(G\times_KY)$ by
\begin{gather*}
\iota_Y:C^\infty(G\times_KY)\hookrightarrow\mc D'(G\times_KY),\quad\iota_Y(f)(\varphi)\coloneqq\int_{G}\pi_Y(f)(g)\pi_{\tilde Y}(\varphi)(g)\intd g.
\end{gather*}
If it is clear from the context we omit the embeddings $\iota_*$ for the sake of readability. We further define the pullback
\begin{gather*}
\pi_{Y}^*:\mc D'(G/M)\to\mc D'(G\times_KY),\quad \pi_{Y}^*(f)(\varphi)\coloneqq f(\pi_{\tilde Y}(\varphi)).
\end{gather*}
\end{definition}

\begin{lemma}\label{la:decomp_smooth}
Let $f\in C^\infty(G)^M$ be a right $M$-invariant smooth function and 
\begin{gather*}
\on{pr}_{Y_\tau}:L^2(K/M)\to Y_\tau
\end{gather*}
denote the orthogonal projection onto $Y_\tau\in\hat K_M$. For every fixed $g\in G$, the series
\begin{gather*}
\sum_{\tau\in\hat K_M}\on{pr}_{Y_\tau}(f(g\bigcdott)),
\end{gather*} 
where $f(g\bigcdott)\in C^\infty(K/M)$ is defined by
\begin{gather*}
f(g\bigcdott):K/M\to\C,\quad kM\mapsto f(gk),
\end{gather*}
converges absolutely and uniformly to $f(g\bigcdott)$. Moreover, we can uniquely decompose 
 \begin{gather*}
 	f=\sum_{\tau\in\hat K_M}f_{Y_\tau}
 \end{gather*}
 with $f_{Y_\tau}\in\pi_{Y_\tau}(C^\infty(G\times_KY_\tau))$ where the series converges pointwise. The functions $f_{Y_\tau}$ are given by
 \begin{gather*}
	f_{Y_\tau}=\pi_{Y_\tau}(g\mapsto\on{pr}_{Y_\tau}(f(g\bigcdott))).
\end{gather*} 
\end{lemma}
 
 \begin{proof}
  We decompose $f$ according to the right regular $K$-representation: For fixed $g\in G$ consider the function
 \begin{gather*}
  f_g:K\to\C,\quad f_g(k)\coloneqq f(gk).
 \end{gather*}
  Then $f_g\in C^\infty(K)^M\cong C^\infty(K/M)$. Since, by definition of $\hat K_M$,
\begin{gather*}
L^2(K/M)\cong\widehat\bigoplus_{\tau\in\hat K_M}Y_\tau
\end{gather*}
we can decompose
  \begin{gather}\label{eq:decomp_smooth}
    f_g=\sum_{\tau\in\hat K_M}\on{pr}_{Y_\tau}(f_g)
  \end{gather}
  which converges in the Hilbert sense. By \cite[ch.\@ V, Theorem 3.5 (iii)]{GAGA} this convergence is absolute and uniform. Thus, we have for every $k\in K$
  \begin{gather*}
  f(gk)=f_g(k)=\sum_{\tau\in\hat K_M}\on{pr}_{Y_\tau}(f_g)(k).
  \end{gather*}
  Note that $(g\mapsto\on{pr}_{Y_\tau}(f_g))\in C^\infty(G\times_KY_\tau)$; indeed
  \begin{gather*}
  	f_{g\tilde k}(k)=f(g\tilde kk)=f_g(\tilde k k)=\sum_{\tau\in\hat K_M}\on{pr}_{Y_\tau}(f_g)(\tilde kk)=\sum_{\tau\in\hat K_M}(\tau(\tilde k)^{-1}\on{pr}_{Y_\tau}(f_g))(k)
  \end{gather*}
  implies $\on{pr}_{Y_\tau}(f_{g\tilde k})=\tau(\tilde k)^{-1}\on{pr}_{Y_\tau}(f_g)$ for every $g\in G,\, \tilde k\in K$.
  We can write
  \begin{gather*}
    f=\sum_{\tau\in\hat K_M}f_\tau,
  \end{gather*}
  where $f_\tau\coloneqq\pi_{Y_\tau}(g\mapsto\on{pr}_{Y_\tau}(f_g))\in C^\infty(G)^M$ since, by Equation \eqref{eq:decomp_smooth},
  \begin{gather*}
   f(g)=f_{g}(e)=\sum_{\tau\in\hat K_M}\on{pr}_{Y_\tau}(f_{g})(e)=\sum_{\tau\in\hat K_M}f_\tau(g)
  \end{gather*}
  for every $g\in G$. 
  
  Conversely, for proving uniqueness, let
  \begin{gather*}
  	f=\sum_{\tau\in\hat K_M}\pi_{Y_\tau}(\varphi_\tau)
  \end{gather*}
  for some $\varphi_\tau\in C^\infty(G\times_KY_\tau)$. We need to show that $\varphi_\tau(g)=\on{pr}_{Y_\tau}(f_g)$ for every $g\in G$. We calculate for $k\in K, g\in G$
  \begin{align*}
  	f_g(k)=f(gk)&=\sum_{\tau\in\hat K_M}\pi_{Y_\tau}(\varphi_\tau)(gk)=\sum_{\tau\in\hat K_M}\varphi_\tau(gk)(e)\\
  	&=\sum_{\tau\in\hat K_M}(\tau(k)^{-1}\varphi_\tau(g))(e)=\sum_{\tau\in\hat K_M}\varphi_\tau(g)(k).
  \end{align*}
  This yields $\on{pr}_{Y_{\tau}}(f_g)=\varphi_\tau(g)$ and proves the uniqueness.
\end{proof}

\begin{notation}\label{not:pi_gamma}
Let
\begin{gather*}
\pi_Y\colon\mc D'(G\times_KY)\to\mc D'(G/M),\quad\pi_Y(f)(\varphi)\coloneqq f(\pi_{\tilde Y}^*(\varphi)).
\end{gather*}
In Lemma \ref{la:pi_gamma}\,\ref{it:la_pi_gamma4} we will see that this extends the definition of $\pi_Y$ from Definition \ref{def:pi_gamma}.
\end{notation}

\begin{lemma}\label{la:pi_gamma}
Let $Y\in\hat K_M$ and recall the maps $\iota_{G/M},\, \iota_Y$ from Definition \ref{def:pi_gamma}.
\begin{enumerate}
\item\label{it:la_pi_gamma2} $\pi_{Y}^*(f)(g)=\on{pr}_{Y}(f(g\bigcdot))$ for each $f\in C^\infty(G/M),\ g\in G$, so that $\pi_Y^*(C^\infty(G/M))\subseteq C^\infty(G\times_KY)$ and $\pi_Y^*(C_c^\infty(G/M))\subseteq C_c^\infty(G\times_KY)$,
\item\label{it:la_pi_gamma3} $f=\sum_{\tau\in\hat K_M}\pi_{Y_\tau}(\pi_{Y_\tau}^*(f))$ pointwise for each $f\in C^\infty(G/M)$,
\item\label{it:la_pi_gamma4} $\pi_Y(\iota_Y(f))=\iota_{G/M}(\pi_Y(f))$ for each $f\in C^\infty(G\times_KY)$ and
\item\label{it:la_pi_gamma5} $\forall\mu\in\mf a^*\colon P_\mu^{Y}=\frac{1}{\dim Y}\pi_{Y}^*\circ\mc Q_\mu$ on $\mc D'(K/M)$.
\end{enumerate}
\end{lemma}

\begin{proof}
\ref{it:la_pi_gamma2} By Lemma \ref{la:decomp_smooth} we can write $f=\sum_{\tau\in\hat K_M}\pi_{Y_\tau}(u_\tau)$, where $u_\tau\in C^\infty(G\times_KY_{\tau})$ is given by $u_\tau(g)=\on{pr}_{Y_\tau}(f(g\bigcdot))$. For each $\varphi\in C_c^\infty(G\times_K\tilde Y)$ we use the orthogonality of the $Y_\tau$ to obtain
\begingroup
\allowdisplaybreaks
\begin{align*}
\pi_{Y}^*(f)(\varphi)&=f(\pi_{\tilde Y}(\varphi))=\int_G\pi_{\tilde Y}(\varphi)(g)f(g)\intd g\\
&=\int_{G/K}\int_K\pi_{\tilde Y}(\varphi)(gk)f(gk)\intd k\intd gK\\
&=\int_{G/K}\int_K\varphi(g)(k){\sum_{\tau\in\hat K_M}\pi_{Y_\tau}(u_\tau)}(gk)\intd k\intd gK\\
&=\int_{G/K}\sum_{\tau\in\hat K_M}\int_K\varphi(g)(k)u_\tau(g)(k)\intd k\intd gK\\
&=\int_{G/K}\int_K\varphi(g)(k)u_{Y}(g)(k)\intd k\intd gK\\
&=\int_{G/K}\int_K\pi_{\tilde Y}(\varphi)(gk)\pi_Y(u_Y)(gk)\intd k\intd gK\\
&=\int_{G}\pi_{\tilde Y}(\varphi)(g)\pi_Y(u_Y)(g)\intd g\\
&=\iota_Y(u_Y)(\varphi).
\end{align*}
Note that if $f$ has compact support $\on{supp}f\subset G/M$ and $\on{pr}:G\to G/M$ denotes the canonical projection, we have that $\on{supp}(\pi_Y^*(f))\subseteq\on{pr}^{-1}(\on{supp}f)\cdot K$ is compact since $M$ is compact.
\endgroup
\par
\ref{it:la_pi_gamma3} follows from Lemma \ref{la:decomp_smooth} and \ref{it:la_pi_gamma2}.\par
\ref{it:la_pi_gamma4} Let $f\in C^\infty(G\times_KY)$ and $\varphi\in C_c^\infty(G/M)$. By \ref{it:la_pi_gamma3} we decompose 
\begin{gather*}
\varphi=\sum_{\tau\in\hat K_M}\pi_{Y_\tau}(\pi_{Y_\tau}^*(\varphi))
\end{gather*}
where $\pi_{Y}^*(\varphi)\in C^\infty(G\times_KY)$. By the orthogonality of the $Y_\tau$ we have
\begin{align*}
\iota_{G/M}(\pi_Y(f))(\varphi)&=\int_G\pi_Y(f)(gM)\varphi(gM)\intd g=\int_{G/K}\int_K\pi_Y(f)(gkM)\varphi(gkM)\intd k\intd gK\\
&=\int_{G/K}\sum_{\tau\in\hat K_M}\int_Kf(g)(kM)\pi_{Y_\tau}^*(\varphi)(g)(k)\intd k\intd gK\\
&=\int_{G/K}\int_Kf(g)(kM)\pi_{\tilde Y}^*(\varphi)(g)(k)\intd k\intd gK\\
&=\int_{G}\pi_Y(f)(g)\pi_{\tilde Y}(\pi_{\tilde Y}^*(\varphi))(g)\intd g=\iota_Y(f)(\pi_{\tilde Y}^*(\varphi))=\pi_Y(\iota_Y(f))(\varphi).
\end{align*}
\ref{it:la_pi_gamma5} By continuity (recall Equation \eqref{eq:end_point}) we restrict our attention to smooth functions $\phi\in C^\infty(K/M)$. Then the equality follows from Lemma \ref{la:orth_proj_scalar} and \ref{it:la_pi_gamma2} (recall that $\phi_{Y_\tau}(e)=1$).
\end{proof}

\subsection{Convergence of generalized Fourier series}
In the following we will prove that the convergence in Lemma \ref{la:pi_gamma}\,\ref{it:la_pi_gamma3} is uniform on compact sets and that the same is true for each derivative. Therefore the convergence is a convergence in $C_c^\infty(G/M)$ for $f\in C_c^\infty(G/M)$, where we equip $C_c^\infty(G/M)$ with the inductive limit topology $C_c^\infty(G/M)=\lim_{C\subseteq G/M}C_C^\infty(G/M)$, where the limit runs over all compact subsets $C\subseteq G/M$ and we denote by $C_C^\infty(G/M)\subseteq C_c^\infty(G/M)$ the subset of all functions which are supported in $C$. 

Let $\mc B\coloneqq\{X_1,\ldots,X_n\}\subseteq\mf g_0$ be a basis of $\mf g_0$. For $\ell\in\N_0$ and $C\subset G$ compact we introduce the following norm on $C^\infty(G/M)$ 
\begin{gather*}
\sn{\ell}{C}{f}\coloneqq\sum_{k=0}^\ell\sum_{X_1,\ldots,X_k\in\mc B}\sup_{g\in C}\abs{(X_1\cdots X_k f)(gM)},
\end{gather*}
where $X\in\mf g_0$ acts on $f\in C^\infty(G/M)$ by the derived left regular representation
\begin{gather*}
\forall g\in G\colon(Xf)(gM)\coloneqq\d f(\exp(-tX)gM).
\end{gather*}
The summand for $k=0$ is understood as not differentiating, i.e.\@ as $\sup_{g\in C}\abs{f(gM)}$. We have the following lemma related to the Riemann-Lebesgue lemma. 

\begin{lemma}\label{la:fourier_conv}
Let $f\in C^\infty(G/M)$. For each $C\subset G$ compact, $\ell\in\N_0$ and $N\in\N$ there exists a constant $C_{f,C,N,\ell}>0$ independent of $Y_\tau$ such that
\begin{gather*}
\forall\, Y_\tau\in\hat K_M\colon\sn{\ell}{C}{\pi_{Y_\tau}(\pi_{Y_\tau}^*(f))}\leq C_{f,C,N,\ell}\cdot (1+\norm{\tau}^2)^{-N},
\end{gather*}
where $\norm{\tau}$ denotes the length of the highest weight of $Y_\tau$. Moreover, if $f_n\to 0$ in $C^\infty(G/M)$ we can find $C_{f_n,C,N,\ell}$ such that $\lim_{n\to\infty}C_{f_n,C,N,\ell}=0$.
\end{lemma}

\begin{proof}
For each $g\in G$ we have $f(g\bigcdot)\in C^\infty(K/M)$. By a slight abuse of notation we will write $\tau$ also for the highest weight of $(\tau,Y_\tau)$. Applying \cite[ch.\@ V, Lemma 3.2]{GAGA} to $C^\infty(K/M)$ with the uniform norm $\norm{\bigcdot}_\infty$ and the left regular representation $\lambda$ we obtain
\begin{gather}\label{eq:gaga_inequality}
\forall\, Y_\tau\in\hat K_M,\,\forall m\in\N\colon\quad\norm{\pi_{Y_\tau}^*(f)(g)}_\infty\leq C_1c_\tau^{-m}\dim(Y_\tau)^2\norm{\lambda(\Omega^m)f(g\bigcdot)}_\infty,
\end{gather}
where
\begin{enumerate}
\item $\Omega$ is a bi-invariant differential operator on $K$ with
\begin{gather*}
\Omega\chi_\tau=c_\tau\chi_\tau
\end{gather*}
for the character $\chi_\tau$ of $Y_\tau$ (cf.\@ proof of \cite[Thm.\@ V.3.1]{GAGA}),
\item $c_\tau\geq1+\langle\tau+\rho_{[\mf k,\mf k]},\tau+\rho_{[\mf k,\mf k]}\rangle-\langle\rho_{[\mf k,\mf k]},\rho_{[\mf k,\mf k]}\rangle=1+\langle\tau,\tau+2\rho_{[\mf k,\mf k]}\rangle$, where $\rho_{[\mf k,\mf k]}$ denotes the half-sum of positive roots in the semisimple part $[\mf k,\mf k]$ of $\mf k$, (see \cite[ch.\@ V, Eq.\@ (16) \& proof of Lemma 3.2]{GAGA})
\item $C_1>0$ is some constant independent of $f,C,N,\ell$ and $g$ given by the continuity of $\lambda$ on $C^\infty(K/M)$.  
\end{enumerate}
By the Weyl dimension formula we have
\begin{gather*}
\dim(Y_\tau)=\prod_{\alpha\in\Delta_{[\mf k,\mf k]}^+}\frac{\langle\tau+\rho_{[\mf k,\mf k]},\alpha\rangle}{\langle\rho_{[\mf k,\mf k]},\alpha\rangle},
\end{gather*}
where $\Delta_{[\mf k,\mf k]}^+$ denotes the positive roots in $[\mf k,\mf k]$. Therefore we can conclude that there exists a constant $\tilde C$ depending only on $\mf k$ such that, for $m\geq m_N\in\N$ large enough,
\begin{gather*}
c_\tau^{-m}\dim(Y_\tau)^2\leq\tilde C\cdot (1+\norm{\tau}^2)^{-N}
\end{gather*}
and thus by Equation \eqref{eq:gaga_inequality}
\begin{gather*}
\forall\, Y_\tau\in\hat K_M\colon\quad\norm{\pi_{Y_\tau}^*(f)(g)}_\infty\leq C_1\tilde C\norm{\lambda(\Omega^{m_N})f(g\bigcdot)}_\infty\cdot (1+\norm{\tau}^2)^{-N}.
\end{gather*}
Taking the supremum over $C$ on both sides we hence infer
\begin{gather*}
\forall\, Y_\tau\in\hat K_M\colon\quad\sup_{g\in C}\norm{\pi_{Y_\tau}^*(f)(g)}_\infty\leq C_1\tilde C\sup_{g\in C}\norm{\lambda(\Omega^{m_N})f(g\bigcdot)}_\infty\cdot (1+\norm{\tau}^2)^{-N}.
\end{gather*}
Note that since the map $g\mapsto\norm{\lambda(\Omega^{m_N})f(g\bigcdot)}_\infty$ from $G$ to $\R_{\geq0}$ is continuous by the smoothness of $f$, the suprema are actually finite. We abbreviate
\begin{gather*}
C_{f,C,N,0}\coloneqq C_1\tilde C\sup_{g\in C}\norm{\lambda(\Omega^{m_N})f(g\bigcdot)}_\infty<\infty.
\end{gather*}
Note that the procedure above also works for $X_1\cdots X_k f$ instead of $f$ for $X_1,\ldots,X_k\in\mc B$ and $0\leq k\leq\ell$. We set
\begin{gather*}
C_{f,C,N,\ell}\coloneqq\max\{C_{\varphi,C,N,0}\mid\exists\, 0\leq k\leq\ell,\,\exists X_1,\ldots,X_k\in\mc B\colon\varphi=X_1\cdots X_kf\}.
\end{gather*}
By the definition of $\pi_{Y_\tau}^*$ we have $\pi_{Y_\tau}^*(X_1\cdots X_kf)=X_1\cdots X_k\pi_{Y_\tau}^*(f)$ for all $X_1,\ldots,X_k$ as above. Finally we obtain that for each $Y_\tau\in\hat K_M$
\begin{align*}
&\hphantom{\leq\ }\sup_{g\in C}\abs{(X_1\cdots X_k\pi_{Y_\tau}(\pi_{Y_\tau}^*(f)))(g)}=
\sup_{g\in C}\abs{(\pi_{Y_\tau}^*(X_1\cdots X_kf))(g)(e)}\\
&\leq\sup_{g\in C}\norm{\pi_{Y_\tau}^*(X_1\cdots X_kf)(g)}_\infty\leq C_{f,C,N,\ell}\cdot (1+\norm{\tau}^2)^{-N}.
\end{align*}
This proves the first part and the second part follows from the definition of $C_{f,C,N,\ell}$.
\end{proof}

\begin{lemma}\label{la:fourier_C_c}
Let $f\in C_c^\infty(G/M)$. Then
\begin{gather}\label{eq:la_series}
\sum_{\tau\in\hat K_M}\pi_{Y_\tau}(\pi_{Y_\tau}^*(f))
\end{gather}
is absolutely convergent with respect to each $\sn{\ell}{C}{\bigcdot}$ and converges to $f$ in $C_c^\infty(G/M)$.
\end{lemma}

\begin{proof}
Let $\on{pr}\colon G\to G/M$ denote the canonical projection. By the definition of the inductive limit topology on $C_c^\infty(G/M)$ we have to find a compact set $C\subset G/M$ such that $\on{supp}(\pi_{Y_\tau}(\pi_{Y_\tau}^*(f)))\subseteq C$ for each $Y_\tau\in\hat K_M$ and such that for each $\ell\in\N_0$ we have that $\sum_{\tau\in\hat K_M}\pi_{Y_\tau}(\pi_{Y_\tau}^*(f))$ converges to $f$ with respect to $\norm{\bigcdot}_{H^\ell(\on{pr}^{-1}(C))}$. As in the proof of Lemma \ref{la:pi_gamma}\,\ref{it:la_pi_gamma2} we see that the condition on the supports is fulfilled if we choose $C\coloneqq\on{supp}(f)\cdot K$. Let $\ell\in\N_0$ and $N\in\N$ be fixed. By Lemma \ref{la:fourier_conv} there exists a constant $C_{f,C,N,\ell}$ independent of $Y_\tau$ such that
\begin{gather*}
\forall\, Y_\tau\in\hat K_M\colon\sn{\ell}{C}{\pi_{Y_\tau}(\pi_{Y_\tau}^*(f))}\leq C_{f,C,N,\ell}\cdot (1+\norm{\tau}^2)^{-N}.
\end{gather*}
Thus we have for each finite subset $F\subseteq\hat K_M$ that
\begin{align}
\norm{\sum_{\tau\in F}\pi_{Y_\tau}(\pi_{Y_\tau}^*(f))}_{H^\ell(\on{pr}^{-1}(C))}\leq\sum_{\tau\in F}\norm{\pi_{Y_\tau}(\pi_{Y_\tau}^*(f))}_{H^\ell(\on{pr}^{-1}(C))}\leq C_{f,C,N,\ell}\sum_{\tau\in F}(1+\norm{\tau}^2)^{-N}.\label{eq:pr_conv}
\end{align}
Let $\varepsilon>0$. Note that the weight lattice of $[\mf k,\mf k]$ is a lattice in the finite dimensional space $(i\mf t_0)^*$, where $\mf t_0$ denotes the Lie algebra of a maximal torus $T$ in $\tilde K$, the analytic subgroup of $[\mf k_0,\mf k_0]$. Therefore, we may identify $\hat K_M$ with a subset of $\Z^d$ in $\R^d$ with $d\coloneqq\dim\mf t_0$. We infer that if $N$ is large enough, there exists a finite set $F_0\subseteq\hat K_M$ such that the right hand side of \eqref{eq:pr_conv} is smaller than $\varepsilon$ for each finite set $F\subseteq\hat K_M$ with $F\cap F_0=\emptyset$. Therefore, for each such $F$,
\begin{gather*}
\norm{\sum_{\tau\in F}\pi_{Y_\tau}(\pi_{Y_\tau}^*(f))}_{H^\ell(\on{pr}^{-1}(C))}\leq\sum_{\tau\in F}\norm{\pi_{Y_\tau}(\pi_{Y_\tau}^*(f))}_{H^\ell(\on{pr}^{-1}(C))}\leq C_{f,C,N,\ell}\cdot\varepsilon.
\end{gather*}
Hence, the series in \eqref{eq:la_series} converges absolutely and to its pointwise limit $f$ (see Lemma \ref{la:pi_gamma}\,\ref{it:la_pi_gamma3}) with respect to $\norm{\bigcdot}_{H^\ell(\on{pr}^{-1}(C))}$. 
\end{proof}

We can also decompose distributions.
\begin{lemma}\label{la:distribution_decomp}
Let $u\in\mc D'(G/M)$ be a distribution. Then the sum
\begin{gather*}
\fourier{u}
\end{gather*}
converges absolutely and to $u$ in the weak sense.
\end{lemma}

\begin{proof}
Let $f\in C_c^\infty(G/M)$. For each $Y_\tau\in\hat K_M$ we have (see Definition \ref{def:pi_gamma} and Notation \ref{not:pi_gamma}) 
\begin{gather*}
\pi_{Y_\tau}(\pi_{Y_\tau}^*(u))(f)=\pi_{Y_\tau}^*(u)(\pi_{\tilde Y}^*(f))=u(\pi_{\tilde Y}(\pi_{\tilde Y}^*(f)))
\end{gather*}
and therefore, by Lemma \ref{la:fourier_C_c} and the continuity of $u$,
\begin{gather*}
\sum_{\tau\in\hat K_M}\pi_{Y_\tau}(\pi_{Y_\tau}^*(f))=f\text{ in }C_c^\infty(G/M)\Rightarrow\sum_{\tau\in\hat K_M}u(\pi_{\tilde Y}(\pi_{\tilde Y}^*(f)))=u(f).
\end{gather*}
For the absolute convergence note that (see \cite[Def.\@ 2.1.1]{H90}) $u$ restricted to $C^\infty(\on{supp}(f)K)$ is of finite order, i.e.\@ there exist $\ell\in\N_0$ and $C>0$ with
\begin{gather*}
\forall \varphi\in C^\infty(\on{supp}(f)K)\colon\qquad\abs{u(\varphi)}\leq C\norm{\varphi}_{H^\ell(\on{supp}(f)K)}.
\end{gather*}
Then 
\begin{gather*}
\abs{\pi_{Y_\tau}(\pi_{Y_\tau}^*(u)(f))}=\abs{u(\pi_{\tilde Y}(\pi_{\tilde Y}^*(f)))}\leq C\norm{\pi_{\tilde Y}(\pi_{\tilde Y}^*(f))}_{H^\ell(\on{supp}(f)K)}.
\end{gather*}
The absolute convergence now follows from Lemma \ref{la:fourier_conv}.
\end{proof}

\begin{lemma}\label{la:def_distr}
Fix $c>0$ and $N\in\N$. If $\psi_\tau\in C^\infty(G\times_KY_\tau)$ for $\tau\in\hat K_M$ are chosen such that	
\begin{gather*}
		\iota_{G/M}(\pi_{Y_\tau}(\psi_{\tau}))(\pi_{\tilde Y_\tau}(\overline{\psi_{\tau}}) )\leq c\cdot(1+\norm{\tau}^2)^N,
\end{gather*}
then $\psi\coloneqq\sum_{\tau\in\hat K_M}\iota_{G/M}(\pi_{Y_\tau}(\psi_\tau))$ is absolutely convergent in the weak sense and defines a distribution on $G/M$.
\end{lemma}

\begin{proof}
We first prove the pointwise convergence of $\psi$ on $C_c^\infty(G/M)$. For each test function $f\in C_c^\infty(G/M)$ we have by Lemma \ref{la:pi_gamma}\,\ref{it:la_pi_gamma4}, Notation \ref{not:pi_gamma} and Definition \ref{def:pi_gamma}
\begin{align*}
\iota_{G/M}(\pi_{Y_\tau}(\psi_{\tau}))(f)=\pi_{Y_\tau}(\iota_{Y_\tau}(\psi_{\tau}))(f)=\iota_{Y_\tau}(\psi_\tau)(\pi_{\tilde Y_\tau}^*(f))
=\int_G\pi_{Y_\tau}(\psi_\tau)(g)\pi_{\tilde Y_\tau}(\pi_{\tilde Y_\tau}^*(f))(g)\intd g.
\end{align*}
The Cauchy-Schwarz inequality thus implies that
\begin{gather*}
\abs{\iota_{G/M}(\pi_{Y_\tau}(\psi_{\tau}))(f)}^2\leq\int_G\abs{\pi_{Y_\tau}(\psi_\tau)(g)}^2\intd g\cdot\int_G\abs{\pi_{\tilde Y_\tau}(\pi_{\tilde Y_\tau}^*(f))(g)}^2\intd g.
\end{gather*}
For the first factor we obtain
\begin{gather*}
\int_G\abs{\pi_{Y_\tau}(\psi_\tau)(g)}^2\intd g=\iota_{G/M}(\pi_{Y_\tau}(\psi_{\tau}))(\pi_{\tilde Y_\tau}(\overline{\psi_{\tau}}) )\leq c\cdot(1+\norm{\tau}^2)^N.	
\end{gather*}
For the second factor Lemma \ref{la:fourier_conv} implies that for each $m\in\N$ there exists a constant $\tilde C\coloneqq C_{\varphi,\on{pr}^{-1}(\on{supp}(f))K,m,0}$ independent of $Y_\tau$ such that
\begin{gather*}
\forall\, Y_\tau\in\hat K_M\colon\sn{0}{\on{pr}^{-1}(\on{supp}(f))K}{\pi_{Y_\tau}(\pi_{Y_\tau}^*(f))}\leq \tilde C\cdot (1+\norm{\tau}^2)^{-m}.
\end{gather*}
Choosing $m$ sufficiently large we thus obtain that
\begin{gather*}
\sum_{\tau\in\hat K_M}\abs{\iota_{G/M}(\pi_{Y_\tau}(\psi_{\tau}))(f)}<\infty
\end{gather*}
converges absolutely. We now prove the continuity of $\psi$. Let $C\subset G$ be a compact set and $(f_n)_{n\in\N}$ be a sequence of functions $f_n\in C_c^\infty(G/M)$ such that $\on{supp}(f_n)\subseteq CM$ for each $n\in\N$ and $\sn{\ell}{CM}{f_n}$ converges to $0$ for each fixed $\ell\in\N_0$. We have to prove that $\psi(f_n)\to0$ (see \cite[Thm.\@ 2.1.4]{H90}). Again by Lemma \ref{la:fourier_conv} we may choose for each $m\in\N$ constants $\tilde C_n$ independent of $Y_\tau$ such that
\begin{gather*}
\forall\, Y_\tau\in\hat K_M\colon\sn{0}{CM}{\pi_{Y_\tau}(\pi_{Y_\tau}^*(f_n))}\leq \tilde C_n\cdot (1+\norm{\tau}^2)^{-m}.
\end{gather*}
Moreover, by the second part of Lemma \ref{la:fourier_conv} we may choose the constants $\tilde C_n$ such that $\lim_{n\to\infty}\tilde C_n=0$. Proceeding as above we arrive at
\begin{gather*}
\sum_{\tau\in\hat K_M}\abs{\iota_{G/M}(\pi_{Y_\tau}(\psi_{\tau}))(f_n)}\leq\sqrt{c\cdot\tilde C_n}\sum_{\tau\in\hat K_M}(1+\norm{\tau}^2)^{\frac{N-m}{2}}\to 0,
\end{gather*}
since the series on the right hand side converges for $m$ large enough. 
\end{proof}

\subsection{Tensor product decompositions}
In this section we prove a number of technical results on the $K$-type decomposition of $Y\otimes\mf p$ for $Y\in\hat K$. Some of the calculations have to be done case by case. Those calculations we put into Appendix \ref{app:comp_scalars} to make the arguments presented in this subsection more transparent. 
\begin{notation}
For $V,\ Y\in\hat K$ we write 
\begin{gather*}
V\leftrightarrow Y\vcentcolon\Leftrightarrow V\leq Y\otimes\mf p\Leftrightarrow Y\leq V\otimes\mf p,
\end{gather*}
where the second equivalence follows from \cite[Remark 2.8]{BransonOlafsson}.
\end{notation}

\begin{definition}\label{def:omega}
Define the $K$-equivariant map 
\begin{gather*}
\omega:\mf p\to C^\infty(K/M),\quad \omega(X)(kM)\coloneqq\langle\Ad(k^{-1})X,H\rangle,
\end{gather*}
where $H\in\mf a_0$ is defined on page \pageref{def:H}. Note that $\omega(H)(eM)=1$. For each $Y\in\hat K_M$ we further define the $K$-equivariant map
\begin{gather*}
\omega_Y:Y\otimes\mf p\to C^\infty(K/M),\quad\omega_Y(\varphi\otimes X)\coloneqq\omega(X)\varphi.
\end{gather*}
Let $V\in\hat K$ with $V\lra Y$. We write
\begin{gather*}
V\lraomega Y\vcentcolon\Leftrightarrow V\leq\omega_Y(Y\otimes\mf p)
\end{gather*}
Note that $V\lraomega Y$ implies $V\in\hat K_M$ since the image of $\omega_Y$ is contained in $C^\infty(K/M)$. By \cite[Lemma 4.4\,(c)]{BransonOlafsson} we have
\begin{gather*}
V\lraomega Y\Leftrightarrow Y\lraomega V.
\end{gather*}
 We realize $V\leq L^2(K/M)$ and define $T_V^Y\in\Hom{K}{Y\otimes\mf p^*}{V}$ by
\begin{gather*}
T_V^Y:Y\otimes\mf p^*\to V,\quad T_V^Y(\varphi\otimes\psi)\coloneqq\on{pr}_V(\omega_Y(\varphi\otimes\I^{-1}(\psi))),
\end{gather*}
where $\on{pr}_V$ denotes the orthogonal projection
\begin{gather*}
\on{pr}_V:L^2(K/M)\cong\widehat\bigoplus_{W\in\hat K_M}W\to V.
\end{gather*}
If $V\lra Y$ but not $V\lraomega Y$ we define
\begin{gather*}
T_V^Y:Y\otimes\mf p^*\to V,\quad T_V^Y\coloneqq\on{pr}_V\circ\,(\on{id}_Y\otimes\,\I^{-1}),
\end{gather*}
with the orthogonal projection $\on{pr}_V:Y\otimes\mf p\to V$. Since the tensor product decomposes multiplicity-freely by Proposition \ref{prop:mult_one}, there exist uniquely determined homomorphisms $\iota_Y^V\in\Hom{K}{V}{Y\otimes\mf p^*}$ such that 
\begin{gather*}
T_V^Y\circ\iota_Y^V=\on{id}_V\text{ and }T_V^Y\circ\iota_Y^W=0
\end{gather*}
for each $W\leftrightarrow Y$ with $V\not\cong W$. In Proposition \ref{prop:iota_gen} we give an explicit formula for $\iota_Y^V$ in the case $V\in\hat K_M$.
\end{definition}

\begin{remark}\label{rem:sum_of_projections}
By definition we have for each $Y\in\hat K_M$
\begin{gather*}
\sum_{V\lraomega Y}T_V^Y=\omega_Y\circ(\on{id}_Y\otimes\,\I^{-1}).
\end{gather*}
\end{remark}

In the following we will describe the embeddings $\iota_Y^V$ from Definition \ref{def:omega} in more detail.

\begin{lemma}\label{la:lambdaVY}
Let $Y,\, V\in\hat K_M$ with $V\lra Y$. Then the operator
\begin{gather*}
\Phi:V\to Y\otimes\mf p^*,\quad\Phi(f)\coloneqq\sum_{j=1}^{\dim\mf p}\on{pr}_Y(\omega(X_j)f)\otimes\I(\tilde X_j)
\end{gather*}
is independent of the basis and $K$-equivariant. Moreover, the map
\begin{gather*}
V\to V,\quad f\mapsto\sum_{j=1}^{\dim\mf p}\on{pr}_V(\omega(\tilde X_j)\on{pr}_Y(\omega(X_j)f))
\end{gather*}
is a multiple of the identity. We denote the corresponding scalar by $\lambda(V,Y)$.
\end{lemma}

\begin{proof}
Let $k\in K$ and consider $Y\otimes\mf p^*$ as $\Hom{}{\mf p}{Y}$ by
\begin{gather*}
Y\otimes\mf p^*\cong\Hom{}{\mf p}{Y},\quad f\otimes\lambda\mapsto(X\mapsto\lambda(X)f).
\end{gather*}
Then, for $f\in V$,
\begin{gather*}
\Phi(k.f)(X_i)=\sum_{j=1}^{\dim\mf p}\on{pr}_Y(\omega(X_j)(k.f))\I(\tilde X_j)(X_i)=\on{pr}_Y(\omega(X_i)(k.f)).
\end{gather*}
By linearity we obtain $\Phi(k.f)(X)=\on{pr}_Y(\omega(X)(k.f))$ for each $X\in\mf p$. Note that this expression and thus $\Phi$ is independent of the basis. On the other hand, note that
\begin{gather*}
k.\Phi(f)=\sum_{j=1}^{\dim\mf p}k.\on{pr}_Y(\omega(X_j)f)\otimes\on{Ad}^*(k)\I(\tilde X_j)
\end{gather*}
and thus
\begin{gather*}
(k.\Phi(f))(\on{Ad}(k)X_i)=k.\on{pr}_Y(\omega(X_i)f)=\on{pr}_Y((k.\omega(X_i))(k.f))=\on{pr}_Y(\omega(\on{Ad}(k)X_i)(k.f)).
\end{gather*}
Since $\on{Ad}(k)X_1,\ldots,\on{Ad}(k)X_{\dim\mf p}$ is a basis of $\mf p$ we have $(k.\Phi(f))(X)=\on{pr}_Y(\omega(X)(k.f))$ for each $X\in\mf p$. This proves $\Phi(k.f)=k.\Phi(f)$ and thus the first part of the lemma. From Definition \ref{def:omega} we recall that
\begin{gather*}
\Psi\coloneqq\on{pr}_V\circ\,\omega_Y\circ(\on{id}_Y\otimes\I^{-1})\colon Y\otimes\mf p^*\to V
\end{gather*}
is $K$-equivariant. The map in the lemma is given by the composition $\Psi\circ\Phi$. It is scalar by Schur's lemma. 
\end{proof}

The scalar $\lambda(V,Y)$ has the following properties.
\begin{proposition}[cf. {\cite[Lemma 4.4, Theorem 4.6]{BransonOlafsson}}]\label{prop:lambdaVY_BO}
Let $V,\, Y\in\hat K_M$ such that $V\lra Y$. Then
\begin{enumerate}
\item\label{it:lambdaVY_1} $\lambda(V,Y)\geq0$,
\item\label{it:lambdaVY_2} $V\lraomega Y\Leftrightarrow\lambda(V,Y)\neq0\Leftrightarrow\lambda(Y,V)\neq0$,
\item\label{it:lambdaVY_3} $\sum_{W\lraomega Y}\lambda(Y,W)=1$,
\item\label{it:lambdaVY_4} $\lambda(V,Y)\dim V=\lambda(Y,V)\dim Y$.
\end{enumerate}
\end{proposition}

\begin{proposition}\label{prop:iota_gen}
Let $Y,\, V\in\hat K_M$ with $V\lraomega Y$. Then we have for each $f\in V$
\begin{gather}\label{eq:iota_gen}
\iota_Y^V(f)=\frac{1}{\lambda(V,Y)}\sum_{j=1}^{\dim\mf p}\on{pr}_Y(\omega(X_j)f)\otimes\I(\tilde X_j).
\end{gather}
\end{proposition}

\begin{proof}
By Lemma \ref{la:lambdaVY} we know that the rand hand side of \eqref{eq:iota_gen} is $K$-equivariant as a function in $f$. The scalar $\lambda(V,Y)$ is non-zero by Proposition \ref{prop:lambdaVY_BO}. For each $W\in\hat K$ with $W\lra Y$ and $V\not\cong W$, the map $T_W^Y\circ\iota_Y^V$ is an intertwiner between $V$ and $W$ and thus zero by Schur's lemma. The normalization by $\lambda(V,Y)$ ensures that $T_V^Y\circ\iota_Y^V$ is the identity on $V$. This finishes the proof since we have multiplicity one by Proposition \ref{prop:mult_one}.
\end{proof}

The following lemma gives a method to calculate the scalars $\lambda(V,Y)$ (see Appendix \ref{app:comp_scalars}).
\begin{lemma}\label{la:lambda_explicit}
The scalar $\lambda(V,Y)$ from Lemma \ref{la:lambdaVY} is given by
\begin{gather*}
\lambda(V,Y)=\on{pr}_Y(\omega(H)\phi_V)(eM).
\end{gather*}
\end{lemma}

\begin{proof}
If $H=X_1,\ldots,X_{\dim\mf p}$ is as in Lemma \ref{la:gen_diffops} and $H=\tilde X_1,\ldots,\tilde X_{\dim p}$ its dual basis (see Notation \ref{not:inn_prod_dual_basis}) we may write, for each $f\in V$,
\begin{gather}\label{eq:proof_base_decomp}
\iota_Y^V(f)=\sum_{j=1}^{\dim\mf p}f_j\otimes\I(\tilde X_j)\in Y\otimes\mf p^*
\end{gather}
for some $f_1,\ldots,f_{\dim\mf p}\in Y$.
In particular, we have $\iota_Y^V(f)(H)(eM)=f_1(eM)$ by considering $\iota_Y^V(f)$ as an element of $\Hom{}{\mf p}{Y}$. By Definition \ref{def:omega} and Remark \ref{rem:sum_of_projections} we infer
\begin{gather*}
f=\sum_{W\lraomega Y}T_W^Y(\iota_Y^V(f))=\omega_Y((\on{id}_Y\otimes\,\I^{-1})(\iota_Y^V(f)))=\sum_{j=1}^{\dim\mf p}\omega_Y(f_j\otimes\tilde X_j)=\sum_{j=1}^{\dim\mf p}\omega(\tilde X_j)f_j.
\end{gather*}
Note that, since $X_j\in\mf k\oplus\mf n$ for $j=2,\ldots,\dim\mf p$ and $X_1\in\mf a$, the orthogonality of $\mf a$ and $\mf k\oplus\mf n$ with respect to $\langle\cdot,\cdot\rangle$ implies 
$\omega(\tilde X_j)(eM)=\langle\tilde X_j,H\rangle=0$ for each $j=2,\ldots,\dim\mf p$ and therefore
\begin{gather*}
f(eM)=\sum_{j=1}^{\dim\mf p}\omega(\tilde X_j)(eM)f_j(eM)=f_1(eM)=\iota_Y^V(f)(H)(eM).
\end{gather*}
In particular, we have for $f=\phi_V$
\begin{gather*}
\iota_Y^V(\phi_V)(H)(eM)=\phi_V(eM)=1.
\end{gather*}
On the other hand, Proposition \ref{prop:iota_gen} shows that 
\begin{gather*}
\iota_Y^V(\phi_V)(H)(eM)=\frac{1}{\lambda(V,Y)}\on{pr}_Y(\omega(H)\phi_V)(eM).\qedhere
\end{gather*}
\end{proof}

Note that, in the situation of Lemma \ref{la:gen_diffops}, we have for $V,Y\in\hat K_M$ with $V\lraomega Y$ that 
\begin{gather}\label{eq:T_decomp}
T_Y^V(p_{V,\mu})(e)=(\mu+\rhoa)(H)\lambda(V,Y)+\nu(V,Y)\text{ with }\nu(V,Y)\coloneqq T_Y^V(p_{V,-\rhoa})(e).
\end{gather}
The following lemma allows us to compute the scalars $T_Y^V(p_{V,\mu})(e)$ from Lemma \ref{la:gen_diffops} explicitly in all the rank one cases (see Appendix \ref{app:comp_scalars}).

\begin{lemma}\label{la:nu_comp_series}
Let $V,\, Y\in\hat K_M$ such that $V\lraomega Y$. If $\{0\}\neq U\leq H^\mu$ is a closed $G$-invariant subspace such that $\on{mult}_K(Y,U)\neq 0$ and $\on{mult}_K(V,U)=0$ we have $T_Y^V(p_{V,\mu})(e)=0$ and thus 
\begin{gather*}
\nu(V,Y)=-(\mu+\rhoa)(H)\lambda(V,Y).
\end{gather*}
Moreover, for $V\in\hat K_M$ with $V\lraomega\C$ we have
\begin{gather*}
T_{\C}^V(p_{V,\mu})(e)=0\Leftrightarrow\mu(H)=\rhoa(H).
\end{gather*}
\end{lemma}

\begin{proof}
Let $0\neq f\in Y\leq U$. Then, by Equation \eqref{eq:Poisson_proj}, we have $P_\mu^Y(f)(e)=\frac{1}{\dim Y}\on{pr}_V(f)\neq 0$. On the other hand Proposition \ref{prop:PoissonInjective} implies that $P_\mu^V(f)=0$. Therefore,
\begin{gather*}
0=\mathrm{d}_Y^V(P_\mu^V(f))(e)=T_Y^V(p_{V,\mu})(e)P_\mu^Y(f)(e)
\end{gather*}
implies that $T_Y^V(p_{V,\mu})(e)=0$. For $\mu(H)=\rhoa(H)$ we have that the constant functions form an invariant subspace, proving one direction. For the equivalence note that for each $V\in\hat K_M$ with $V\lraomega\C$, $T_{\C}^V(p_{V,\mu})(e)=\nu(V,\C)+(\mu+\rhoa)(H)\lambda(V,\C)$ is an affine map in $\mu(H)$ with $\lambda(V,\C)\neq0$ (by Proposition \ref{prop:lambdaVY_BO}.\ref{it:lambdaVY_2}).
\end{proof}

We have the following multiplicity one result.
\begin{proposition}\label{prop:mult_one}
Let $Y\in\hat K$. Then $Y\otimes\mf p^*$ decomposes multiplicity-freely.
\end{proposition}

\begin{proof}
By \cite[Ch. IX.8, Problem 15]{knapplie} it suffices to prove that all weights of $\mf p\cong_K\mf p^*$ have multiplicity one, i.e.\@ if $\mf t_0\leq\mf k_0$ is a maximal torus we have that $\mf t$ acts multiplicity-freely on $\mf p$. 

Let us first assume that the ranks $\on{rk}\mf k_0$ and $\on{rk}\mf g_0$ coincide. Then $\mf t\leq\mf k\leq\mf g$ is a Cartan subalgebra of $\mf g$ and we have the root-space decomposition
\begin{gather*}
\mf g=\mf t\oplus\bigoplus_{\alpha\in\Delta(\mf g,\mf t)}\mf g_{\alpha},
\end{gather*}
where each $\mf g_\alpha$ is one-dimensional. We note that the root spaces $\mf g_\alpha$ are invariant under the ($\C$-linear continuation of the) Cartan involution $\theta$; indeed we have for each $X\in\mf g_\alpha$
\begin{gather*}
\forall H\in\mf t\colon\quad[H,\theta X]=\theta[\theta H,X]=\theta[H,X]=\alpha(H)\theta X\qquad\Rightarrow\qquad \theta X\in\mf g_\alpha.
\end{gather*}
Therefore, writing $X=\frac{X+\theta X}{2}+\frac{X-\theta X}{2}$, we obtain $\mf g_\alpha=(\mf k\cap\mf g_\alpha)\oplus(\mf p\cap\mf g_\alpha)$ and thus
\begin{gather*}
\mf p=\bigoplus_{\alpha\in\Delta(\mf g,\mf t)}(\mf p\cap\mf g_\alpha).
\end{gather*}
Since $\dim_\C(\mf p\cap\mf g_\alpha)\in\{0,1\}$ we see that $\mf t$ acts multiplicity-freely on $\mf p$.

Let us now consider the case $\on{rk}\mf k_0<\on{rk}\mf g_0$. By \cite[Prop.\@ 6.60]{knapplie} the centralizer $\mf h_0\coloneqq Z_{\mf g_0}(\mf t_0)=\mf t_0\oplus Z_{\mf p_0}(\mf t_0)$ is a $\theta$-stable Cartan subalgebra of $\mf g_0$. Our real rank one assumption shows that 
$\mf a_0\coloneqq Z_{\mf p_0}(\mf t_0)$ is one-dimensional. For $\alpha\in\Delta$ we first note that
\begin{gather*}
X\in\mf g_\alpha\qquad\Rightarrow\qquad\theta X\in\mf g_{\theta\alpha},
\end{gather*}
where we define $(\theta\alpha)(H)\coloneqq\alpha(\theta H)$. Thus, $\mf g_\alpha+\mf g_{\theta\alpha}$ is $\theta$-stable and decomposes into a $\mf k$- and $\mf p$-part.

We claim that if $\alpha,\:\alpha'\in\Delta$ are two roots with $\restr{\alpha}{\mf t}=\restr{\alpha'}{\mf t}$, then $\alpha'=\alpha$ or $\alpha'=\theta\alpha$. If this is true we obtain the result as follows. Let $\beta\in\mf t^*$. For $\beta=0$ the weight space of $\beta$ in $\mf p$ is given by $\mf a$, which is one-dimensional. For $\beta\neq0$ the weight space of $\beta$ in $\mf p$ is given by
\begin{gather*}
\sum_{\substack{\alpha\in\Delta\\\alpha|_{\mf t}=\beta}}\pi(\mf g_\alpha+\mf g_{\theta\alpha}),
\end{gather*}
where $\pi\colon \mf g\to\mf p,\quad X\mapsto\frac{X-\theta X}{2}$ denotes the projection onto $\mf p$. Then our claim implies that there are at most two roots $\alpha,\: \theta\alpha\in\Delta$ with $\restr{\alpha}{\mf t}=\restr{\theta\alpha}{\mf t}=\beta$. Therefore, the weight space of $\beta$ in $\mf p$ is given by the one-dimensional space $\pi(\mf g_\alpha+\mf g_{\theta\alpha})$.

Let us finally prove our claim in the rank one case. By the classification of real forms it suffices to consider the groups $\GSO{n}$ with $n=2p+1$ odd (recall that we are in the case $\on{rk}\mf k_0<\on{rk}\mf g_0$). In this case all roots have the same length and this implies our claim since every root $\alpha\in\Delta$ is determined by its restrictions to $\mf t$ and $\mf a$.
\end{proof}

\begin{remark}
Note that the proof of Proposition \ref{prop:mult_one} does not use our rank one assumption if $\on{rk}\mf g=\on{rk}\mf k$. In this case we can say more.
\end{remark}

\begin{proposition}\label{prop:tensor_decomp_eq_rk}
Let $\on{rk}\mf g=\on{rk}\mf k$ and $Y_\tau\in\hat K$ with highest weight $\tau$. Denote the non-compact roots by $\Delta_n$. Then the tensor product $Y_\tau\otimes\mf p^*$ decomposes into
\begin{gather*}
Y_\tau\otimes\mf p^*\cong\bigoplus_{\beta\in\Delta_n}m(\beta) Y_{\tau,\beta},
\end{gather*}
where the multiplicities $m(\beta)$ are at most $1$ and $Y_{\tau,\beta}$ has weight $\tau+\beta$. Moreover, we have
\begin{gather*}
m(\beta)=1\qquad\quad\Longrightarrow\qquad\quad \beta\in S,
\end{gather*}
with $S\coloneqq\{\beta\in\Delta_n\mid\tau+\beta\text{ dominant}\}\subseteq\Delta_n$.
\end{proposition}

\begin{proof}
First we note that $\mf p\cong_K\mf p^*$ by the Killing form. By \cite[Prop.\@ 9.72]{knapplie} the highest weight of each irreducible constituent of $Y_\tau\otimes\mf p$ is of the form $\tau+\beta$, where $\beta$ is a weight of $\mf p$, i.e.\@ $\beta\in\Delta_n$. Moreover each irreducible constituent occurs at most with multiplicity one by \cite[Ch.\@ IX.8, Problem 15]{knapplie} since the weight spaces of $\mf p$ have multiplicity one by the root space decomposition. Since the highest weight $\tau+\beta$ has to be dominant we can restrict the sum to the subset $S\subseteq\Delta_n$.
\end{proof}

\begin{proposition}\label{prop:gen_grad_omega}
Let $Y\in\hat K_M$ and $V\in\hat K$ with $V\lra Y$. Then, for each $\mu\in\mf a$,
\begin{gather*}
\mathrm{d}_V^Y\circ P_\mu^Y\neq0\Rightarrow V\lraomega Y.
\end{gather*}
\end{proposition}

\begin{proof}[Proof for $G\neq\GSO{3}$\protect\footnotemark]\footnotetext{For $G=\GSO{3}$ we have, for $k\in\N_0$, $Y_k\lra Y_k$ but $Y_k\notleftrightomega Y_k$ by Proposition \ref{prop:tensor_decomp_odd} and Lemma \ref{la:decomp_omega_R}. Realizing $Y_k$ explicitly as a subrepresentation of $Y_k\otimes\mf p^*$ one can prove that $\on{pr}_{Y_k}((\on{id}_Y\otimes\,\I^{-1})(p_{Y_k,\mu}))(e)=0$ for each $\mu\in\mf a$ and thus $\mathrm{d}_{Y_k}^{Y_k}\circ P_\mu^{Y_k}=0$ by Lemma \ref{la:gen_diffops}.\ref{it:gen_diffops2}. Thus, Proposition \ref{prop:gen_grad_omega} is also valid for $G=\GSO{3}$.}
By Lemma \ref{la:gen_diffops}.\ref{it:gen_diffops3} we see that $\mathrm{d}_V^Y\circ P_\mu^Y\neq0$ implies that $V\in\hat K_M$. Using Proposition \ref{prop:lambdaVY_BO}.\ref{it:lambdaVY_2}, Lemma \ref{la:lambda_explicit} and Lemma \ref{la:decomp_omega_R}, \ref{la:decomp_omega_C}, \ref{la:decomp_omega_H} resp.\@ \ref{la:decomp_omega_F} we infer that $V\lraomega Y$ if and only if $V\lra Y$ and $V\in\hat K_M$.
\end{proof}

\subsection{Computations for the Fourier characterization}
The aim of this subsection is proving the converse direction in Lemma \ref{la:gen_diffops}, i.e.\@ we want to prove that if the equations derived from Lemma \ref{la:gen_diffops} are satisfied for some distribution $f\in\mc D'(G/M)$ we already have $f\in\pstd{\mu}$. The precise result is given in Theorem \ref{thm:fourier_char}. It provides a technique to determine images for Poisson transforms. We start with the following reformulation of Lemma \ref{la:gen_diffops}.
\begin{lemma}\label{la:gen_diffops2}
Assume the setting from Lemma \ref{la:gen_diffops}. Then for each $f\in\pstd{\mu}$
\begin{enumerate}
\item\label{it:gen_diffops21} $(\mathrm{d}_V^Y\circ\pi_{Y}^*)(f)=T_V^Y(p_{Y,\mu})(e)\frac{\dim Y}{\dim V}\pi_{V}^*(f)$ if $V$ is $M$-spherical, i.e.\@ $V\leq L^2(K)^M$,
\item\label{it:gen_diffops22}$(\mathrm{d}_V^Y\circ\pi_{Y}^*)(f)=0$ if $V$ is not $M$-spherical, i.e.\@ $V^M=0$.
\end{enumerate}
\end{lemma}

\begin{proof}
This is a direct consequence of Lemma \ref{la:gen_diffops} and Lemma \ref{la:pi_gamma}\,\ref{it:la_pi_gamma5}.
\end{proof}

 We consider the $\mf a_0$- and $\mf n_0$-action separately and start with the first one.
\begin{lemma}\label{la:a_action_gen_gradients_general}
Let $\mu\in\mf a^*$ and $f=\fourier{f}\in\mc D'(G/M)$ (recall Lemma \ref{la:distribution_decomp}) with $\pi_{Y_\tau}^*(f)\in C^\infty(G\times_KY_\tau)$ such that the equations from Lemma \ref{la:gen_diffops2}\,\ref{it:gen_diffops21} and \ref{it:gen_diffops22} hold for $f$ for every irreducible constituent of $Y_\tau\otimes\mf p^*$ and every $Y_\tau\in\hat K_M$. Let $X\in\mf a_0$. For each $V,\,Y_\tau\in\hat K_M$ with $V\leftrightarrow Y_\tau$ we define
\begin{gather*}
f_{V,\tau,X}\in C^\infty(G/M),\qquad f_{V,\tau,X}(gM)\coloneqq\iota_V^{Y_{\tau}}(\pi_{Y_\tau}^*(f)(g))(X)(e).
\end{gather*}
Then, in the weak sense,
\begin{gather*}
r(X)f=\sum_{\tau\in\hat K_M}\sum_{\substack{V\lraomega Y_\tau\\V\in\hat K_M}}\frac{\dim V}{\dim Y_\tau}T_{Y_{\tau}}^V(p_{V,\mu})(e)f_{V,\tau,X},
\end{gather*}
where $r$ denotes the right regular representation of $\mf a_0$ on $\mc D'(G/M)$.
\end{lemma}

\begin{proof}
We first prove that $f_{V,\tau,X}\in C^\infty(G/M)$. For each $g\in G$ and $m\in M$ we have
\begin{gather*}
\iota_V^{Y_{\tau}}(\pi_{Y_\tau}^*(f)(gm))(X)(e)=\iota_V^{Y_{\tau}}(\tau(m^{-1})\pi_{Y_\tau}^*(f)(g))(X)(e)
\end{gather*}
since $\pi_{Y_\tau}^*(f)\in  C^\infty(G\times_KY_\tau)$. As $\iota_V^{Y_{\tau}}$ is $K$-equivariant we obtain
\begin{gather*}
\iota_V^{Y_{\tau}}(\tau(m^{-1})\pi_{Y_\tau}^*(f)(g))(X)(e)=\iota_V^{Y_{\tau}}(\pi_{Y_\tau}^*(f)(g))(m.X)(m)
\end{gather*}
which equals $\iota_V^{Y_{\tau}}(\pi_{Y_\tau}^*(f)(g))(X)(e)$, since $M$ acts trivially on $\mf a_0$ and each element of $V$ is right $M$-invariant.

For each $\varphi\in C_c^\infty(G/M)$ we have (denoting $\iota_{G/M}(f)(\varphi)$ by $\langle f,\varphi\rangle$)
\begin{gather*}
\langle r(X)f,\varphi\rangle=-\langle f,r(X)\varphi\rangle=-\sum_{\tau\in\hat K_M}\langle\pi_{Y_\tau}(\pi_{Y_\tau}^*(f)),r(X)\varphi\rangle=\sum_{\tau\in\hat K_M}\langle r(X)\pi_{Y_\tau}(\pi_{Y_\tau}^*(f)),\varphi\rangle.
\end{gather*}
In particular, by the absolute convergence from Lemma \ref{la:distribution_decomp}, we obtain that
\begin{gather*}
\sum_{\tau\in\hat K_M}r(X)\pi_{Y_\tau}(\pi_{Y_\tau}^*(f))
\end{gather*}
converges absolutely to $r(X)f$ in the weak sense. We will now compute the summands explicitly. Note first that for each $g\in G$
\begin{gather}\label{eq:proof_a_action_gen1}
(r(X)\pi_{Y_\tau}(\pi_{Y_\tau}^*(f)))(g)=\d\pi_{Y_\tau}^*(f)(g\exp tX)(e)=(((\nabla\circ\pi_{Y_\tau}^*(f))(g))(X))(e).
\end{gather}
We claim that
\begin{gather*}
(\nabla\circ\pi_{Y_\tau}^*(f))(g)=\sum_{V\leftrightarrow Y_\tau}(\iota^V_{Y_{\tau}}\circ T^{Y_{\tau}}_V)((\nabla\circ\pi_{Y_\tau}^*(f))(g)).
\end{gather*}
Indeed, both sides are elements of $Y_\tau\otimes\mf p^*$ and by Definition \ref{def:omega} they are equal if 
\begin{gather*}
T^{Y_\tau}_W((\nabla\circ\pi_{Y_\tau}^*(f))(g))=T^{Y_\tau}_W\left(\sum_{V\leftrightarrow Y_\tau}(\iota^V_{Y_{\tau}}\circ T^{Y_{\tau}}_V)((\nabla\circ\pi_{Y_\tau}^*(f))(g))\right)
\end{gather*}
for each irreducible subrepresentation $W$ with $W\leftrightarrow Y_\tau$. But this follows from the definition of the $\iota^V_{Y_\tau}$.

Note that, since $\mathrm{d}^{Y_{\tau}}_V=T^{Y_\tau}_V\circ\nabla$ (see Lemma \ref{la:gen_diffops}),
\begin{gather*}
\sum_{V\leftrightarrow Y_\tau}(\iota^V_{Y_{\tau}}\circ T^{Y_{\tau}}_V)((\nabla\circ\pi_{Y_\tau}^*(f))(g))=\sum_{V\leftrightarrow Y_\tau}\iota^V_{Y_{\tau}}(\mathrm{d}^{Y_{\tau}}_V(\pi_{Y_\tau}^*(f))(g)).
\end{gather*}
The equations from Lemma \ref{la:gen_diffops2} yield
\begin{align*}
(\nabla\circ\pi_{Y_\tau}^*(f))(g)&=\sum_{V\leftrightarrow Y_\tau}\iota^V_{Y_{\tau}}(\mathrm{d}^{Y_{\tau}}_V(\pi_{Y_\tau}^*(f))(g))\\
&=\sum_{\substack{V\leftrightarrow Y_\tau\\V\in\hat K_M}}\iota^V_{Y_{\tau}}\left(\frac{\dim Y_\tau}{\dim V}T_V^{Y_\tau}(p_{Y_\tau,\mu})(e)\pi_{V}^*(f)(g)\right)\\
&=\sum_{\substack{V\leftrightarrow Y_\tau\\V\in\hat K_M}}\frac{\dim Y_\tau}{\dim V}T_V^{Y_\tau}(p_{Y_\tau,\mu})(e)\iota^V_{Y_{\tau}}(\pi_{V}^*(f)(g)).
\end{align*}
By Proposition \ref{prop:gen_grad_omega} it suffices to sum over all $V\in\hat K_M$ with $V\lraomega Y_\tau$.
Using Equation \eqref{eq:proof_a_action_gen1} we thus obtain
\begin{align*}
(r(X)\pi_{Y_\tau}(\pi_{Y_\tau}^*(f)))(g)&=\sum_{\substack{V\lraomega Y_\tau\\V\in\hat K_M}}\frac{\dim Y_\tau}{\dim V}T_V^{Y_\tau}(p_{Y_\tau,\mu})(e)((\iota^V_{Y_{\tau}}(\pi_{V}^*(f)(g)))(X))(e)\\
&=\sum_{\substack{V\lraomega Y_\tau\\V\in\hat K_M}}\frac{\dim Y_\tau}{\dim V}T_V^{Y_\tau}(p_{Y_\tau,\mu})(e)f_{Y_\tau,V,X}(gM)
\end{align*}
and $r(X)f=\sum_{\tau\in\hat K_M}r(X)\pi_{Y_\tau}(\pi_{Y_\tau}^*(f))$ equals
\begin{gather*}
\sum_{\tau\in\hat K_M}\sum_{\substack{V\lraomega Y_\tau\\V\in\hat K_M}}\frac{\dim Y_\tau}{\dim V}T_V^{Y_\tau}(p_{Y_\tau,\mu})(e)f_{Y_\tau,V,X}=\sum_{V\in\hat K_M}\sum_{\substack{V\lraomega Y_\tau\\\tau\in\hat K_M}}\frac{\dim Y_\tau}{\dim V}T_V^{Y_\tau}(p_{Y_\tau,\mu})(e)f_{Y_\tau,V,X}.\qedhere
\end{gather*}
\end{proof}

In order to compute the sums occurring in the proof of Lemma \ref{la:a_action_gen_gradients_general} we write
\begin{gather}\label{eq:decomp_p}
p_{Y_\tau,\mu}=(\mu+\rhoa)(H)\phi_Y\otimes\I(H)+p_{Y_\tau,-\rhoa}.
\end{gather}
We first consider the contribution of first summand in this decomposition.
\begin{lemma}\label{la:p_first_part}
Let $Y\in\hat K_M,\ X\in\mf p$ and $\varphi\in Y$. Then
\begin{gather*}
\sum_{V\lraomega Y}\frac{\dim V}{\dim Y}T_{Y}^V(\phi_V\otimes\I(H))(e)\iota_V^{Y}(\varphi)(X)(e)=(\omega(X)\varphi)(e).
\end{gather*}
\end{lemma}

\begin{proof}
By Definition \ref{def:omega} and Lemma \ref{la:lambda_explicit} we have for each $V\in\hat K$ with $V\lraomega Y$
\begin{gather*}
T_{Y}^V(\phi_V\otimes\I(H))(e)=\on{pr}_Y(\omega(H)\phi_V)(e)=\lambda(V,Y).
\end{gather*}
Using Proposition \ref{prop:lambdaVY_BO}\:\ref{it:lambdaVY_4} and \ref{prop:iota_gen} we calculate
\begin{align*}
&\phantom{{}={}}\sum_{V\lraomega Y}\frac{\dim V}{\dim Y}T_{Y}^V(\phi_V\otimes\I(H))(e)\iota_V^{Y}(\varphi)(X)(e)\\
&=\sum_{V\lraomega Y}\frac{\lambda(Y,V)}{\lambda(V,Y)}\lambda(V,Y)\frac{1}{\lambda(Y,V)}\sum_{j=1}^{\dim\mf p}\on{pr}_V(\omega(X_j)\varphi)(e)\I(\tilde X_j)(X)\\
&=\sum_{V\lraomega Y}\sum_{j=1}^{\dim\mf p}\on{pr}_V(\omega(X_j)\varphi)(e)\I(\tilde X_j)(X)\\
&=\sum_{V\lraomega Y}\on{pr}_V(\omega(X)\varphi)(e)=(\omega(X)\varphi)(e).\qedhere
\end{align*}
\end{proof}

For the contribution of the second summand in \eqref{eq:decomp_p} we need some preparation. This is the content of the following three lemmas.
\begin{lemma}\label{la:2rho}
Let $\mf g_0$ be a semisimple Lie algebra, $B$ be some non-zero multiple of the Killing form $\kappa$ and $\theta$ be a Cartan involution. If $X_1,\ldots,X_{\dim(\mf p_0/\mf a_0)}$ is a basis of $\mf p_0\cap(\mf k_0\oplus\mf n_0)$ let $\tilde X_1,\ldots,\tilde X_{\dim(\mf p_0/\mf a_0)}$ denote the dual basis defined by $B(\tilde X_i,X_j)=\delta_{ij}$. Then $\sum_{j=1}^{\dim(\mf p_0/\mf a_0)}[\tilde X_j,\Ik{X_j}]\in\mf a_0$ and 
\begin{gather*}
\sum_{j=1}^{\dim(\mf p_0/\mf a_0)}B([\tilde X_j,\Ik{X_j}],H)=2\rhoa(H)\quad\forall H\in\mf a_0.
\end{gather*}
\end{lemma}

\begin{proof}[Proof of Lemma \ref{la:2rho}]
We first claim that $\sum_{j=1}^{\dim(\mf p_0/\mf a_0)}[\tilde X_j,\Ik{X_j}]\in\mf p_0$ is independent of the basis. Let $X_1',\ldots,X_{\dim(\mf p_0/\mf a_0)}'$ be another basis with base change matrix $(a_{ij})$, i.e.\@ $X_j'=\sum_ma_{mj}X_m$.
If $(b_{ij})$ denotes the inverse of $(a_{ij})$ we claim that $\tilde X_j'=\sum_\ell b_{j\ell}\tilde X_\ell$. Indeed,
\begin{align*}
B(\sum_\ell b_{j\ell}\tilde X_\ell, X_i')=B(\sum_\ell b_{j\ell}\tilde X_\ell, \sum_ma_{mi}X_m)=\sum_\ell\sum_m b_{j\ell}a_{mi}B(\tilde X_\ell, X_m)=\sum_m b_{jm}a_{mi}=\delta_{ij}.
\end{align*}
Thus,
\begin{align*}
\sum_j[\tilde X_j',\Ik{X_j'}]&=
\sum_j[\sum_\ell b_{j\ell}\tilde X_\ell,\Ik{\sum_ma_{mj}X_m}]=\sum_m\sum_\ell[\tilde X_\ell,\Ik{X_m}]\sum_ja_{mj}b_{j\ell}\\
&=\sum_m\sum_\ell[\tilde X_\ell,\Ik{X_m}]\delta_{m\ell}=\sum_m[\tilde X_m,\Ik{X_m}]
\end{align*}
is independent of the basis.\par
We will now construct a convenient basis of $\mf p_0\cap(\mf k_0\oplus\mf n_0)$. Let $\Sigma^+$ denote the set of positive restricted roots. We may assume that $B$ is a positive multiple of the Killing form (otherwise $-B$ is of this form and the signs of the $\tilde X_j$'s are flipped). For each $\lambda\in\Sigma^+$ we choose a basis $Y_1^\lambda,\ldots,Y_{\dim\mf g^\lambda}^\lambda$ of the restricted root space $\mf g^\lambda$ such that $B(Y_j^\lambda,\theta Y_k^\lambda)=-\frac12\delta_{jk}$ and define
\begin{gather*}
X_j^\lambda\coloneqq Y_j^\lambda-\theta Y_j^\lambda,\quad j\in\{1,\ldots,\dim\mf g^\lambda\}.
\end{gather*}
Note that, since
\begin{gather*}
B(X_j^\lambda,X_k^\mu)=-2B(Y_j^\lambda,\theta Y_k^\mu)=-2B(Y_j^\lambda,\theta Y_k^\mu)\delta_{\lambda\mu},
\end{gather*}
we have that the $X_j^\lambda$'s are orthonormal, i.e.\@ $\tilde X_j^\lambda=X_j^\lambda$.
By the restricted root space decomposition, every $X\in\mf p_0\cap(\mf k_0\oplus\mf n_0)$ is of the form $\sum_{\lambda\in\Sigma^+}X_\lambda-\theta X_\lambda$ for some $X_\lambda\in\mf g^\lambda$. Therefore, the $X_j^\lambda,\ \lambda\in\Sigma^+,$ form a basis of $\mf p_0\cap(\mf k_0\oplus\mf n_0)$. Note that
\begin{gather*}
X_j^\lambda=2Y_j^\lambda-(Y_j^\lambda+\theta Y_j^\lambda)\in\mf n_0\oplus\mf k_0\qquad\Longrightarrow\qquad\Ik{X_j^\lambda}=-(Y_j^\lambda+\theta Y_j^\lambda).
\end{gather*}
By the invariance of the Killing form we deduce for each $H\in\mf a_0$
\begin{alignat*}{2}
B([\tilde X_j^\lambda,\Ik{X_j^\lambda}],H)&=B(\tilde X_j^\lambda,[\Ik{X_j^\lambda},H])&&=B(\tilde X_j^\lambda,[H,Y_j^\lambda+\theta Y_j^\lambda])\\
&=B(\tilde X_j^\lambda,\lambda(H)(Y_j^\lambda-\theta Y_j^\lambda))&&=\lambda(H)B(\tilde X_j^\lambda,X_j^\lambda)=\lambda(H).
\end{alignat*}
Thus,
\begin{gather*}
\sum_{\lambda\in\Sigma^+}\sum_{j=1}^{\dim\mf g^\lambda}B([\tilde X_j^\lambda,\Ik{X_j^\lambda}],H)=\sum_{\lambda\in\Sigma^+}\lambda(H)\dim\mf g^\lambda=2\rhoa(H).
\end{gather*}
Moreover,
\begin{gather*}
[\tilde X_j^\lambda,\Ik{X_j^\lambda}]=[X_j^\lambda,\Ik{X_j^\lambda}]=[Y_j^\lambda-\theta Y_j^\lambda,-(Y_j^\lambda+\theta Y_j^\lambda)]=2[\theta Y_j^\lambda,Y_j^\lambda]\in\mf g^0\cap\mf p_0
\end{gather*}
implies that $[\tilde X_j^\lambda,\Ik{X_j^\lambda}]\in\mf a_0$ since $\mf g^0=\mf m_0\oplus\mf a_0$.
\end{proof}

\begin{lemma}\label{la:2rho_rank1}
Let $X_1,\ldots,X_{\dim\mf p}$ be as in Lemma \ref{la:gen_diffops}. Then $\sum_{j=2}^{\dim\mf p}\ell(\Ik{X_j})\omega(\tilde X_j)=-2\rhoa(H)\omega(H)$.
\end{lemma}

\begin{proof}
Since $\omega:\mf p\to C^\infty(K/M)$ is $K$-equivariant we have
\begin{gather*}
\sum_{j=2}^{\dim\mf p}\ell(\Ik{X_j})\omega(\tilde X_j)=\sum_{j=2}^{\dim\mf p}\omega([\Ik{X_j},\tilde X_j])
\end{gather*}
By Lemma \ref{la:2rho}, $\sum_{j=2}^{\dim\mf p}[\Ik{X_j},\tilde X_j]$ is an element of $\mf a_0$ and therefore a multiple of $H$. Let $\lambda\in\R$ denote this multiple. Then Lemma \ref{la:2rho} implies that
\begin{gather*}
\lambda=\langle\lambda H,H\rangle=\sum_{j=2}^{\dim\mf p}\langle[\Ik{X_j},\tilde X_j],H\rangle=-2\rhoa(H).\qedhere
\end{gather*}
\end{proof}

\begin{lemma}\label{la:p_second_part_local}
Let $Y\in\hat K_M$ and $X\in\mf p$. Then
\begin{gather*}
\sum_{V\lraomega Y}\frac{\dim V}{\dim Y}\overline{T_{Y}^V(p_{V,-\rhoa})(e)}\iota_Y^V(\phi_V)(X)=
\begin{cases}
-2\rhoa(H)\phi_Y&\colon X=H\\
\ell(\Ik{X})\phi_Y&\colon X\perp\mf a
\end{cases},
\end{gather*}
where the bar denotes complex conjugation.
\end{lemma}

\begin{proof}
For each $\psi\in L^2(K/M)$ we have by orthogonality and Proposition \ref{prop:helg_general}\,\ref{prop:helg_general_ii},
\begin{gather}\label{eq:proof_proj_at_e}
\on{pr}_Y(\psi)(e)=\langle\on{pr}_{Y}(\psi),\frac{\phi_Y}{\langle\phi_Y,\phi_Y\rangle_{L^2(K)}}\rangle_{L^2(K)}=\langle\psi,\frac{\phi_Y}{\langle\phi_Y,\phi_Y\rangle_{L^2(K)}}\rangle_{L^2(K)}.
\end{gather}
Therefore, since $\omega(X),\ X\in\mf p_0,$ is real valued (third step) and using the product rule and Lemma \ref{la:2rho_rank1} (fourth step), $T_{Y}^V(p_{V,-\rhoa})(e)$ equals
\begin{align*}
-\sum_{j=2}^{\dim\mf p}\on{pr}_Y(\omega(\tilde X_j)\ell(\Ik{X_j})\phi_V)(e)&=-\langle\sum_{j=2}^{\dim\mf p}\omega(\tilde X_j)\ell(\Ik{X_j})\phi_V,\frac{\phi_Y}{\langle\phi_Y,\phi_Y\rangle_{L^2(K)}}\rangle_{L^2(K)}\\
&=-\sum_{j=2}^{\dim\mf p}\langle\ell(\Ik{X_j})\phi_V,\omega(\tilde X_j)\frac{\phi_Y}{\langle\phi_Y,\phi_Y\rangle_{L^2(K)}}\rangle_{L^2(K)}\\
&=\langle\phi_V,-2\rhoa(H)\omega(H)\frac{\phi_Y}{\langle\phi_Y,\phi_Y\rangle_{L^2(K)}}\rangle_{L^2(K)}\\
&\qquad+\langle\phi_V,\sum_{j=2}^{\dim\mf p}\omega(\tilde X_j)\ell(\Ik{X_j})\frac{\phi_Y}{\langle\phi_Y,\phi_Y\rangle_{L^2(K)}}\rangle_{L^2(K)}.
\end{align*}
Applying Equation \eqref{eq:proof_proj_at_e} for $V$ and Proposition \ref{prop:helg_general}\,\ref{prop:helg_general_iii} we infer $\dim V\cdot T_{Y}^V(p_{V,-\rhoa})(e)=\dim Y\cdot\overline{T_V^Y(-p_{Y,\rhoa})(e)}$ and thus
\begin{gather*}
\sum_{V\lraomega Y}\frac{\dim V}{\dim Y}\overline{T_{Y}^V(p_{V,-\rhoa})(e)}\iota_Y^V(\phi_V)(X)=\sum_{V\lraomega Y}T_V^Y(-p_{Y,\rhoa})(e)\iota_Y^V(\phi_V)(X).
\end{gather*}
Note that $T_V^Y(-p_{Y,\rhoa})\in V$ is left $M$-invariant since $p_{Y,\rhoa}$ is left $M$-invariant by Lemma \ref{la:gen_diffops}\,\ref{it:gen_diffops1} and $T_V^Y:Y\otimes\mf p^*\to V$ is $K$-equivariant. Therefore it is a multiple of $\phi_V$ and we have $T_V^Y(-p_{Y,\rhoa})=T_V^Y(-p_{Y,\rhoa})(e)\phi_V$.
We infer that
\begin{gather*}
\sum_{V\lraomega Y}T_V^Y(-p_{Y,\rhoa})(e)\iota_Y^V(\phi_V)(X)=\sum_{V\lraomega Y}\iota_Y^V(T_V^Y(-p_{Y,\rhoa}))(X)=-p_{Y,\rhoa}(X).
\end{gather*}
The lemma now follows from the definition of $p_{Y,\rhoa}(X)$.
\end{proof}

We are now able to compute the contribution of the second part in \eqref{eq:decomp_p}.
\begin{lemma}\label{la:p_second_part}
Let $Y\in\hat K_M,\ X\in\mf p$ and $\varphi\in Y$. Then
\begin{gather*}
\sum_{V\lraomega Y}\frac{\dim V}{\dim Y}T_{Y}^V(p_{V,-\rhoa})(e)\iota_V^{Y}(\varphi)(X)(e)=\begin{cases}
-2\rhoa(H)\varphi(e)&\colon X=H\\
-(\ell(\Ik{X})\varphi)(e)&\colon X\perp\mf a
\end{cases}.
\end{gather*}
\end{lemma}

\begin{proof}
Note first that Proposition \ref{prop:iota_gen} implies that
\begin{gather*}
\sum_{V\lraomega Y}\frac{\dim V}{\dim Y}T_{Y}^V(p_{V,-\rhoa})(e)\iota_V^{Y}(\varphi)(X)(e)
=\sum_{V\lraomega Y}\frac{\dim V}{\dim Y}T_{Y}^V(p_{V,-\rhoa})(e)\frac{1}{\lambda(Y,V)}\on{pr}_V(\omega(X)\varphi)(e).
\end{gather*}
By Equation \eqref{eq:proof_proj_at_e} we infer that
\begin{align*}
&\phantom{{}={}}\sum_{V\lraomega Y}\frac{\dim V}{\dim Y}T_{Y}^V(p_{V,-\rhoa})(e)\frac{1}{\lambda(Y,V)}\on{pr}_V(\omega(X)\varphi)(e)\\
&=\sum_{V\lraomega Y}\frac{\dim V}{\dim Y}T_{Y}^V(p_{V,-\rhoa})(e)\frac{1}{\lambda(Y,V)}\langle\omega(X)\varphi,\frac{\phi_V}{\langle\phi_V,\phi_V\rangle_{L^2(K)}}\rangle_{L^2(K)}\\
&=\langle\varphi,\sum_{V\lraomega Y}\frac{\dim V}{\dim Y}\overline{T_{Y}^V(p_{V,-\rhoa})(e)}\frac{1}{\lambda(Y,V)}\omega(X)\frac{\phi_V}{\langle\phi_V,\phi_V\rangle_{L^2(K)}}\rangle_{L^2(K)}\\
&=\langle\varphi,\on{pr}_Y(\sum_{V\lraomega Y}\frac{\dim V}{\dim Y}\overline{T_{Y}^V(p_{V,-\rhoa})(e)}\frac{1}{\lambda(Y,V)}\omega(X)\frac{\phi_V}{\langle\phi_V,\phi_V\rangle_{L^2(K)}})\rangle_{L^2(K)},
\end{align*}
where the last equation follows from $\varphi\in Y$ and the orthogonality of the $K$-types. Using Proposition \ref{prop:iota_gen} and Proposition \ref{prop:lambdaVY_BO} we deduce that
\begin{align*}
&\phantom{{}={}}\on{pr}_Y(\sum_{V\lraomega Y}\frac{\dim V}{\dim Y}\overline{T_{Y}^V(p_{V,-\rhoa})(e)}\frac{1}{\lambda(Y,V)}\omega(X)\frac{\phi_V}{\langle\phi_V,\phi_V\rangle_{L^2(K)}})\\
&=\sum_{V\lraomega Y}\frac{\dim V}{\dim Y}\overline{T_{Y}^V(p_{V,-\rhoa})(e)}\frac{1}{\lambda(Y,V)}\on{pr}_Y(\omega(X)\frac{\phi_V}{\langle\phi_V,\phi_V\rangle_{L^2(K)}})\\
&=\sum_{V\lraomega Y}\frac{\dim V}{\dim Y}\overline{T_{Y}^V(p_{V,-\rhoa})(e)}\frac{\lambda(V,Y)}{\lambda(Y,V)}\iota_Y^V(\frac{\phi_V}{\langle\phi_V,\phi_V\rangle_{L^2(K)}})(X)\\
&=\sum_{V\lraomega Y}\overline{T_{Y}^V(p_{V,-\rhoa})(e)}\iota_Y^V(\frac{\phi_V}{\langle\phi_V,\phi_V\rangle_{L^2(K)}})(X).
\end{align*}
Finally Proposition \ref{prop:helg_general}\,\ref{prop:helg_general_iii} and Lemma \ref{la:p_second_part_local} imply that
\begin{align*}
&\phantom{{}={}}\sum_{V\lraomega Y}\overline{T_{Y}^V(p_{V,-\rhoa})(e)}\iota_Y^V(\frac{\phi_V}{\langle\phi_V,\phi_V\rangle_{L^2(K)}})(X)\\
&=\frac{1}{\langle\phi_Y,\phi_Y\rangle_{L^2(K)}}\sum_{V\lraomega Y}\frac{\langle\phi_Y,\phi_Y\rangle_{L^2(K)}}{\langle\phi_V,\phi_V\rangle_{L^2(K)}}\overline{T_{Y}^V(p_{V,-\rhoa})(e)}\iota_Y^V(\phi_V)(X)\\
&=\frac{1}{\langle\phi_Y,\phi_Y\rangle_{L^2(K)}}\sum_{V\lraomega Y}\frac{\dim V}{\dim Y}\overline{T_{Y}^V(p_{V,-\rhoa})(e)}\iota_Y^V(\phi_V)(X)\\
&=\frac{1}{\langle\phi_Y,\phi_Y\rangle_{L^2(K)}}\begin{cases}
-2\rhoa(H)\phi_Y&\colon X=H\\
\ell(\Ik{X})\phi_Y&\colon X\perp\mf a
\end{cases}.
\end{align*}
Summarizing, we have for $X=H$
\begin{align*}
&\phantom{{}={}}\sum_{V\lraomega Y}\frac{\dim V}{\dim Y}T_{Y}^V(p_{V,-\rhoa})(e)\iota_V^{Y}(\varphi)(X)(e)\\
&=-2\rhoa(H)\langle\varphi,\frac{\phi_Y}{\langle\phi_Y,\phi_Y\rangle_{L^2(K)}}\rangle_{L^2(K)}=-2\rhoa(H)\varphi(e)
\end{align*}
and for $X\in\mf p$ with $X\perp\mf a$
\begin{align*}
&\phantom{{}={}}\sum_{V\lraomega Y}\frac{\dim V}{\dim Y}T_{Y}^V(p_{V,-\rhoa})(e)\iota_V^{Y}(\varphi)(X)(e)=\langle\varphi,\ell(\Ik{X})\frac{\phi_Y}{\langle\phi_Y,\phi_Y\rangle_{L^2(K)}}\rangle_{L^2(K)}\\
&=-\langle\ell(\Ik{X})\varphi,\frac{\phi_Y}{\langle\phi_Y,\phi_Y\rangle_{L^2(K)}}\rangle_{L^2(K)}=-(\ell(\Ik{X})\varphi)(e).\qedhere
\end{align*}
\end{proof}

We are now ready to prove the Theorem \ref{thm:fourier_char}.
\begin{proposition}\label{prop:a_action}
In the setting of Lemma \ref{la:a_action_gen_gradients_general} we have
\begin{gather*}
r(H)f=(\mu-\rhoa)(H)f.
\end{gather*}
\end{proposition}

\begin{proof}
By Lemma \ref{la:a_action_gen_gradients_general} and Proposition \ref{prop:gen_grad_omega} we have
\begin{gather*}
r(H)f=\sum_{\tau\in\hat K_M}\sum_{V\lraomega Y_\tau}\frac{\dim V}{\dim Y_\tau}T_{Y_{\tau}}^V(p_{V,\mu})(e)f_{V,\tau,H},
\end{gather*}
with (for $g\in G$) $f_{V,\tau,H}(gM)\coloneqq\iota_V^{Y_{\tau}}(\pi_{Y_\tau}^*(f)(g))(H)(e)$.
Lemma \ref{la:p_first_part} and \ref{la:p_second_part} imply that
\begin{align*}
&\phantom{{}={}}\sum_{V\lraomega Y_\tau}\frac{\dim V}{\dim Y_\tau}T_{Y_{\tau}}^V(p_{V,\mu})(e)\iota_V^{Y_{\tau}}(\pi_{Y_\tau}^*(f)(g))(H)(e)\\
&=(\mu+\rhoa)(H)\pi_{Y_\tau}^*(f)(g)(e)-2\rhoa(H)\pi_{Y_\tau}^*(f)(g)(e)\\
&=(\mu-\rhoa)(H)\pi_{Y_\tau}^*(f)(g)(e)\\
&=(\mu-\rhoa)(H)\pi_{Y_\tau}(\pi_{Y_\tau}^*(f))(g).
\end{align*}
Thus, $r(H)f=\sum_{\tau\in\hat K_M}(\mu-\rhoa)(H)\pi_{\tau}^*(\pi_{Y_\tau}^*(f))=(\mu-\rhoa)(H)f$.
\end{proof}

\begin{proposition}\label{prop:n_action}
Let $\mu\in\mf a^*$ and $f=\fourier{f}\in\mc D'(G/M)$ (recall Lemma \ref{la:distribution_decomp}) with $\pi_{Y_\tau}^*(f)\in C^\infty(G\times_KY_\tau)$ such that the equations from Lemma \ref{la:gen_diffops2}\,\ref{it:gen_diffops21} and \ref{it:gen_diffops22} hold for $f$ for every irreducible constituent of $Y_\tau\otimes\mf p^*$ and every $Y_\tau\in\hat K_M$. Let $U_+\in C^\infty(G\times_M\mf n)$ be a smooth section. Then $U_+f=0$.
\end{proposition}

\begin{proof}
Note first that
\begin{gather*}
U_+f=\sum_{\tau\in\hat K_M}U_+\pi_{Y_\tau}(\pi_{Y_\tau}^*(f)).
\end{gather*}
Let $X_1,\ldots,X_{\dim\mf n}$ be a basis of $\mf n_0$. Then there exist functions $\kappa_j\in C^\infty(G)$ such that
\begin{gather*}
U_+(g)=\sum_{j=1}^{\dim\mf n}\kappa_j(g)X_j\quad\forall g\in G.
\end{gather*}
Writing $\Ck{X_j}$ resp.\@ $\Cp{X_j}$ for the $\mf k$- resp.\@ $\mf p$-part of the Cartan decomposition of $Y_j$ we define
\begin{gather*}
U_+^\mf k(g)\coloneqq\sum_{j=1}^{\dim\mf n}\kappa_j(g)\Ck{X_j},\quad U_+^\mf p(g)\coloneqq\sum_{j=1}^{\dim\mf n}\kappa_j(g)\Cp{X_j}.
\end{gather*}
Note that, by definition of $U_+$ and since $M$ preserves the Cartan decomposition, we have
\begin{align*}
U_+(gm)=\on{Ad}(m^{-1})U_+(g),\ U_+^\mf k(gm)=\on{Ad}(m^{-1})U_+^\mf k(g),\ U_+^\mf p(gm)=\on{Ad}(m^{-1})U_+^\mf p(g)
\end{align*}
for each $g\in G$ and $m\in M$. We have
\begin{align}
\nonumber U_+^\mf k\pi_{Y_\tau}(\pi_{Y_\tau}^*(f))(gM)&=\sum_{j=1}^{\dim\mf n}\kappa_j(g)\d\pi_{Y_\tau}(\pi_{Y_\tau}^*(f))(g\exp t\Ck{X_j}M)\\
\nonumber&=\sum_{j=1}^{\dim\mf n}\kappa_j(g)\d\pi_{Y_\tau}^*(f)(g\exp t\Ck{X_j})(e)\\
\label{eq:proof_p_k_part}&=-\sum_{j=1}^{\dim\mf n}\kappa_j(g)(\ell(\Ck{X_j})\pi_{Y_\tau}^*(f)(g))(e).
\end{align}
For the $\mf p$-part we obtain
\begin{align*}
U_+^\mf p\pi_{Y_\tau}(\pi_{Y_\tau}^*(f))(gM)&=\sum_{j=1}^{\dim\mf n}\kappa_j(g)\d\pi_{Y_\tau}(\pi_{Y_\tau}^*(f))(g\exp t\Cp{X_j}M)\\
&=\sum_{j=1}^{\dim\mf n}\kappa_j(g)\d\pi_{Y_\tau}^*(f)(g\exp t\Cp{X_j})(e)\\
&=\sum_{j=1}^{\dim\mf n}\kappa_j(g)(((\nabla\circ\pi_{Y_\tau}^*(f))(g))(\Cp{X_j}))(e).
\end{align*}
As in the proof of Lemma \ref{la:a_action_gen_gradients_general} we infer that
\begin{gather*}
U_+^\mf p\pi_{Y_\tau}(\pi_{Y_\tau}^*(f))(gM)=
\sum_{j=1}^{\dim\mf n}\kappa_j(g)\sum_{V\lraomega Y_\tau}\frac{\dim Y_\tau}{\dim V}T_V^{Y_\tau}(p_{Y_\tau,\mu})(e)\iota_{Y_\tau}^V(\pi_{V}^*(f)(g))(\Cp{X_j})(e).
\end{gather*}
If we define
\begin{gather*}
\Psi_{V,Y_\tau}\in C^\infty(G/M),\quad\Psi_{V,Y_\tau}(gM)\coloneqq\sum_{j=1}^{\dim\mf n}\kappa_j(g)\iota_{Y_\tau}^V(\pi_{V}^*(f)(g))(\Cp{X_j})(e)
\end{gather*}
we thus have
\begin{gather*}
U_+^\mf p\pi_{Y_\tau}(\pi_{Y_\tau}^*(f))=
\sum_{V\lraomega Y_\tau}\frac{\dim Y_\tau}{\dim V}T_V^{Y_\tau}(p_{Y_\tau,\mu})(e)\Psi_{V,Y_\tau}.
\end{gather*}
Therefore
\begin{align*}
U_+^{\mf p}f&=\sum_{Y_\tau\in\hat K_M}U_+^{\mf p}\pi_{Y_\tau}(\pi_{Y_\tau}^*(f))=\sum_{\tau\in\hat K_M}\sum_{V\lraomega Y_\tau}\frac{\dim Y_\tau}{\dim V}T_V^{Y_\tau}(p_{Y_\tau,\mu})(e)\Psi_{V,Y_\tau}\\
&=\sum_{V\in\hat K_M}\sum_{V\lraomega Y_\tau}\frac{\dim Y_\tau}{\dim V}T_V^{Y_\tau}(p_{Y_\tau,\mu})(e)\Psi_{V,Y_\tau}.
\end{align*}
Finally Lemma \ref{la:p_first_part} and \ref{la:p_second_part} imply that, for $V\in\hat K_M$ fixed,
\begingroup
\allowdisplaybreaks
\begin{align}
\nonumber&\phantom{{}={}}\sum_{V\lraomega Y_\tau}\frac{\dim Y_\tau}{\dim V}T_V^{Y_\tau}(p_{Y_\tau,\mu})(e)\Psi_{V,Y_\tau}(gM)\\
\nonumber&=\sum_{j=1}^{\dim\mf n}\kappa_j(g)\sum_{V\lraomega Y_\tau}\frac{\dim Y_\tau}{\dim V}T_V^{Y_\tau}(p_{Y_\tau,\mu})(e)\iota_{Y_\tau}^V(\pi_{V}^*(f)(g))(\Cp{X_j})(e)\\
\nonumber&=\sum_{j=1}^{\dim\mf n}\kappa_j(g)(-\ell(\Ik{\Cp{X_j}})\pi_{V}^*(f)(g))(e)\\
\label{eq:proof_p_p_part}&=\sum_{j=1}^{\dim\mf n}\kappa_j(g)(\ell(\Ck{X_j})\pi_{V}^*(f)(g))(e).
\end{align}
\endgroup
Combining Equation \eqref{eq:proof_p_k_part} and \eqref{eq:proof_p_p_part} we infer
\begin{gather*}
U_+f=U_+^\mf kf+U_+^\mf pf=\sum_{V\in\hat K_M}U_+^\mf k\pi_V(\pi_{V}^*(f))+\sum_{V\lraomega Y_\tau}\frac{\dim Y_\tau}{\dim V}T_V^{Y_\tau}(p_{Y_\tau,\mu})(e)\Psi_{V,Y_\tau}=0.\qedhere
\end{gather*}
\end{proof}

\begin{theorem}[Fourier characterization of spherical principal series]\label{thm:fourier_char}
Let $\mu\in\mf a^*$ and $f=\fourier{f}\in\mc D'(G/M)$ (recall Lemma \ref{la:distribution_decomp}) with $\pi_{Y_\tau}^*(f)\in C^\infty(G\times_KY_\tau)$ such that the equations from Lemma \ref{la:gen_diffops2}\,\ref{it:gen_diffops21} and \ref{it:gen_diffops22} hold for $f$ for every irreducible constituent of $Y_\tau\otimes\mf p^*$ and every $Y_\tau\in\hat K_M$. Then $f\in\pstd{\mu}$.
\end{theorem}

\begin{proof}
This follows from Proposition \ref{prop:a_action}, Proposition \ref{prop:n_action} and the characterization $\mc R(\mu-\rhoa)$ of $\pstd{\mu}$ from page \pageref{eq:distr_elem_ps}.
\end{proof}

\begin{lemma}\label{la:sum_gen_grad}
Let $Y\in\hat K_M$ and $f\in C^\infty(G\times_KY)$. Then, for each $g\in G$,
\begin{gather*}
\sum_{V\lraomega Y}\mathrm{d}_V^Y(f)(g)=\sum_{j=1}^{\dim\mf p}\omega(\widetilde X_j)(r(X_j)f)(g).
\end{gather*}
\end{lemma}

\begin{proof}
By definition we have $(\nabla f)(g)=\sum_{j=1}^{\dim\mf p}(r(X_j)f)(g)\otimes\I(\widetilde X_j)\in Y\otimes\mf p^*$.
Therefore,
\begin{gather*}
(\omega_Y\circ(\on{id}_Y\otimes\I^{-1}))((\nabla f)(g))=\sum_{j=1}^{\dim\mf p}\omega(\widetilde X_j)(r(X_j)f)(g)
\end{gather*}
and by Remark \ref{rem:sum_of_projections} and the definition of the generalized gradients $\mathrm{d}_V^Y=T_V^Y\circ\nabla$ we obtain
\begin{gather*}
\sum_{V\lraomega Y}\mathrm{d}_V^Y(f)(g)=\sum_{j=1}^{\dim\mf p}\omega(\widetilde X_j)(r(X_j)f)(g).\qedhere
\end{gather*}
\end{proof}

\begin{lemma}\label{la:gen_grad_adjoint}
Let $V,\, Y\in\hat K_M$ with $V\lraomega Y$, $\varphi\in C^\infty(G\times_KY)$ and $\psi\in C^\infty(G\times_KV)$. Then, if one side exists,
\begin{gather*}
\langle\pi_Y(\varphi),\pi_Y(\mathrm{d}_Y^V(\psi))\rangle_{L^2(G)}=-\langle\pi_V(\mathrm{d}_V^Y(\varphi)),\pi_V(\psi)\rangle_{L^2(G)}.
\end{gather*}
\end{lemma}

\begin{proof}
Note first that if $Y\neq W\in\hat K$ and $\eta\in C^\infty(G\times_KW)$ we have
\begin{gather*}
\langle\pi_V(\varphi),\pi_W(\eta)\rangle_{L^2(G)}=0
\end{gather*}
by splitting the integral into $G/K$ and $K$. Therefore we obtain
\begin{gather*}
\langle\pi_Y(\varphi),\pi_Y(\mathrm{d}_Y^V(\psi))\rangle_{L^2(G)}=\langle\pi_Y(\varphi),\sum_{W\lraomega V}\pi_W(\mathrm{d}_W^V(\psi))\rangle_{L^2(G)}.
\end{gather*}
Evaluating Lemma \ref{la:sum_gen_grad} at $eM\in K/M$ yields (since $\omega(\widetilde X_j)(eM)=0$ for $j\geq 2$)
\begin{gather*}
\sum_{W\lraomega V}\pi_W(\mathrm{d}_W^V(\psi))=r(H)\pi_V(\psi).
\end{gather*}
Together we conclude that
\begin{gather*}
\langle\pi_Y(\varphi),\pi_Y(\mathrm{d}_Y^V(\psi))\rangle_{L^2(G)}=\langle\pi_Y(\varphi),r(H)\pi_V(\psi)\rangle_{L^2(G)}=-\langle r(H)\pi_Y(\varphi),\pi_V(\psi)\rangle_{L^2(G)},
\end{gather*}
where we used the right-invariance of the Haar measure on $G$. The same argument yields
\begin{gather*}
\langle r(H)\pi_Y(\varphi),\pi_V(\psi)\rangle_{L^2(G)}=\langle\pi_V(\mathrm{d}_V^Y(\varphi)),\pi_V(\psi)\rangle_{L^2(G)}.\qedhere
\end{gather*}
\end{proof}

\section{Spectral Correspondence}\label{sec:spec_corres}
In this section we describe the image of the minimal $K$-type Poisson transforms occurring in Proposition \ref{prop:Poisson_sum} restricted to the socle. This will yield a quantum-classical correspondence for the first band by Remark \ref{rem:poisson_qc}. The characterization of the Poisson images require some case by case calculations to decompose certain tensor products. We put these calculations into Appendix \ref{app:comp_scalars} to make the arguments presented in this section more transparent.

\subsection{\texorpdfstring{The Case of $G=\GSO{n},\ n\geq 3$}{The Case of G=SO(n,1), n>=3}}
By Propositions \ref{prop:tensor_decomp_odd} and \ref{prop:tensor_decomp_even} we have for each $k\in\N_0$
\begin{gather*}
Y_k\otimes\mf p^*\cong Y_{k-1}\oplus Y_{k+1}\oplus V_k\quad\text{ if }n\neq3, \qquad Y_k\otimes\mf p^*\cong Y_{k-1}\oplus Y_{k+1}\oplus Y_k\quad\text{ if }n=3,
\end{gather*}
where $V_k$, with highest weight $ke_1+e_2$, is not $M$-spherical. We define generalized gradients $\mathrm{d}_V^{Y_k}\coloneqq T_V^{Y_k}\circ\nabla$ with $T_V^{Y_k}\in\Hom{K}{Y_k\otimes\mf p^*}{V}$ as in Definitions \ref{def:gen_grad}, \ref{def:omega} and abbreviate
\begin{gather*}
\mathrm{d}_\pm\coloneqq\mathrm{d}_{Y_{k\pm1}}^{Y_k},\qquad\quad\mathrm{D}\coloneqq\mathrm{d}_{V_k}^{Y_k}\quad\text{ resp.\@ }\quad\mathrm{D}\coloneqq\mathrm{d}_{Y_k}^{Y_k}
\end{gather*}
for $n\neq 3$ resp.\@ $n=3$. Let $\mu=-\rhoa-\ell\alpha\in\mathbf{Ex}$, see Theorem \ref{thm:socle}, be an exceptional parameter and recall the structure and properties of $\on{soc}(\pst{\mu})$ from Theorem \ref{thm:socle}. Using Proposition \ref{prop:gen_grad_omega}, Proposition \ref{prop:lambdaVY_BO}.\ref{it:lambdaVY_2} and Remark \ref{rem:lambda_R} we infer for each $k\in\N_0$
\begin{gather*}
V\notleftrightomega Y_k\:\Longrightarrow\:\mathrm{d}_V^{Y_k}\circ P_\mu^{Y_k}=0\qquad\text{and}\qquad V\lraomega Y_k\:\Longleftrightarrow\: V\in\{Y_{k-1}, Y_{k+1}\}
\end{gather*}
if $Y_{k-1}$ exists\footnote{For $k=0$ we only have $Y_1$ on the right hand side of the second equivalence.}. Therefore,
\begin{gather*}
\mathrm{D}\circ P_\mu^{Y_k}=0.
\end{gather*}
By Theorem \ref{thm:socle} the minimal $K$-type of $\on{soc}(\pst{\mu})$ is $Y_{\ell+1}$. Since
\begin{gather*}
\mathrm{d}_-\circ P_\mu^{Y_{\ell+1}}=T_{Y_{\ell}}^{Y_{\ell+1}}(p_{Y_{\ell+1},\mu})(e)P_\mu^{Y_{\ell}}
\end{gather*}
by Lemma \ref{la:gen_diffops}.\ref{it:gen_diffops2}, and since Proposition \ref{prop:PoissonInjective} implies that $\restr{P_\mu^{Y_\ell}}{(\on{soc}(\pst{\mu}))^{-\infty}}=0$, we obtain
\begin{gather*}
\mathrm{d}_-\circ\restr{P_\mu^{Y_{\ell+1}}}{(\on{soc}(\pst{\mu}))^{-\infty}}=0.
\end{gather*}
Summarizing, we have
\begin{gather*}
P_\mu^{Y_{\ell+1}}:(\on{soc}(\pst{\mu}))^{-\infty}\to\{f\in C^\infty(G\times_KY_{\ell+1})\mid\mathrm{d}_-f=0,\ \mathrm{D}f=0\}.
\end{gather*}
We will now investigate which $K$-types $\mu$ with highest weight $\mu_1e_1+\ldots+\mu_me_m$, $m\coloneqq\on{rk}\mf k=\lfloor\frac{n}{2}\rfloor$, occur on the right hand side. Applying \cite[Thm.\@ 6]{DGK88} to the minimal $K$-type $\tau\coloneqq Y_{\ell+1}$ (with highest weight $(\ell+1)e_1$) of $\on{soc}(\pst{\mu})$, we find that $\mu_j=0$ for $j>1$, $\mu_1\geq\ell+1$ and that each $\mu$ of this form occurs with multiplicity one. Therefore, the highest weights of the $K$-types in $\{f\in C^\infty(G\times_KY_{\ell+1})\mid\mathrm{d}_-f=0,\ \mathrm{D}f=0\}$ are given by $ke_1$ for $k\geq\ell+1$. Since $Y_k$ has highest weight $ke_1$, these $K$-types are exactly the same as the $K$-types of $\on{soc}(\pst{\mu})$ (see Theorem \ref{thm:socle}). Hence, we have
\begin{gather*}
(\on{soc}(\pst{\mu}))_K\cong\{f\in C^\infty(G\times_KY_{\ell+1})\mid\mathrm{d}_-f=0,\ \mathrm{D}f=0\}_K,
\end{gather*}
where the $K$ in the index denotes the Harish-Chandra module. Proceeding as in \cite[Satz 4.13]{Ol} we infer that the Poisson transform $P_\mu^{Y_{\ell+1}}$ yields an isomorphism (similar to the scalar case, see Equation \eqref{eq:E_distr}, Definition \ref{def:sc_PT}) from $(\on{soc}(\pst{\mu}))^{-\infty}$ to
\begin{gather*}
\{f\in C^\infty(G\times_KY_{\ell+1})\mid\mathrm{d}_-f=0,\ \mathrm{D}f=0,\exists\: r\geq0\colon \sup_{g\in G}\abs{e^{-rd_{G/K}(eK,gK)}f(g)}<\infty\}.
\end{gather*}
In particular, we have the following correspondence for the $\Gamma$-invariant elements.
\begin{theorem}[Spectral Correspondence]\label{thm:spectral_corres_SOn}
Let $\mathbf{Ex}\ni\mu=-(\rhoa+\ell\alpha),\ \ell\in\N_0,$ be an exceptional parameter. Then the socle $\on{soc}(\pst{\mu})$ of $\pst{\mu}$ is irreducible, unitary and its $K$-types are given by $Y_k$ for $k\geq\ell+1$. The minimal $K$-type is $Y_{\ell+1}$ and the corresponding Poisson transform induces an isomorphism
\begin{gather*}
P_\mu^{Y_{\ell+1}}\colon{}^\Gamma(\on{soc}(\pst{\mu}))^{-\infty}\cong{}^\Gamma\{f\in C^\infty(G\times_KY_{\ell+1})\mid\mathrm{d}_-f=0,\ \mathrm{D}f=0\}.
\end{gather*}
\end{theorem}

\begin{proof}
This follows from the discussion above and the fact that each $\Gamma$-invariant function fulfills the growth condition (for each $r\geq 0$)
\begin{gather*}
\sup_{g\in G}\abs{e^{-rd_{G/K}(eK,gK)}f(g)}=\sup_{g\in\mc{F}}\abs{e^{-rd_{G/K}(eK,gK)}f(g)}<\infty,
\end{gather*}
where $\mc{F}$ denotes a fundamental domain of $\Gamma\backslash G$ (note that the latter is compact by assumption).
\end{proof}

\begin{example}
For the first exceptional parameter $\mu=-\rhoa$ we get ($Y_1\cong\mf p^*$)
\begin{gather*}
P_{-\rhoa}^{Y_{1}}\colon{}^\Gamma(\on{soc}(\pst{-\rhoa}))^{-\infty}\cong\{f\in C^\infty(\Lambda^1(\Gamma\backslash G/K))\mid\delta f=0,\ d{ }f=0\},
\end{gather*}
where $\Lambda^1(\Gamma\backslash G/K)$ denotes the bundle of one forms and ($\delta$ resp.\@) $d$ is the (co)-differential. The dimension is given by the first Betti number $b_1(\Gamma\backslash G/K)$ in this case.
\end{example}

\subsection{\texorpdfstring{The Case of $G=\GSU{n},\ n\geq 2$}{The Case of G=SU(n,1), n>=2}}
By Proposition \ref{prop:tensor_decomp_eq_rk} and Remark \ref{rem:dim_C} we have for $p,q\in\N_0$
\begin{gather*}
Y_{p,q}\otimes\mf p^*\cong\bigoplus_{\beta\in S}Y_{p,q,\beta},
\end{gather*}
where $S\coloneqq\{\pm(e_1-e_{n+1}),e_2-e_{n+1},-e_{n-1}+e_{n+1},\pm(e_n-e_{n+1})\}\subseteq\Delta_n$. The representations $V_1$ resp.\@ $V_2$ with highest weights $qe_1+e_2-pe_n+(p-q-1)e_{n+1}$ resp.\@  $qe_1-e_{n-1}-pe_n+(p-q+1)e_{n+1}$ are not $M$-spherical. In this notation we have
\begin{gather*}
Y_{p,q}\otimes\mf p^*\cong Y_{p-1,q}\oplus Y_{p+1,q}\oplus Y_{p,q-1}\oplus Y_{p,q+1}\oplus V_1\oplus V_2,
\end{gather*}
whenever these representations exist (i.e.\@ whenever the corresponding weights of $Y_{p,q,\beta}$ are indeed dominant). We define generalized gradients $\mathrm{d}_V^{Y_{p,q}}\coloneqq T_V^{Y_{p,q}}\circ\nabla$ with $T_V^{Y_{p,q}}\in\Hom{K}{Y_{p,q}\otimes\mf p^*}{V}$ as in Definition \ref{def:omega} and abbreviate
\begin{gather*}
\mathrm{d}_{\pm,1}\coloneqq\mathrm{d}_{Y_{p\pm1,q}}^{Y_{p,q}},\qquad\mathrm{d}_{\pm,2}\coloneqq\mathrm{d}_{Y_{p,q\pm1}}^{Y_{p,q}},\qquad\quad\mathrm{D}_j\coloneqq\mathrm{d}_{V_j}^{Y_{p,q}},\quad j=1,2.
\end{gather*}
Let $\mu=-(\rhoa+2\ell\alpha)\in\mathbf{Ex},\ \ell\in\N_0,$ be an exceptional parameter and recall the structure and properties of $\on{soc}(\pst{\mu})$ from Theorem \ref{thm:socle}. Using Proposition \ref{prop:gen_grad_omega}, Proposition \ref{prop:lambdaVY_BO}.\ref{it:lambdaVY_2} and Remark \ref{rem:lambda_C} we infer
\begin{gather*}
V\notleftrightomega Y_{p,q}\:\Longrightarrow\:\mathrm{d}_V^{Y_{p,q}}\circ P_\mu^{Y_{p,q}}=0\qquad\text{and}\qquad V\lraomega Y_{p,q}\:\Longleftrightarrow\: V\in\{Y_{p\pm1,q}, Y_{p,q\pm1}\}
\end{gather*}
whenever the occurring representations exist. Therefore, for $j\in\{1,2\}$,
\begin{gather}\label{eq:D_j_SUn}
\mathrm{D}_j\circ P_\mu^{Y_{p,q}}=0.
\end{gather}
The minimal $K$-type of $\on{soc}(\pst{\mu})$ is given by $Y_{\ell+1,\ell+1}$ (see Theorem \ref{thm:socle}). By Lemma \ref{la:gen_diffops}.\ref{it:gen_diffops2},
\begin{align*}
\mathrm{d}_{-,1}\circ P_\mu^{Y_{\ell+1,\ell+1}}&=T_{Y_{\ell,\ell+1}}^{Y_{\ell+1,\ell+1}}(p_{Y_{\ell+1,\ell+1},\mu})(e)P_\mu^{Y_{\ell,\ell+1}}\\
\mathrm{d}_{-,2}\circ P_\mu^{Y_{\ell+1,\ell+1}}&=T_{Y_{\ell+1,\ell}}^{Y_{\ell+1,\ell+1}}(p_{Y_{\ell+1,\ell+1},\mu})(e)P_\mu^{Y_{\ell+1,\ell}}.
\end{align*}
Since Proposition \ref{prop:PoissonInjective} implies that $P_\mu^{Y_{\ell,\ell+1}}|_{(\on{soc}(\pst{\mu}))^{-\infty}}=0$ and $P_\mu^{Y_{\ell+1,\ell}}|_{(\on{soc}(\pst{\mu}))^{-\infty}}=0$, we obtain that, for $j\in\{1,2\}$,
\begin{gather*}
\mathrm{d}_{-,j}\circ\restr{P_\mu^{Y_{\ell+1,\ell+1}}}{(\on{soc}(\pst{\mu}))^{-\infty}}=0.
\end{gather*}
Summarizing, we have
\begin{gather}\label{eq:PT_ker_inj}
P_\mu^{Y_{\ell+1,\ell+1}}:(\on{soc}(\pst{\mu}))^{-\infty}\to\{f\in C^\infty(G\times_KY_{\ell+1,\ell+1})\mid\mathrm{d}_{-,j}f=0,\ \mathrm{D}_jf=0,\ j\in\{1,2\}\}.
\end{gather}
We will first present a method similar to the case of $G=\GSO{n}$. For this method we have to assume $n\neq 2$ and $\ell\neq 0$. Then \cite[Eq.\@ (2.7.3), (2.7.4), Lem.\@ 6.2.1, Prop.\@ 6.4.6]{M89} imply that the highest weights of the $K$-types on the right hand side of \eqref{eq:PT_ker_inj} are given by $p'e_1-q'e_n+(q'-p')e_{n+1}$ with $p'\geq\ell+1$ and $q'\geq\ell+1$, each occurring with multiplicity at most one. By definition, the corresponding representations are $Y_{p,q}$ for $p,q\geq\ell+1$. Since the Poisson transform $P_\mu^{Y_{\ell+1,\ell+1}}$ is injective by Proposition \ref{prop:Poisson_sum}, each $K$-type of the socle (see Theorem \ref{thm:socle}) has to occur in its image (restricted to the socle). Therefore the $K$-types of $\{f\in C^\infty(G\times_KY_{\ell+1,\ell+1})\mid\mathrm{d}_{-,j}f=0,\ \mathrm{D}_jf=0,\ j\in\{1,2\}\}$ are given by $Y_{p,q},\ p,q\geq\ell+1,$ each one occurring with multiplicity one. Hence, we obtain
\begin{gather*}
(\on{soc}(\pst{\mu}))_K\cong\{f\in C^\infty(G\times_KY_{\ell+1,\ell+1})\mid\mathrm{d}_{-,j}f=0,\ \mathrm{D}_jf=0,\ j\in\{1,2\}\}_K.
\end{gather*}
Proceeding as in the case of $G=\GSO{n}$ we find
\begin{theorem}[Spectral Correspondence 1]\label{thm:spectral_corres_SUn1}
Let $n\neq 2$ and $\mathbf{Ex}\ni\mu=-(\rhoa+2\ell\alpha),\ \ell\in\N_0,$ be an exceptional parameter with $\ell\neq0$. Then the socle $\on{soc}(\pst{\mu})$ of $\pst{\mu}$ is irreducible, unitary and its $K$-types are given by $Y_{p,q}$ for $p,q\geq\ell+1$. The minimal $K$-type is $Y_{\ell+1,\ell+1}$ and the corresponding Poisson transform induces an isomorphism
\begin{gather*}
P_\mu^{Y_{\ell+1,\ell+1}}\colon{}^\Gamma(\on{soc}(\pst{\mu}))^{-\infty}\cong{}^\Gamma\{f\in C^\infty(G\times_KY_{\ell+1,\ell+1})\colon\mathrm{d}_{-,j}f=0,\ \mathrm{D}_jf=0,\ j\in\{1,2\}\}.
\end{gather*}
\end{theorem}

\begin{proof}
See Theorem \ref{thm:spectral_corres_SOn}.
\end{proof}

In order to treat the remaining parameters ($n=2$ or $\ell=0$) we will use the Fourier characterization of the principal series. The following lemma is based on Lemma \ref{la:def_distr}.
\begin{lemma}\label{la:poly_growth_SUn}
Let $\mu\coloneqq-(\rhoa+2\ell\alpha),\ \ell\in\N_0,$ an exceptional parameter. Let $\psi_{p,q}\in C^\infty(G\times_KY_{p,q})$ for $p,q\geq\ell+1$ be such that the equations from Lemma \ref{la:gen_diffops2} are fulfilled (with $\psi_{p,q}$ instead of $\pi_{Y_{p,q}}^*(f)$). Assume that $\pi_{Y_{\ell+1,\ell+1}}(\psi_{Y_{\ell+1,\ell+1}})\in C^\infty(G)$ has finite $L^2$-norm.
Then the formal sum
\begin{gather*}
f\coloneqq\sum_{p,q\geq\ell+1}\iota_{G/M}(\pi_{Y_{p,q}}(\psi_{p,q}))
\end{gather*}
defines a distribution on $G/M$.
\end{lemma}

\begin{proof}
We abbreviate $T^{p_1,q_1}_{p_2,q_2}\coloneqq T_{Y_{p_2,q_2}}^{Y_{p_1,q_1}}(p_{Y_{p_1,q_1,\mu}})(e)\in\C$. It suffices to prove the estimate in Lemma \ref{la:def_distr}. Using Lemma \ref{la:gen_grad_adjoint} (second step) and the equations from Lemma \ref{la:gen_diffops2} (first and third step) we infer for the $L^2$ inner product
\begin{align*}
\norm{\pi_{Y_{p,q}}(\psi_{p,q})}^2&=\frac{\dim Y_{p,q}}{\dim Y_{p-1,q}}\frac{1}{T^{p-1,q}_{p,q}}\langle\pi_{Y_{p,q}}(\psi_{p,q}),\pi_{Y_{p,q}}(\mathrm{d}_{+,1}\psi_{p-1,q})\rangle\\
&=-\frac{\dim Y_{p,q}}{\dim Y_{p-1,q}}\frac{1}{T^{p-1,q}_{p,q}}\langle\pi_{Y_{p-1,q}}(\mathrm{d}_{-,1}\psi_{p,q}),\pi_{Y_{p-1,q}}(\psi_{p-1,q})\rangle\\
&=-\left(\frac{\dim Y_{p,q}}{\dim Y_{p-1,q}}\right)^2\frac{T^{p,q}_{p-1,q}}{T^{p-1,q}_{p,q}}\langle\pi_{Y_{p-1,q}}(\psi_{p-1,q}),\pi_{Y_{p-1,q}}(\psi_{p-1,q})\rangle.
\end{align*}
By Proposition \ref{prop:nu_explicit}, Remark \ref{rem:lambda_C} and Remark \ref{rem:dim_C} this equals
\begin{gather*}
\frac{(n+p-2)(n+p+q-1)}{p(n+p+q-2)}\frac{n+p}{p-1-\ell}\norm{\pi_{Y_{p-1,q}}(\psi_{p-1,q})}^2.
\end{gather*}
Iteratively applying this equation we find that for each $m\in\N_0$
\begin{align*}
\norm{\pi_{Y_{\ell+m,q}}(\psi_{\ell+m,q})}^2=\prod_{r=2}^m\frac{(n+\ell+r-2)(n+\ell+r+q-1)}{(\ell+r)(n+\ell+r+q-2)}\frac{n+\ell+r}{r-1}\norm{\pi_{Y_{\ell+1,q}}(\psi_{\ell+1,q})}^2.
\end{align*}
The latter product equals
\begin{gather*}
\frac{(n+\ell+m+q-1)(n+\ell+m-2)!(\ell+1)!(n+\ell+m)!}{(n+\ell+q)(n+\ell-1)!(\ell+m)!(m-1)!(n+\ell+1)!}\norm{\pi_{Y_{\ell+1,q}}(\psi_{\ell+1,q})}^2,
\end{gather*}
which grows polynomially in $m$ (in fact it is $\mc O(m^{2n+\ell})$). Interchanging the roles of $p$ and $q$ this proves the estimate in Lemma \ref{la:def_distr} and therefore the lemma.
\end{proof}

\begin{theorem}[Spectral Correspondence 2]\label{thm:spectral_corres_SUn2}
Let $\mathbf{Ex}\ni\mu=-(\rhoa+2\ell\alpha),\ \ell\in\N_0,$ be an exceptional parameter. Then the socle $\on{soc}(\pst{\mu})$ of $\pst{\mu}$ is irreducible, unitary and its $K$-types are given by $Y_{p,q}$ for $p,q\geq\ell+1$. The minimal $K$-type is $Y_{\ell+1,\ell+1}$ and the corresponding Poisson transform induces an isomorphism from $\colon{}^\Gamma(\on{soc}(\pst{\mu}))^{-\infty}$ onto
\begin{gather*}
{}^\Gamma\{u\in C^\infty(G\times_KY_{\ell+1,\ell+1})\mid properties\ i)-vi)\ below\},
\end{gather*}
where the properties are as follows. For $u\in C^\infty(G\times_KY_{\ell+1,\ell+1})$ let $\psi_{\ell+1,\ell+1}\coloneqq\dim Y_{\ell+1,\ell+1}\cdot u$ and define recursively for $p,q\geq\ell+1$ (see Lemma \ref{la:gen_diffops2})
\begin{align*}
\psi_{p,\ell+1}\coloneqq\frac{\dim Y_{p,\ell+1}}{\dim Y_{p-1,\ell+1}}\frac{1}{T^{p-1,\ell+1}_{p,\ell+1}}\mathrm{d}_{+,1}\psi_{p-1,\ell+1},\qquad \psi_{p,q}\coloneqq\frac{\dim Y_{p,q}}{\dim Y_{p,q-1}}\frac{1}{T^{p,q-1}_{p,q}}\mathrm{d}_{+,2}\psi_{p,q-1},
\end{align*}
where we abbreviate $T^{p_1,q_1}_{p_2,q_2}\coloneqq T_{Y_{p_2,q_2}}^{Y_{p_1,q_1}}(p_{Y_{p_1,q_1,\mu}})(e)\in\C$. Then we define the properties
\begin{enumerate}
\item $\mathrm{d}_{+,1}\psi_{p,q}=T_{p+1,q}^{p,q}\frac{\dim Y_{p,q}}{\dim Y_{p+1,q}}\psi_{p+1,q},\hfill (p\geq\ell+1,\ q\geq\ell+2),$
\item $\mathrm{d}_{-,1}\psi_{p,q}=T_{p-1,q}^{p,q}\frac{\dim Y_{p,q}}{\dim Y_{p-1,q}}\psi_{p-1,q},\hfill (p\geq\ell+2,\ q\geq\ell+1),$
\item $\mathrm{d}_{-,1}\psi_{\ell+1,q}=0,\hfill (q\geq\ell+1),$
\item $\mathrm{d}_{-,2}\psi_{p,q}=T_{p,q-1}^{p,q}\frac{\dim Y_{p,q}}{\dim Y_{p,q-1}}\psi_{p,q-1},\hfill (p\geq\ell+1,\ q\geq\ell+2),$
\item $\mathrm{d}_{-,2}\psi_{p,\ell+1}=0,\hfill (p\geq\ell+1),$
\item $\mathrm{D}_j\psi_{p,q}=0,\hfill (p,q\geq\ell+1,\ j\in\{1,2\}).$
\end{enumerate}
\end{theorem}

\begin{proof}
We first prove that the Poisson transform maps into the claimed space. If $u=P_\mu^{Y_{\ell+1,\ell+1}}(f)$ for some $f\in(\on{soc}(\pst{\mu}))^{-\infty}$ we have $\psi_{\ell+1,\ell+1}=\pi_{Y_{\ell+1,\ell+1}}^*(f)$ by Lemma \ref{la:pi_gamma}.\ref{it:la_pi_gamma5}. Properties $i),ii),iv)$ and $vi)$ are exactly the equations from Lemma \ref{la:gen_diffops2}. To prove the third property we note that 
\begin{gather*}
\mathrm{d}_{-,1}\psi_{\ell+1,q}=\mathrm{d}_{-,1}\pi_{Y_{\ell+1,q}}^*(f)=T_{\ell,q}^{\ell+1,q}\frac{\dim Y_{\ell+1,q}}{\dim Y_{\ell,q}}\pi_{Y_{\ell,q}}^*(f)=0,
\end{gather*}
since the socle does not contain the $K$-type $Y_{\ell,q}$. Similarly we see that property $v)$ is fulfilled. Since the Poisson transform is $G$-equivariant it preserves $\Gamma$-invariant elements.

For the surjectivity let $u\in {}^\Gamma C^\infty(G\times_KY_{\ell+1,\ell+1})$ with the desired properties. Define 
\begin{gather*}
f\coloneqq\sum_{p,q\geq\ell+1}\iota_{G/M}(\pi_{Y_{p,q}}(\psi_{p,q})).
\end{gather*}
By Lemma \ref{la:poly_growth_SUn}, $f$ defines a distribution on $G/M$ (note that, since $\Gamma$ is co-compact, the norm $\norm{\pi_{Y_{\ell+1,\ell+1}}(\psi_{Y_{\ell+1,\ell+1}})}_{L^2(G)}$ is finite). By Theorem \ref{thm:fourier_char} we have $f\in\pstd{\mu}$ and since there are only terms for $p,q\geq\ell+1$ in the defining sum for $f$ we also have $f\in(\on{soc}(\pst{\mu}))^{-\infty}$. Since each $\psi_{p,q}$ is $\Gamma$-invariant and each involved map is $G$-equivariant, $f$ is also $\Gamma$-invariant. The orthogonality of the $K$-types implies $\pi_{Y_{\ell+1,\ell+1}}^*(\sum_{p,q\in J}\iota_{G/M}(\pi_{Y_{p,q}}(\psi_{p,q})))=\sum_{p,q\in J}\iota_{G/M}(\pi_{Y_{p,q}}(\psi_{p,q}))\circ\pi_{Y_{\ell+1,\ell+1}}=0$ for $J\coloneqq\{(p,q)\in\N_0^2\colon(p,q)\neq(\ell+1,\ell+1)\}$ and $\pi_{Y_{\ell+1,\ell+1}}^*(\iota_{G/M}(\pi_{Y_{\ell+1,\ell+1}}(\psi_{\ell+1,\ell+1})))=\iota_{Y_{\ell+1,\ell+1}}(\psi_{\ell+1,\ell+1})$ (see Definition \ref{def:pi_gamma} for the relevant definitions). Using Lemma \ref{la:pi_gamma}.\ref{it:la_pi_gamma5} again we obtain
\begin{gather*}
P_\mu^{Y_{\ell+1,\ell+1}}(f)=\frac{1}{\dim Y_{\ell+1,\ell+1}}\pi_{Y_{\ell+1,\ell+1}}^*(f)=\frac{1}{\dim Y_{\ell+1,\ell+1}}\psi_{\ell+1,\ell+1}=u.\qedhere
\end{gather*}
\end{proof}

\subsection{\texorpdfstring{The Case of $G=\GSp{n},\ n\geq 2$}{The Case of G=Sp(n,1), n>=2}}

By Proposition \ref{prop:tensor_decomp_eq_rk} and Remark \ref{rem:dim_H} we have for each $a,b\in\N_0$ with $a\geq b$
\begin{gather*}
V_{a,b}\otimes\mf p^*\cong V_{a+1,b}\oplus V_{a-1,b}\oplus V_{a,b+1}\oplus V_{a,b-1}\oplus\bigoplus_{\substack{\beta\in S\\V_{a,b,\beta}\not\in\hat K_M}} V_{a,b,\beta}.
\end{gather*}
We define generalized gradients $\mathrm{d}_V^{V_{a,b}}\coloneqq T_V^{V_{a,b}}\circ\nabla$ with $T_V^{V_{a,b}}\in\Hom{K}{V_{a,b}\otimes\mf p^*}{V}$ as in Definition \ref{def:omega} and abbreviate
\begin{gather*}
\mathrm{d}_{1,\pm}\coloneqq\mathrm{d}_{V_{a\pm1,b}}^{V_{a,b}},\qquad\quad\mathrm{d}_{2,\pm}\coloneqq\mathrm{d}_{V_{a,b\pm1}}^{V_{a,b}},\qquad\quad\mathrm{D}_\beta\coloneqq\mathrm{d}_{V_{a,b,\beta}}^{V_{a,b}}
\end{gather*}
for each $\beta\in S$ with $V_{a,b,\beta}\not\in\hat K_M$. Let $\mu=-(\rhoa+(2\ell-2)\alpha)\in\mathbf{Ex}$ be an exceptional parameter and recall the structure and properties of $\on{soc}(\pst{\mu})$ from Theorem \ref{thm:socle}. Using Proposition \ref{prop:gen_grad_omega}, Proposition \ref{prop:lambdaVY_BO}.\ref{it:lambdaVY_2} and Remark \ref{rem:lambda_H} we infer for each $a,b\in\N_0$ with $a\geq b$
\begin{gather*}
V\notleftrightomega V_{a,b}\:\Longrightarrow\:\mathrm{d}_V^{V_{a,b}}\circ P_\mu^{V_{a,b}}=0\quad\text{and}\quad V\lraomega V_{a,b}\:\Longleftrightarrow\: V\in\{V_{a+1,b}, V_{a-1,b}, V_{a,b+1}, V_{a,b-1}\}
\end{gather*}
whenever the occurring representations exist. The minimal $K$-type of $\on{soc}(\pst{\mu})$ is given by $V_{\ell+1,\ell+1}$ (see Theorem \ref{thm:socle}).

The spectral correspondence in the quaternionic case is established by using the Fourier characterization of the principal series (see Theorem \ref{thm:fourier_char}). By Lemma \ref{la:def_distr} we obtain the following result.
\begin{lemma}\label{la:poly_growth_Spn}
Let $\mu\coloneqq-(\rhoa+(2\ell-2)\alpha),\ \ell\in\N_0,$ an exceptional parameter. Let $\psi_{a,b}\in C^\infty(G\times_KV_{a,b})$ for $a,b\geq\ell+1$ be such that the equations from Lemma \ref{la:gen_diffops2} are fulfilled (with $\psi_{a,b}$ instead of $\pi_{V_{a,b}}^*(f)$). Assume that $\pi_{V_{\ell+1,\ell+1}}(\psi_{V_{\ell+1,\ell+1}})\in C^\infty(G)$ has finite $L^2$-norm.
Then the formal sum
\begin{gather*}
f\coloneqq\sum_{a\geq b\geq\ell+1}\iota_{G/M}(\pi_{V_{a,b}}(\psi_{a,b}))
\end{gather*}
defines a distribution on $G/M$.
\end{lemma}

\begin{proof}
We abbreviate $T^{a_1,b_1}_{a_2,b_2}\coloneqq T_{V_{a_2,b_2}}^{V_{a_1,b_1}}(p_{V_{a_1,b_1,\mu}})(e)\in\C$. It suffices to prove the estimate in Lemma \ref{la:def_distr}. Using Lemma \ref{la:gen_grad_adjoint} (second step) and the equations from Lemma \ref{la:gen_diffops2} (first and third step) we infer for the $L^2$-norm as in Lemma \ref{la:poly_growth_SUn}
\begin{align*}
\norm{\pi_{V_{a,b}}(\psi_{a,b})}^2&=
-\left(\frac{\dim V_{a,b}}{\dim V_{a-1,b}}\right)^2\frac{T^{a,b}_{a-1,b}}{T^{a-1,b}_{a,b}}\norm{\pi_{V_{a-1,b}}(\psi_{a-1,b})}^2.
\end{align*}
By Equation \eqref{eq:T_decomp}, Proposition \ref{prop:nu_explicit} and Proposition \ref{prop:lambdaVY_BO}.\ref{it:lambdaVY_4} we have
\begin{gather*}
\frac{T^{a,b}_{a-1,b}}{T^{a-1,b}_{a,b}}=\frac{-2n+1-a-\ell}{a-\ell}\frac{\lambda(V_{a,b},V_{a-1,b})}{\lambda(V_{a,b},V_{a-1,b})}=\frac{-2n+1-a-\ell}{a-\ell}\frac{\dim V_{a-1,b}}{\dim V_{a,b}}
\end{gather*}
and thus
\begin{gather*}
\norm{\pi_{V_{a,b}}(\psi_{a,b})}^2=\frac{2n-1+a+\ell}{a-\ell}\frac{\dim V_{a,b}}{\dim V_{a-1,b}}\norm{\pi_{V_{a-1,b}}(\psi_{a-1,b})}^2.
\end{gather*}
Iteratively applying this equation we infer that for each $m\in\N_0$
\begin{align*}
\norm{\pi_{V_{\ell+m,b}}(\psi_{\ell+m,b})}^2&=\prod_{r=2}^m\frac{2n-1+2\ell+r}{r}\frac{\dim V_{\ell+r,b}}{\dim V_{\ell+r-1,b}}\norm{\pi_{V_{\ell+1,b}}(\psi_{\ell+1,b})}^2\\
&=\frac{\dim V_{\ell+m,b}}{\dim V_{\ell+1,b}}\prod_{r=2}^m\frac{2n-1+2\ell+r}{r}\norm{\pi_{V_{\ell+1,b}}(\psi_{\ell+1,b})}^2.
\end{align*}
Note that $\prod_{r=2}^m\frac{2n-1+2\ell+r}{r}=\frac{(2n-1+2\ell+m)!}{m!(2n+2\ell)!}$ is $\mc O(m^{2n-1+2\ell})$. Moreover, the dimension formula from Remark \ref{rem:dim_H} shows that $\dim V_{\ell+m,b}$ grows at most polynomially in $m$. A similar argument works for the $b$-variable.
\end{proof}

\begin{theorem}[Spectral Correspondence]\label{thm:spectral_corres_Spn}
Let $\mathbf{Ex}\ni\mu=-(\rhoa+(2\ell-2)\alpha),\ \ell\in\N_0,$ be an exceptional parameter. Then the socle $\on{soc}(\pst{\mu})$ of $\pst{\mu}$ is irreducible, unitary and its $K$-types are given by $V_{a,b}$ for $a\geq b\geq\ell+1$. The minimal $K$-type is $V_{\ell+1,\ell+1}$ and the corresponding Poisson transform induces an isomorphism from $\colon{}^\Gamma(\on{soc}(\pst{\mu}))^{-\infty}$ onto
\begin{gather*}
{}^\Gamma\{u\in C^\infty(G\times_KV_{\ell+1,\ell+1})\colon properties\ i)-v)\ below\},
\end{gather*}
where the properties are as follows. For $u\in C^\infty(G\times_KV_{\ell+1,\ell+1})$ let $\psi_{\ell+1,\ell+1}\coloneqq\dim V_{\ell+1,\ell+1}\cdot u$ and define recursively for $a\geq b\geq\ell+1$ (see Lemma \ref{la:gen_diffops2})
\begin{align*}
\psi_{a,\ell+1}\coloneqq\frac{\dim V_{a,\ell+1}}{\dim V_{a-1,\ell+1}}\frac{1}{T^{a-1,\ell+1}_{a,\ell+1}}\mathrm{d}_{+,1}\psi_{a-1,\ell+1},\qquad \psi_{a,b}\coloneqq\frac{\dim V_{a,b}}{\dim V_{a,b-1}}\frac{1}{T^{a,b-1}_{a,b}}\mathrm{d}_{+,2}\psi_{a,b-1},
\end{align*}
where we abbreviate $T^{a_1,b_1}_{a_2,b_2}\coloneqq T_{V_{a_2,b_2}}^{V_{a_1,b_1}}(p_{V_{a_1,b_1,\mu}})(e)\in\C$. Then we define the properties
\begin{enumerate}
\item $\mathrm{d}_{+,1}\psi_{a,b}=T_{a+1,b}^{a,b}\frac{\dim V_{a,b}}{\dim V_{a+1,b}}\psi_{a+1,b},\hfill (a\geq b\geq\ell+2),$
\item $\mathrm{d}_{-,1}\psi_{a,b}=T_{a-1,b}^{a,b}\frac{\dim V_{a,b}}{\dim V_{a-1,b}}\psi_{a-1,b},\hfill (a\geq\ell+2,\ b\geq\ell+1),$
\item $\mathrm{d}_{-,2}\psi_{a,b}=T_{a,b-1}^{a,b}\frac{\dim V_{a,b}}{\dim V_{a,b-1}}\psi_{a,b-1},\hfill (a\geq b\geq\ell+2),$
\item $\mathrm{d}_{-,2}\psi_{a,\ell+1}=0,\hfill (a\geq\ell+1),$
\item $\mathrm{d}^{V_{a,b}}_V\psi_{a,b}=0,\hfill (a\geq b\geq\ell+1,\ V\lra V_{a,b},\ V\not\in\hat K_M).$
\end{enumerate}
\end{theorem}

\begin{proof}
The proof is analogous to the prove of Theorem \ref{thm:spectral_corres_SUn2}.
\end{proof}

\subsection{\texorpdfstring{The Case of $G=\GF$}{The Case of G=F4}}

By Proposition \ref{prop:tensor_decomp_eq_rk} and Remark \ref{rem:dim_F} we have for each $m,k\in\N_0$ with $m\geq k$ and $m\equiv k$ mod $2$
\begin{gather*}
V_{m,k}\otimes\mf p^*\cong V_{m+1,k+1}\oplus V_{m-1,k-1}\oplus V_{m+1,k-1}\oplus V_{m-1,k+1}\oplus\bigoplus_{\substack{\beta\in S\\V_{m,k,\beta}\not\in\hat K_M}} V_{m,k,\beta}.
\end{gather*}
We define generalized gradients $\mathrm{d}_V^{V_{m,k}}\coloneqq T_V^{V_{m,k}}\circ\nabla$ with $T_V^{V_{m,k}}\in\Hom{K}{V_{m,k}\otimes\mf p^*}{V}$ as in Definition \ref{def:omega} and abbreviate
\begin{gather*}
\mathrm{d}_{1,\pm}\coloneqq\mathrm{d}_{V_{m\pm1,k\pm1}}^{V_{m,k}},\qquad\quad\mathrm{d}_{2,\pm}\coloneqq\mathrm{d}_{V_{m\pm1,k\mp1}}^{V_{m,k}},\qquad\quad\mathrm{D}_\beta\coloneqq\mathrm{d}_{V_{m,k,\beta}}^{V_{m,k}}
\end{gather*}
for each $\beta\in S$ with $V_{m,k,\beta}\not\in\hat K_M$. Let $\mu=-(\rhoa+(2\ell-6)\alpha)\in\mathbf{Ex},\ \ell\in\N_0,$ be an exceptional parameter and recall the structure and properties of $\on{soc}(\pst{\mu})$ from Theorem \ref{thm:socle}. Using Proposition \ref{prop:gen_grad_omega}, Proposition \ref{prop:lambdaVY_BO}.\ref{it:lambdaVY_2} and Remark \ref{rem:lambda_F} we infer for each $m\geq k\in\N_0$ with $m\equiv k$ mod $2$
\begin{gather*}
V\notleftrightomega V_{m,k}\:\Longrightarrow\:\mathrm{d}_V^{V_{m,k}}\circ P_\mu^{V_{m,k}}=0\qquad\text{and}\qquad V\lraomega V_{m,k}\:\Longleftrightarrow\: V\in\{V_{m\pm1,k\pm1}\}
\end{gather*}
whenever the occurring representations exist. The minimal $K$-type of $\on{soc}(\pst{\mu})$ is given by $V_{2\ell+2,0}$ (see Theorem \ref{thm:socle}).

As in the quaternionic case we use Theorem \ref{thm:fourier_char} to prove a spectral correspondence. By Lemma \ref{la:def_distr} we obtain 
\begin{lemma}\label{la:poly_growth_F}
Let $\mu\coloneqq-(\rhoa+(2\ell-6)\alpha),\ \ell\in\N_0,$ an exceptional parameter. Let $\psi_{m,k}\in C^\infty(G\times_KV_{m,k})$ for $m\equiv k\on{ mod }2,\ m-k\geq2(\ell+1),$ be such that the equations from Lemma \ref{la:gen_diffops2} are satisfied (with $\psi_{m,k}$ instead of $\pi_{V_{m,k}}^*(f)$). Assume that $\pi_{V_{2\ell+2,0}}(\psi_{V_{2\ell+2,0}})\in C^\infty(G)$ has finite $L^2$-norm.
Then the formal sum
\begin{gather*}
f\coloneqq\sum_{\substack{m-k\geq2\ell+2\\m\equiv k\on{ mod }2}}\iota_{G/M}(\pi_{V_{m,k}}(\psi_{m,k}))
\end{gather*}
defines a distribution on $G/M$.
\end{lemma}

\begin{proof}
We abbreviate $T^{m_1,k_1}_{m_2,k_2}\coloneqq T_{V_{m_2,k_2}}^{V_{m_1,k_1}}(p_{V_{m_1,k_1,\mu}})(e)\in\C$. It suffices to prove the estimate in Lemma \ref{la:def_distr}. Using Lemma \ref{la:gen_grad_adjoint} (second step) and the equations from Lemma \ref{la:gen_diffops2} (first and third step) we infer for the $L^2$-norm as in Lemma \ref{la:poly_growth_SUn}
\begin{align*}
\norm{\pi_{V_{m,k}}(\psi_{m,k})}^2&=
-\left(\frac{\dim V_{m,k}}{\dim V_{m-1,k-1}}\right)^2\frac{T^{m,k}_{m-1,k-1}}{T^{m-1,k-1}_{m,k}}\norm{\pi_{V_{m-1,k-1}}(\psi_{m-1,k-1})}^2.
\end{align*}
By Equation \eqref{eq:T_decomp}, Proposition \ref{prop:nu_explicit} and Proposition \ref{prop:lambdaVY_BO}.\ref{it:lambdaVY_4} we have
\begin{gather*}
\frac{T^{m,k}_{m-1,k-1}}{T^{m-1,k-1}_{m,k}}=\frac{-14-2\ell-m-k}{4-2\ell+m+k}\frac{\lambda(V_{m,k},V_{m-1,k-1})}{\lambda(V_{m,k},V_{m-1,k-1})}=\frac{-14-2\ell-m-k}{4-2\ell+m+k}\frac{\dim V_{m-1,k-1}}{\dim V_{m,k}}
\end{gather*}
and thus
\begin{gather*}
\norm{\pi_{V_{m,k}}(\psi_{m,k})}^2=\frac{14+2\ell+m+k}{4-2\ell+m+k}\frac{\dim V_{m,k}}{\dim V_{m-1,k-1}}\norm{\pi_{V_{m-1,k-1}}(\psi_{m-1,k-1})}^2.
\end{gather*}
Iteratively applying this equation we infer for $a(m,k)\coloneqq\frac{m+k}{2}$ and $p\coloneqq a(m,k)-(\ell+1)$
\begin{align*}
\norm{\pi_{V_{m,k}}(\psi_{m,k})}^2&=\prod_{r=1}^{p-1}\frac{7+\ell+a(m,k)-r}{2-\ell+a(m,k)-r}\frac{\dim V_{m-r,k-r}}{\dim V_{m-r-1,k-r-1}}\norm{\pi_{V_{m-p,k-p}}(\psi_{m-p,k-p})}^2\\
&=\frac{\dim V_{m,k}}{\dim V_{m-p,k-p}}\prod_{r=1}^{p-1}\frac{7+\ell+a(m,k)-r}{2-\ell+a(m,k)-r}\norm{\pi_{V_{m-p,k-p}}(\psi_{m-p,k-p})}^2,
\end{align*}
with $a(m-p,k-p)=\ell+1$. Note that $\prod_{r=1}^{p-1}\frac{7+\ell+a(m,k)-r}{2-\ell+a(m,k)-r}=\frac{(7+2\ell+p)!\cdot 6}{(7+2\ell+1)!\cdot(2+p)!}$ is $\mc O(p^{7+2\ell-2})$. Moreover, the dimension formula from Remark \ref{rem:dim_F} shows that $\dim V_{m,k}$ grows at most polynomially in $m$ and $k$. A similar argument works for the step from $V_{m,k}$ with $a(m,k)=\ell+1$ to $V_{2(\ell+1),0}$ by decreasing $b(m,k)\coloneqq\frac{m-k}{2}$ (by going from $V_{m,k}$ to $V_{m-1,k+1}$).
\end{proof}

\begin{theorem}[Spectral Correspondence]\label{thm:spectral_corres_F4}
Let $\mathbf{Ex}\ni\mu=-(\rhoa+(2\ell-6)\alpha),\ \ell\in\N_0,$ be an exceptional parameter. Then the socle $\on{soc}(\pst{\mu})$ of $\pst{\mu}$ is irreducible, unitary and its $K$-types are given by $V_{m,k}$ for $m\equiv k\on{ mod }2,\ m-k\geq2(\ell+1)$. The minimal $K$-type is $V_{2\ell+2,0}$ and the corresponding Poisson transform induces an isomorphism from $\colon{}^\Gamma(\on{soc}(\pst{\mu}))^{-\infty}$ onto
\begin{gather*}
{}^\Gamma\{u\in C^\infty(G\times_KV_{2\ell+2,0})\colon properties\ i)-v)\ below\},
\end{gather*}
where the properties are as follows. Let $a(m,k)\coloneqq\frac{m+k}{2}$ and $b(m,k)\coloneqq\frac{m-k}{2}$. For $u\in C^\infty(G\times_KV_{2\ell+2,0})$ let $\psi_{\ell+1,\ell+1}\coloneqq\dim V_{2\ell+2,0}\cdot u$ and define recursively for $m\equiv k\on{ mod }2,\ m-k\geq2(\ell+1)$ (see Lemma \ref{la:gen_diffops2})
\begin{align*}
\psi_{a,\ell+1}\coloneqq\frac{\dim V_{m,m-2\ell-2}}{\dim V_{m-1,m-2\ell-3}}\frac{1}{T^{a-1,\ell+1}_{a,\ell+1}}\mathrm{d}_{+,1}\psi_{a-1,\ell+1},\ \psi_{a,b}\coloneqq\frac{\dim V_{m,k}}{\dim V_{m-1,k+1}}\frac{1}{T^{a,b-1}_{a,b}}\mathrm{d}_{+,2}\psi_{a,b-1},
\end{align*}
where we abbreviate $T^{a_1,b_1}_{a_2,b_2}\coloneqq T_{V_{a_2+b_2,a_2-b_2}}^{V_{a_1+b_1,a_1-b_1}}(p_{V_{a_1+b_1,a_1-b_1,\mu}})(e)\in\C$. Then we define the properties
\begin{enumerate}
\item $\mathrm{d}_{+,1}\psi_{a,b}=T_{a+1,b}^{a,b}\frac{\dim V_{a,b}}{\dim V_{a+1,b}}\psi_{a+1,b},\hfill (a\geq b\geq\ell+2),$
\item $\mathrm{d}_{-,1}\psi_{a,b}=T_{a-1,b}^{a,b}\frac{\dim V_{a,b}}{\dim V_{a-1,b}}\psi_{a-1,b},\hfill (a\geq\ell+2,\ b\geq\ell+1),$
\item $\mathrm{d}_{-,2}\psi_{a,b}=T_{a,b-1}^{a,b}\frac{\dim V_{a,b}}{\dim V_{a,b-1}}\psi_{a,b-1},\hfill (a\geq b\geq\ell+2),$
\item $\mathrm{d}_{-,2}\psi_{a,\ell+1}=0,\hfill (a\geq\ell+1),$
\item $\mathrm{d}^{V_{a,b}}_V\psi_{a,b}=0,\hfill (a\geq b\geq\ell+1,\ V\lra V_{m,k},\ V\not\in\hat K_M).$
\end{enumerate}
\end{theorem}

\begin{proof}
The proof is analogous to the prove of Theorem \ref{thm:spectral_corres_SUn2}.
\end{proof}

\appendix 
\setcounter{equation}{0}
\renewcommand{\theequation}{\thesection.\arabic{equation}}
\section{Computations of scalars relating Poisson transforms}\label{app:comp_scalars}
\begin{table}[ht]
\renewcommand{\arraystretch}{1.3}
\centering
    \begin{tabular}{C|C|C|C|C|C}
        G & K & K/M & m_\alpha & m_{2\alpha} & \rhoa(H)\\
        \hline
        \GSO{n},\ n\geq 2 & \mathrm{S}(\mathrm{O}(n)\times \mathrm{O}(1))\cong\mathrm{SO}(n) & \mathbb{S}^{n-1} & n-1 & 0 & \frac{n-1}{2}\\
        \GSU{n},\ n\geq 2 & \mathrm{S}(\mathrm{U}(n)\times\mathrm{U}(1))\cong\mathrm{U}(n) & \mathbb{S}^{2n-1} & 2n-2 & 1 & n\\
        \GSp{n},\ n\geq 2 & \mathrm{Sp}(n)\times\mathrm{Sp}(1) & \mathbb{S}^{4n-1} & 4n-4 & 3 & 2n+1\\
        \GF & \mathrm{Spin}(9) & \mathbb{S}^{15} & 8 & 7 & 11
    \end{tabular}
    \caption{Structural data of rank one groups (recall that $\alpha(H)=1$ for the unique simple positive restricted root $\alpha$ of $(\mf g,\mf a)$). The isomorphism of $K/M$ with a sphere is given by the adjoint action of $K$ on $H\in\mf a_0\subseteq\mf p$.}
    \label{table:structure_rank_one}
\end{table}
In order to compute the scalars $T_V^Y(p_{Y,\mu})$ occurring in Lemma \ref{la:gen_diffops} we first compute the scalars $\lambda(V,Y)$ in each case and then conclude by using Lemma \ref{la:nu_comp_series} and Equation \eqref{eq:T_decomp}. For the explicit calculations we will use \emph{hypergeometric functions}.

\begin{definition}	
  The (Gaussian, ordinary) \emph{hypergeometric function} $F$ (of type $(2,1)$) is defined by
  $$F(a,b,c,z)\coloneqq\sum_{n=0}^\infty \frac{(a)_n(b)_n}{(c)_n}\frac{z^n}{n!},$$
  where $a,b,c,z\in\R,c>0$ if the series converges where 
  $$(q)_n\coloneqq
  \begin{cases}
    1&:n=0\\
    q(q+1)\ldots(q+n-1)&:n>0
  \end{cases}$$
  denotes the \emph{Pochhammer symbol}. Note that $F$ is a polynomial in $z$ if $a$ or $b$ is a non-positive integer. 
 \end{definition}
\begin{lemma}(cf. \cite[Lemma 4.1]{JW})\label{hypergeo} Assume $\abs{z}<1$ or $a\in-\N_0$ or $b\in-\N_0$. Then $F$ has the following properties:
 \begin{enumerate}[label=(\roman*), ref=(\roman*)]
  \item \label{prop1}$\frac{d}{dz}F(a,b,c,z)=\frac{ab}{c}F(a+1,b+1,c+1,z)$,
  \item \label{prop2}$(c-b-a)F(a,b,c,z)=(c-b)F(a,b-1,c,z)+a(z-1)F(a+1,b,c,z)$,
  \item \label{prop2'}$(c-b-a)F(a,b,c,z)=(c-a)F(a-1,b,c,z)+b(z-1)F(a,b+1,c,z)$,
  \item \label{prop3}$F(a,b+1,c,z)-F(a,b,c,z)=\frac{az}{c}F(a+1,b+1,c+1,z)$,
  \item \label{prop3'}$F(a+1,b,c,z)-F(a,b,c,z)=\frac{bz}{c}F(a+1,b+1,c+1,z)$.
 \end{enumerate}
\end{lemma}
\subsection{\texorpdfstring{The Case of $G=\GSO{n},\ n\geq 3$}{The Case of G=SO(n,1), n>=3}}
Considering the compact picture and the isomorphism $K/M\cong\mb S^{n-1}$ we see that $\pst{\mu}$ decomposes as the Hilbert space direct sum
\begin{gather}\label{eq:app_KM_sphere_iso_R}
\pst{\mu}\cong_KL^2(K/M)\cong_KL^2(\mb S^{n-1})\cong_K\widehat\bigoplus_{\ell\in\mathbb{N}_0}Y_\ell,
\end{gather}
where $Y_\ell$ denotes the space of all harmonic, homogeneous polynomials of degree $\ell$ restricted to $\mb S^{n-1}$.
\begin{remark}
For $G=\GSO{2}$ we have $\pst{\mu}\cong_K\widehat\bigoplus_{\ell\in\Z}Y_{\ell}$, with $Y_{\ell}\coloneqq\C\cdot z^\ell\subset C^\infty(\mb S^1)$.
\end{remark}
 We choose a Cartan subalgebra $\mf t$ of $\mf k$ as in \cite[§II.1 Example 2, 4]{knapplie} with roots
\begin{gather*}
\Delta_\mf k=\{\pm e_i\pm e_j\colon 1\leq i\neq j\leq m\}\cup\{\pm e_i\colon 1\leq i\leq m\}\text{ resp.\@ }\Delta_\mf k=\{\pm e_i\pm e_j\colon 1\leq i\neq j\leq m\}
\end{gather*}
if $K\cong\mathrm{SO}(2m+1)$ resp.\@ $K\cong\mathrm{SO}(2m)$ for some $m\in\N$. We choose the positive systems
\begin{gather*}
\Delta^+_\mf k=\{e_i\pm e_j\colon 1\leq i<j\leq m\}\cup\{e_i\colon 1\leq i\leq m\}\text{ resp.\@ }\Delta^+_\mf k=\{e_i\pm e_j\colon 1\leq i<j\leq m\}.
\end{gather*}
The corresponding half sum of positive roots is given by
\begin{gather*}
\rhoc=\left(m-\frac12\right)e_1+\left(m-\frac32\right)e_2+\ldots+\frac12 e_m\text{ resp.\@ }\rhoc=(m-1)e_1+\ldots+e_{m-1}.
\end{gather*}
The highest weight of $Y_\ell$ is $\ell e_1$ (see e.g.\@ \cite[Example 1 of §V.1, p. 277]{knapplie}). Introducing the angular coordinates
\begin{gather*}
x_1=r\cos(\xi),\quad x_i=r\sin(\xi)\omega_i,\ i\geq2,
\end{gather*}
where $\sum_{i=2}^n\omega_i^2=1,\ 0\leq\xi\leq\pi$, we infer by \cite[Theorem 3.1(2)]{JW} that
\begin{gather*}
 \phi_{Y_k}=\cos^k(\xi) F\left(-\frac{k}{2},\frac{1-k}{2},\frac{n-1}{2},-\tan^2(\xi)\right).
\end{gather*}
In order to compute the scalars $\lambda(V,Y)$ for $Y,\,V\in\hat K_M=\{[Y_\ell]\colon\ell\in\N_0\}$ it suffices to decompose $\omega(H)\phi_V$ by Lemma \ref{la:lambda_explicit}.

\begin{lemma}\label{la:decomp_omega_R}
For each $k\in\N_0$ we have
\begin{gather*}
\omega(H)\phi_{Y_k}=\frac{k}{n+2k-2}\phi_{Y_{k-1}}+\frac{n+k-2}{n+2k-2}\phi_{Y_{k+1}}.
\end{gather*}
\end{lemma}

\begin{proof}
Recall that the identification from Equation \eqref{eq:app_KM_sphere_iso_R} comes from the $K$-action on $\mf p$, where $e_1\in\mb S^{n-1}$ corresponds to $H\in\mf a$. This implies that
\begin{gather*}
\omega(H)=x_1=\cos(\xi)
\end{gather*}
as a function in $C^\infty(\mb S^{n-1})$. Therefore,
\begin{gather*}
\omega(H)\phi_{Y_k}=\cos^{k+1}(\xi) \hyp{-\frac{k}{2}}{\frac{1-k}{2}}{\frac{n-1}{2}}{-\tan^2\xi}.
\end{gather*}
By Lemma \ref{hypergeo}.\ref{prop2} with $a=-\frac{k}{2},\ b=\frac{1-k}{2},\ c=\frac{n-1}{2}$ and $z=-\tan^2\xi$ we infer that $(n+2k-2)\hyp{-\frac{k}{2}}{\frac{1-k}{2}}{\frac{n-1}{2}}{z}$ equals
\begin{gather*}
(n+k-2)\hyp{-\frac{k+1}{2}}{-\frac{k}{2}}{\frac{n-1}{2}}{z}+\frac{k}{\cos^2\xi}\hyp{\frac{1-k}{2}}{\frac{2-k}{2}}{\frac{n-1}{2}}{z}.
\end{gather*}
Multiplying by $\cos^{k+1}\xi$ yields the result.
\end{proof}

\begin{remark}\label{rem:lambda_R}
Note that Lemma \ref{la:lambda_explicit} implies that
\begin{gather*}
\lambda(Y_k,Y_{k+1})=\on{pr}_{Y_{k+1}}(\omega(H)\phi_{Y_k})(eM)=\frac{n+k-2}{n+2k-2}\phi_{Y_{k+1}}(eM)=\frac{n+k-2}{n+2k-2}.
\end{gather*}
Similarly, we have $\lambda(Y_k,Y_{k-1})=\frac{k}{n+2k-2}$. The scalars $T_{Y_{k\pm 1}}^{Y_k}(p_{Y_k,\mu})(e)$ will be computed in Proposition \ref{prop:nu_explicit}.
\end{remark}

In order to describe the generalized gradients properly we will now decompose the relevant tensor products.
\begin{proposition}\label{prop:tensor_decomp_odd}
    Let $K=\SO{2m+1},\ m\geq1$. For $m>1$ the tensor product $Y_k\otimes\mf p^*$ decomposes for $k\in\mathbb{N}$ into
    \begin{gather*}
        Y_k\otimes\mf p^*\cong Y_{k-1}\oplus Y_{k+1}\oplus V_k
    \end{gather*}
    where $V_k$ is the $K$-representation with highest weight $k e_1+ e_2$. Moreover we have $Y_k\otimes\mf p^*\cong Y_{k-1}\oplus Y_k\oplus Y_{k+1}$ if $m=1$.
\end{proposition}
\begin{proof}
The coadjoint representation of $K$ on $\mf p^*\cong \C^{2m+1}$ is equivalent to the defining representation (as well as $Y_1$) and has weights $\pm e_i,\ i\in\{1,\ldots,m\},$ and $0$. Writing 
\begin{gather*}
Y_k\otimes\mf p^*\cong Y_k\otimes Y_1\cong\bigoplus_{\Lambda_i\in\hat K}\mathcal{L}_i\Lambda_i,
\end{gather*}
where $\mc{L}_i\coloneqq\on{mult}(\Lambda_i, Y_k\otimes Y_1)$ denotes the multiplicity, we have by \cite[p.274]{FuchsSchweigert}
\begin{gather*}
 \mc{L}_i=\sum_{w\in W}\on{sign}(w)\on{mult}_{Y_1}(w(\Lambda_i+\rhoc)-\rhoc-k e_1),
\end{gather*}
    where $\on{mult}_{Y_1}(\mu)\in\mathbb N_0$ denotes the multiplicity of the weight $\mu$ in $Y_1$ and $W$ denotes the Weyl group of $\mf k$. If $\mc L_i\neq0$ there has to exist some $w\in W$ such that $w(\Lambda_i+\rhoc)-\rhoc-k e_1$ is a weight of $Y_1$, i.e.
    \begin{gather*}
    	w(\Lambda_i+\rhoc)-\rhoc-k e_1=\pm e_j\Leftrightarrow\Lambda_i=w^{-1}(\rhoc+k e_1\pm e_j)-\rhoc
	\end{gather*}     
	for some $j\in\{1,\ldots,m\}$ or
	 \begin{gather*}
    	w(\Lambda_i+\rhoc)-\rhoc-k e_1=0\Leftrightarrow\Lambda_i=w^{-1}(\rhoc+k e_1)-\rhoc.
	\end{gather*} 
	Let us first consider the case $m\neq 1$. Since $\Lambda_i$ is a highest weight it is dominant. Thus, $\rhoc+k e_1\pm e_j$ resp. $\rhoc+k e_1$ must not lie on the boundary of any Weyl chamber. This is the case if and only if the weight of $Y_1$ is contained in $\{0,\pm e_1,  e_2,- e_m\}$. In the first three cases we obtain for $\Lambda_i+\rhoc$
	\begin{align*}
w^{-1}(\rhoc+k e_1)&=w^{-1}\left(\left(k+m-\frac12\right) e_1+\left(m-\frac32\right) e_2+\ldots+\frac12 e_m\right)\\
w^{-1}(\rhoc+k e_1\pm e_1)&=w^{-1}\left(\left(k\pm1+m-\frac12\right) e_1+\left(m-\frac32\right) e_2+\ldots+\frac12 e_m\right)\\
w^{-1}(\rhoc+k e_1+ e_2)&=w^{-1}\left(\left(k+m-\frac12\right) e_1+\left(m-\frac12\right) e_2+\ldots+\frac12 e_m\right)
	\end{align*}
	which is dominant if and only if $w=id$ yielding $\Lambda_i=k e_1,(k\pm1) e_1, k e_1+ e_2$ respectively. For $\Lambda_i+\rhoc=w^{-1}(\rhoc+k e_1- e_m)$ we have 
	\begin{gather*}
	\Lambda_i+\rhoc=w^{-1}\left(\left(k+m-\frac12\right) e_1+\left(m-\frac32\right) e_2+\ldots+\frac32 e_{m-1}-\frac12 e_m\right)
	\end{gather*}
	which is dominant if and only if $w=s_{ e_m}$ is the reflection along $ e_m$. For this $w$ we have $\Lambda_i=k e_1$. Altogether we have
	\begin{align*}
	\on{mult}(k e_1,Y_k\otimes Y_1)&=\sum_{w\in W}\on{sign}(w)\on{mult}_{Y_1}(w(\Lambda_i+\rhoc)-\rhoc-k e_1)\\
	&=\on{sign}(id)\on{mult}_{Y_1}(0)+\on{sign}(s_{ e_m})\on{mult}_{Y_1}(- e_m)=0
	\end{align*}
	and similarly that the representations with highest weights $(k\pm1) e_1$ resp. $k e_1+ e_2$ occur with multiplicity one. For $m=1$ the weights of $Y_1$ are $-e_1,0$ and $e_1$. We get $\Lambda_i=(k-1)e_1,\, ke_1$ resp.\@ $(k+1)e_1$ in this case, each with multiplicity one.
\end{proof}

\begin{proposition}\label{prop:tensor_decomp_even}
    Let $K=\SO{2m},\ m\geq2$. The tensor product $Y_k\otimes\mf p^*$ decomposes for $k\in\mathbb{N}$ into
    \begin{gather*}
        Y_k\otimes\mf p^*\cong Y_{k-1}\oplus Y_{k+1}\oplus V_k,
    \end{gather*}
    where $V_k$ is the $K$-representation with highest weight $k e_1+ e_2$.
\end{proposition}

\begin{proof}
    The coadjoint representation of $K$ on $\mf p^*\cong \C^{2m}$ is equivalent to the defining representation (as well as $Y_1$) and has weights $\pm e_i,\ i\in\{0,\ldots,m-1\}$. Each weight occurs with multiplicity one. We can now decompose $Y_k\otimes\mf p^*\cong Y_k\otimes Y_1$ using the Racah-Speiser algorithm. Let 
    \begin{gather*}
        Y_k\otimes Y_1\cong\bigoplus_{\Lambda_i\in\hat K}\mathcal{L}_i\Lambda_i
    \end{gather*}
    with $\mc{L}_i\coloneqq\on{mult}(\Lambda_i, Y_k\otimes Y_1)=\sum_{w\in W}\on{sign}(w)\on{mult}_{Y_1}(w(\Lambda_i+\rhoc)-\rhoc-k e_1)$ as in the odd case. Since
       $ w(\Lambda_i+\rhoc)-\rhoc-k e_1=\pm e_i\Leftrightarrow \Lambda_i=w^{-1}(\rhoc+k e_1\pm e_i)-\rhoc$
    has to be dominant (since $\Lambda_i$ is a highest weight) the weight $\rhoc+k e_1\pm e_i$ must not lie on the boundary of any Weyl chamber. This is the case if and only if the weight $\pm e_i$ is $\pm e_1$ or $ e_2$. In these cases the weight $\rhoc+k e_1\pm e_i$ is dominant, so $w=id$. Moreover, the weight $w^{-1}(\rhoc+k e_1\pm e_i)-\rhoc$ is given by $k e_1\pm e_1=(k\pm1) e_1$ resp. $k e_1+ e_2$.
\end{proof}

\begin{remark}\label{rem:dim_Yl}
 Using the Weyl dimension formula we see that 
 \begin{gather*}
 	\dim Y_k = \binom{n+k-3}{k}\frac{\frac{n}2+k-1}{\frac{n}2-1}=\binom{n+k-3}{k}\frac{n+2k-2}{n-2}.
 \end{gather*}
\end{remark}

\subsection{\texorpdfstring{The Case of $G=\GSU{n},\ n\geq 2$}{The Case of G=SU(n,1), n>=2}}
Using the isomorphism $K/M\cong\mb S^{2n-1}$ we see that $\pst{\mu}$ decomposes as the Hilbert space direct sum
\begin{gather}\label{eq:app_KM_sphere_iso_C}
\pst{\mu}\cong_KL^2(K/M)\cong_KL^2(\mb S^{2n-1})\cong_K\widehat\bigoplus_{p,q\in\mathbb{N}_0}Y_{p,q},
\end{gather}
where
\begin{gather}\label{eq:sph_harm_Ypq_SUn}
 Y_{p,q}\coloneqq\{f\in Y_{p+q}\colon f(\alpha z)=\alpha^{p}\overline{\alpha}^qf(z)\ \forall\alpha\in\C,\abs{\alpha}=1,z\in\mathbb S^{2n-1}\}
\end{gather}
with $f(z)\coloneqq f(\on{Re}(z_1), \on{Im}(z_1),\ldots,\on{Re}(z_n),\on{Im}(z_n))$. 
Let $\mf{t}_0$ denote the diagonal matrices in $\mf{su}(n,1)$. Then $\mf{t}_0=\mf{z}(\mf{k}_0)\oplus\mf{h}_0$ where $\mf{h}_0$ is a Cartan subalgebra of $[\mf{k}_0,\mf{k}_0]\cong\mf{su}(n)$ (traceless diagonal matrices). Denoting the dual basis of the standard diagonal matrix basis $E_{ii},\ 1\leq i\leq n+1,$ by $(e_i)_i$ we obtain that the roots $\Delta_\mf k$ of $(\mf k,\mf t)$ resp.\@ $\Delta$ of $(\mf g,\mf t)$ are given by 
\begin{gather}\label{eq:roots_C}
\Delta_\mf k=\{e_i-e_j\colon 1\leq i\neq j\leq n\}\text{ resp.\@ }\Delta=\{e_i-e_j\colon 1\leq i\neq j\leq n+1\}.
\end{gather}
We choose the positive system $\Delta^+_\mf k=\{e_i-e_j\colon 1\leq i<j\leq n\}$ with
\begin{gather*}
\rhoc=\left(\frac{n-1}{2}\right)e_1+\left(\frac{n-3}{2}\right)e_2+\ldots-\frac{n-1}{2}e_n.
\end{gather*}
The highest weight of $Y_{p,q}$ is given by $q e_1-p e_n+(p-q)e_{n+1}$ (see e.g.\@ \cite[Example 1 of §V.1, p. 276]{knapplie}, the $e_{n+1}$-part accounts for the trivial action of the center). Introducing the angular coordinates (on $\C^n\cong\mb{R}^{2n}$)
\begin{gather*}
z_1=r\cos(\xi)e^{i\varphi},\quad z_j=r\sin(\xi)\omega_j,\ 2\leq j\leq n
\end{gather*}
where $\sum_{j=2}^n\abs{\omega_j}^2=1,\ 0\leq\varphi\leq 2\pi$ and $0\leq\xi\leq\frac{\pi}{2}$ we have (see \cite[Theorem 3.1(3)]{JW})
\begin{gather*}
\phi_{Y_{p,q}}=e^{i(p-q)\varphi}\cos^{p+q}(\xi)F(-p,-q,n-1,-\tan^2(\xi)).
\end{gather*}

\begin{lemma}\label{la:decomp_omega_C}
For each $p,q\in\N_0$ we have
\begin{align*}
2(p+q+n-1)\omega(H)\phi_{Y_{p,q}}&=(p+n-1)\phi_{Y_{p+1,q}}+q\phi_{Y_{p,q-1}}\\
&\quad+(q+n-1)\phi_{Y_{p,q+1}}+p\phi_{Y_{p-1,q}}.
\end{align*}
\end{lemma}

\begin{proof}
Write $\phi_{Y_{p,q}}=e^{i(p-q)\varphi}h_{p,q}(\xi)$. In the angular coordinates introduced above we have
\begin{gather*}
\omega(H)=\on{Re}(z_1)=\cos(\xi)\cos(\varphi)
\end{gather*}
as a function in $C^\infty(\mb S^{2n-1})$. Therefore,
\begin{align}
\nonumber\omega(H)\phi_{Y_{p,q}}&=\cos(\xi)\cos(\varphi)e^{i(p-q)\varphi}h_{p,q}(\xi)\\
\label{eq:hgC1}&=\frac{\cos(\xi)h_{p,q}(\xi)}{2}e^{i(p-q+1)\varphi}+\frac{\cos(\xi)h_{p,q}(\xi)}{2}e^{i(p-q-1)\varphi}.
\end{align}
Lemma \ref{hypergeo}.\ref{prop2'} implies that
\begin{gather}
\label{eq:hgC2}\cos(\xi)h_{p,q}(\xi)=\frac{p+n-1}{p+q+n-1}h_{p+1,q}(\xi)+\frac{q}{p+q+n-1}h_{p,q-1}(\xi)
\end{gather}
and Lemma \ref{hypergeo}.\ref{prop2} implies that
\begin{gather}
\label{eq:hgC3}\cos(\xi)h_{p,q}(\xi)=\frac{q+n-1}{p+q+n-1}h_{p,q+1}(\xi)+\frac{p}{p+q+n-1}h_{p-1,q}(\xi).
\end{gather}
Combining the equations \eqref{eq:hgC1}, \eqref{eq:hgC2} and \eqref{eq:hgC3} yields the result.
\end{proof}

\begin{remark}\label{rem:lambda_C}
As in Remark \ref{rem:lambda_R}, Lemma \ref{la:decomp_omega_C} determines the scalars $\lambda(Y_{p,q},V)$ for each $V\in\hat K_M$ with $V\lra Y_{p,q}$.
\end{remark}

To decompose the relevant tensor products we use Proposition \ref{prop:tensor_decomp_eq_rk}. By Equation \eqref{eq:roots_C} we infer that the non-compact roots are given by 
\begin{gather*}
\Delta_n=\{\pm(e_i-e_{n+1})\colon1\leq i\leq n\}.
\end{gather*}
The following remark ensures that each representation $Y_{\tau,\beta},\ \beta\in S,$ in Proposition \ref{prop:tensor_decomp_eq_rk} actually occurs.
\begin{remark}\label{rem:dim_C}
Using the Weyl dimension formula we see that 
\begin{align*}
\dim Y_{p,q}&=\binom{q+n-2}{n-2}\binom{p+n-2}{n-2}\frac{n+p+q-1}{n-1}=\dim Y_{q,p},\\
\dim Y_{p,q,-e_{n-1}+e_{n+1}}&=\binom{q+n-1}{q}\binom{p+n-2}{p}\frac{(n+p+q-1)p(n-2)}{(n+q-2)(p+1)}=\dim Y_{q,p,e_2-e_{n+1}}.
\end{align*}
For $n=2$ this has to be read as $\dim Y_{p,0,-e_1+e_3}=p=\dim Y_{0,p,e_2-e_{3}}$. We get that
\begin{gather*}
\sum_{\beta\in S\subseteq\Delta_n}\dim Y_{p,q,\beta}=\dim\mf p\cdot\dim Y_{p,q}=2n\cdot\dim Y_{p,q},
\end{gather*}
which implies that $m(\beta)=1$ if and only if the corresponding formula for the dimension of $Y_{p,q,\beta}$ in not zero.
\end{remark}

\subsection{\texorpdfstring{The Case of $G=\GSp{n},\ n\geq 2$}{The Case of G=Sp(n,1), n>=2}}
In this case we have $K=\mathrm{Sp}(n)\times\mathrm{Sp}(1)$ and $\mf g=\mf{sp}(n,1)_\C=\mf{sp}(n+1,\C)$. We choose a Cartan subalgebra of $\mf{sp}(n,\C)\times\mf{sp}(1,\C)$ and introduce notation as in \cite[§II.2 Ex.\@ 3]{knapplie} such that we have for the roots $\Delta_{\mf k}$ of $(\mf k,\mf h)$ resp.\@ $\Delta$ of $(\mf g,\mf h)$
\begin{align}
\nonumber\Delta_{\mf k}&=\{\pm e_i\pm e_j\colon 1\leq i\neq j\leq n\}\cup\{\pm 2e_i\colon 1\leq i\leq n+1\}\\
\Delta&=\{\pm e_i\pm e_j\colon 1\leq i\neq j\leq n+1\}\cup\{\pm 2e_i\colon 1\leq i\leq n+1\}.\label{eq:roots_H}
\end{align}
We choose the positive system
\begin{align*}
\Delta^+_{\mf k}&=\{e_i\pm e_j\colon 1\leq i<j\leq n\}\cup\{2e_i\colon 1\leq i\leq n+1\}.
\end{align*}
The corresponding half sum of positive roots is given by
\begin{gather*}
\rhoc=ne_1+(n-1)e_2+\ldots+2e_{n-1}+e_n+e_{n+1}.
\end{gather*}
By the isomorphism $K/M\cong\mb S^{4n-1}$ and \cite[Ch.\@ IX.8, Problem 12]{knapplie} we see that $\pst{\mu}$ decomposes as the Hilbert space direct sum
\begin{gather}\label{eq:app_KM_sphere_iso_H}
\pst{\mu}\cong_KL^2(K/M)\cong_KL^2(\mb S^{4n-1})\cong_K\widehat\bigoplus_{a\geq b\geq0}V_{a,b},
\end{gather}
where $V_{a,b}$ has highest weight $ae_1+be_2+(a-b)e_{n+1}$. We now introduce angular coordinates on $\H^n\cong\mb{R}^{4n}$ as in \cite[Theorem 3.1(4)]{JW}. For $(w_1,\ldots,w_n)\in\H^n$ we write
\begin{gather*}
w_1=r\cos(\xi)(\cos(t)+y\sin(t)),\quad w_i=r\sin(\xi)q_i,\ i\geq 2
\end{gather*}
where $q_i,y\in\H$ such that $\abs{y}^2=1=\sum_{i=2}^n\abs{q_i}^2,\ \on{Re}(y)=0$ and $0\leq\xi\leq\frac{\pi}{2},\ 0\leq t\leq\pi$.
Then we have by \cite[Theorem 3.1(4)]{JW}\footnote{There is a sign error in \cite[Theorem 3.1(4)]{JW}; solving the differential equation in \cite[p.147]{JW} actually gives
$\frac{\sin((q+1)t)}{\sin(t)}\cos^p(\xi)F\left(\frac{-p+q}{2},-\frac{p+q+2}{2},2(n-1),-\tan^2(\xi)\right)$.} (our $V_{a,b}$ corresponds to $V^{p,q}$ of \cite{JW} with $p\coloneqq a+b$ and $q\coloneqq a-b$ by \cite[Lemma 3.3]{JW})
\begin{gather*}
\phi_{V_{a,b}}=\frac{1}{a-b+1}\frac{\sin((a-b+1)t)}{\sin(t)}\cos^{a+b}(\xi)F\left(-b,-(a+1),2(n-1),-\tan^2(\xi)\right),
\end{gather*}
where the normalizing factor $\frac{1}{a-b+1}$ follows from $\phi_{V_{a,b}}(eM)=1$, where $eM$ corresponds to $t=\xi=0$, and using $\lim_{t\to0}\frac{\sin((a-b+1)t)}{\sin(t)}=a-b+1$.

\begin{lemma}\label{la:decomp_omega_H}
For $a,b\in\N_0$ with $a\geq b$ we have
\begin{align*}
2(a-b+1)(2n-1+a+b)\omega(H)\phi_{V_{a,b}}&=(a-b+2)(2n-1+a)\phi_{V_{a+1,b}}\\
&\quad+b(a-b+2)\phi_{V_{a,b-1}}\\
&\quad+(a-b)(2n-2+b)\phi_{V_{a,b+1}}\\
&\quad+(a-b)(a+1)\phi_{V_{a-1,b}}.
\end{align*}
\end{lemma}

\begin{proof}
Write $\phi_{V_{a,b}}=\frac{1}{a-b+1}\chi_q(t)h_{a,b}(\xi)$ such that $\chi_q(t)=\frac{\sin((q+1)t)}{\sin(t)}$. In the angular coordinates above we have
\begin{gather*}
\omega(H)=\on{Re}(w_1)=\cos(\xi)\cos(t)
\end{gather*}
as a function in $C^\infty(\mb S^{4n-1})$. Note that $2\cos(t)\chi_{q}(t)=\chi_{q+1}(t)+\chi_{q-1}(t)$. Therefore,
\begin{align}
\nonumber\omega(H)\phi_{V_{a,b}}&=\cos(\xi)\cos(t)\chi_q(t)h_{a,b}(\xi)\\
\label{eq:hgH1}&=\frac{\cos(\xi)h_{a,b}(\xi)}{2}\chi_{q+1}(t)+\frac{\cos(\xi)h_{a,b}(\xi)}{2}\chi_{q-1}(t).
\end{align}
Lemma \ref{hypergeo}.\ref{prop2'} implies
\begin{gather}
\label{eq:hgH2}\cos(\xi)h_{a,b}(\xi)=\frac{2n-2+b}{2n+a+b-1}h_{a,b+1}(\xi)+\frac{a+1}{2n+a+b-1}h_{a-1,b}(\xi)
\end{gather}
and Lemma \ref{hypergeo}.\ref{prop2} implies that
\begin{gather}
\label{eq:hgH3}\cos(\xi)h_{a,b}(\xi)=\frac{2n-1+a}{2n+a+b-1}h_{a+1,b}(\xi)+\frac{b}{2n+a+b-1}h_{a,b-1}(\xi).
\end{gather}
Inserting Equation \eqref{eq:hgH2} and \eqref{eq:hgH3} into Equation \eqref{eq:hgH1} proves the result.
\end{proof}

\begin{remark}\label{rem:lambda_H}
As in Remark \ref{rem:lambda_R}, Lemma \ref{la:decomp_omega_H} determines the scalars $\lambda(Y_{a,b},V)$ for each $V\in\hat K_M$ with $V\lra Y_{a,b}$.
\end{remark}

To decompose the relevant tensor products we use Proposition \ref{prop:tensor_decomp_eq_rk}. By Equation \eqref{eq:roots_H} we infer that the non-compact roots are given by 
\begin{gather*}
\Delta_n=\{\pm e_i\pm e_{n+1}\colon1\leq i\leq n\}.
\end{gather*}
The following remark ensures that each representation $Y_{\tau,\beta},\ \beta\in S,$ in Proposition \ref{prop:tensor_decomp_eq_rk} actually occurs.

\begin{remark}\label{rem:dim_H}
Using the Weyl dimension formula we see that the representation $W_{\xi_1,\xi_2,\xi_3}$ with highest weight $\xi_1e_1+\xi_2e_2+\xi_3e_{n+1}$ has dimension 
\begin{align*}
\dim W_{\xi_1,\xi_2,\xi_3}&=\frac{\xi_1+\xi_2+2n-1}{(2n-1)(2n-2)}(\xi_1-\xi_2+1)(\xi_3+1)\binom{\xi_1+2n-2}{2n-3}\binom{\xi_2+2n-3}{2n-3}
\end{align*}
and the representation $W^1_{\xi_1,\xi_2,\xi_3}$ with highest weight $\xi_1e_1+\xi_2e_2+e_3+\xi_3e_{n+1}$ has dimension
\begin{align*}
\dim W^1_{\xi_1,\xi_2,\xi_3}&=\binom{\xi_1+2n-1}{2n-3}\binom{\xi_2+2n-2}{2n-1}\frac{(\xi_1+\xi_2+2n-1)(2n-4)(\xi_1-\xi_2+1)}{2(\xi_1+2n-2)(\xi_2+2n-3)}\cdot\\
&\qquad\quad\cdot\frac{(\xi_1+1)(\xi_3+1)}{\xi_2+1}.
\end{align*}
Using these dimension formulas we get that
\begin{gather*}
\sum_{\beta\in S\subseteq\Delta_n}\dim V_{p,q,\beta}=\dim\mf p\cdot\dim V_{a,b}=4n\cdot\dim V_{a,b},
\end{gather*}
so that $m(\beta)=1$ if and only if the corresponding formula for the dimension of $V_{a,b,\beta}$ is not zero. Alternatively, the algorithm we used in the case of $\GSO{n}$ can be applied to verify this result.
\end{remark}

\subsection{\texorpdfstring{The Case of $G=\GF$}{The Case of G=F4}}
In this case we have $K=\on{Spin}(9)$ with $\mf k_0=\mf{so}(9)$ and $\on{rk}\mf g=\on{rk}\mf k=4$. Therefore, we may choose a Cartan subalgebra $\mf t$ of both $\mf k$ and $\mf g$. The root system can be realized in $V=\R^4$ with the standard basis $e_1,e_2,e_3,e_4$ in the following way (see \cite[Plate VIII]{B02})
\begin{gather}\label{eq:roots_F}
\Delta=\{\pm e_i\colon 1\leq i\leq 4\}\cup\{\pm e_i\pm e_j\colon 1\leq i<j\leq 4\}\cup\{\frac{1}{2}(\pm e_1\pm e_2\pm e_3\pm e_4)\}\\
\nonumber\Delta_\mf k=\{\pm e_i\colon 1\leq i\leq 4\}\cup\{\pm e_i\pm e_j\colon 1\leq i<j\leq 4\}.
\end{gather}
We choose the positive system $\Delta^+_\mf k=\{e_i-e_j\colon 1\leq i<j\leq 4\}\cup\{e_i\colon 1\leq i\leq 4\}$ with
\begin{gather*}
\rhoc=\frac{7}{2}e_1+\frac{5}{2}e_2+\frac{3}{2}e_3+\frac{1}{2}e_4.
\end{gather*}
By \cite[Thm.\@ 3.1]{JW2} we see that $\pst{\mu}$ decomposes as the Hilbert space direct sum
\begin{gather}\label{eq:app_KM_sphere_iso_F}
\pst{\mu}\cong_KL^2(K/M)\cong_KL^2(\mb S^{15})\cong_K\widehat\bigoplus_{\substack{m\geq\ell\geq 0\\m\equiv\ell\text{ mod }2}}V_{m,\ell},
\end{gather}
where $V_{m,\ell}$ is the $K$-representation with highest weight $\frac m2e_1+\frac\ell2 e_2+\frac\ell2 e_3+\frac\ell2 e_4$ (see \cite[p.\ 278]{JW2}).
Introducing angular coordinates on $\R^{16}$ as in \cite[p.\@ 275]{JW2} we can write (see \cite[Thm.\@ 3.1]{JW2})
\begin{gather*}
\phi_{V_{m,\ell}}=\chi_\ell(\varphi)h_{m,\ell}(\xi)
\end{gather*}
with
\begin{align*}
\chi_\ell(\varphi)&\coloneqq\cos(\varphi)^\ell\hyp{-\frac{\ell}{2}}{\frac{-\ell+1}{2}}{\frac{7}{2}}{-\tan(\varphi)^2},\\
h_{m,\ell}(\xi)&\coloneqq\cos(\xi)^m\hyp{\frac{\ell-m}{2}}{\frac{-m-\ell-6}{2}}{4}{-\tan(\xi)^2}.
\end{align*}

\begin{lemma}\label{la:decomp_omega_F}
For $m,\ell\in\N_0,\ \ell\leq m,\ m\equiv\ell\,\on{mod}\,2,$ we have
\begin{align*}
(6+2\ell)(14+2m)\omega(H)\phi_{V_{m,\ell}}&=(6+\ell)(14+m+\ell)\phi_{V_{m+1,\ell+1}}+(6+\ell)(m-\ell)\phi_{V_{m-1,\ell+1}}\\
&\qquad\quad+\ell(8+m-\ell)\phi_{V_{m+1,\ell-1}}+\ell(m+\ell+6)\phi_{V_{m-1,\ell-1}}.
\end{align*}
\end{lemma}

\begin{proof}
In the angular coordinates of \cite[p.\@ 275]{JW2} we have
\begin{gather*}
\omega(H)=x=\cos(\xi)\cos(\varphi)
\end{gather*}
as a function in $C^\infty(\mb S^{15})$. We claim that
\begin{gather}
\label{eq:hgO1}\cos(\varphi)\chi_\ell(\varphi)=\frac{6+\ell}{6+2\ell}\chi_{\ell+1}(\varphi)+\frac{\ell}{6+2\ell}\chi_{\ell-1}(\varphi).
\end{gather}
Using Lemma \ref{hypergeo}.\ref{prop2} and the symmetry of the hypergeometric function in the first two variables we infer that for $z\coloneqq-\tan(\varphi)^2$
\begin{align*}
(6+2\ell)\hyp{-\frac{\ell}{2}}{\frac{-\ell+1}{2}}{\frac72}{z}&=(6+\ell)\hyp{\frac{-(\ell+1)}{2}}{-\frac{\ell}{2}}{\frac72}{z}\\
&\quad+\frac{\ell}{\cos(\varphi)^2}\hyp{\frac{-\ell+1}{2}}{\frac{-\ell+2}{2}}{\frac72}{z}.
\end{align*}
Multiplying both sides by $\cos(\varphi)^{\ell+1}$ now proves the claim.
We now express the product $\cos(\xi)h_{m,\ell}(\xi)$ in two different forms. By Lemma \ref{hypergeo}.\ref{prop2'} we have
\begin{gather}
\label{eq:hgO2}\cos(\xi)h_{m,\ell}(\xi)=\frac{8+m-\ell}{14+2m}h_{m+1,\ell-1}(\xi)+\frac{m+\ell-6}{14+2m}h_{m-1,\ell-1}(\xi)
\end{gather}
and by Lemma \ref{hypergeo}.\ref{prop2} similarly
\begin{gather}
\label{eq:hgO3}\cos(\xi)h_{m,\ell}(\xi)=\frac{14+m+\ell}{14+2m}h_{m+1,\ell+1}(\xi)+\frac{m-\ell}{14+2m}h_{m-1,\ell+1}(\xi).
\end{gather}
Since $\omega(H)\phi_{V_{m,\ell}}=\cos(\varphi)\chi(\varphi)\cos(\xi)h_{m,\ell}(\xi)$ we arrive at the desired result by combining Equations \eqref{eq:hgO1}, \eqref{eq:hgO2} and \eqref{eq:hgO3}.
\end{proof}

\begin{remark}\label{rem:lambda_F}
As in Remark \ref{rem:lambda_R}, Lemma \ref{la:decomp_omega_F} determines the scalars $\lambda(Y_{m,\ell},V)$ for each $V\in\hat K_M$ with $V\lra Y_{m,\ell}$.
\end{remark}

To decompose the relevant tensor products we use Proposition \ref{prop:tensor_decomp_eq_rk}. By Equation \eqref{eq:roots_F} we infer that the non-compact roots are given by 
\begin{gather*}
\Delta_n=\left\{\frac{1}{2}(\pm e_1\pm e_2\pm e_3\pm e_4)\right\}.
\end{gather*}
The following remark ensures that each representation $Y_{\tau,\beta},\ \beta\in S,$ in Proposition \ref{prop:tensor_decomp_eq_rk} actually occurs.

\begin{remark}\label{rem:dim_F}
Using the Weyl dimension formula we see that the representation $W_{a_1,a_2,a_3,a_4}$ with highest weight $a_1e_1+a_2e_2+a_3e_3+a_4e_4$ has dimension 
\begin{align*}
\dim W_{a_1,a_2,a_3,a_4}&=\frac{1}{6!\cdot4!\cdot2\cdot7\cdot5\cdot3}\cdot\delta_1\cdot\delta_2\cdot\delta_3\cdot\prod_{i=1}^4(9+2(a_i-i)),
\end{align*}
with $\delta_i\coloneqq\prod_{j=i+1}^4(a_i+a_j+9-i-j)(a_i-a_j+j-i)$. Using this dimension formula we get
\begin{gather*}
\sum_{\beta\in S\subseteq\Delta_n}\dim V_{m,\ell,\beta}=\dim\mf p\cdot\dim V_{m,\ell}=16\cdot\dim V_{m,\ell},
\end{gather*}
so that $m(\beta)=1$ if and only if the corresponding formula for the dimension of $V_{m,\ell,\beta}$ is not zero. Alternatively, the algorithm we used in the case of $\GSO{n}$ can be applied to verify this result.
\end{remark}

We will now compute the scalars $T_Y^V(p_{V,\mu})(e)$ from Lemma \ref{la:gen_diffops}. Since we already computed the scalars $\lambda(V,Y)$ in each case, it suffices to determine the scalars $\nu(V,Y)$ (see Equation \eqref{eq:T_decomp} for the notation).

\begin{proposition}[Scalars between Poisson transforms]\label{prop:nu_explicit}\
\begin{enumerate}
\item $G=\GSO{n},\ n\geq 3\colon$ For $\ell\in\N_0$,
\begin{gather*}
\nu(Y_\ell,Y_{\ell+1})=\ell\lambda(Y_\ell,Y_{\ell+1}),\qquad\nu(Y_{\ell},Y_{\ell-1})=-(2\rhoa(H)+\ell-1)\lambda(Y_{\ell},Y_{\ell-1}),
\end{gather*}
\item $G=\GSU{n},\ n\geq 2\colon$ For $p,q\in\N_0$,
\begin{align*}
\nu(Y_{p,q},Y_{p+1,q})&=2p\lambda(Y_{p,q},Y_{p+1,q}),\\
\nu(Y_{p,q},Y_{p,q-1})&=-2(\rhoa(H)+q-1)\lambda(Y_{p,q},Y_{p,q-1}),\\
\nu(Y_{p,q},Y_{p,q+1})&=2q\lambda(Y_{p,q},Y_{p,q+1}),\\
\nu(Y_{p,q},Y_{p-1,q})&=-2(\rhoa(H)+p-1)\lambda(Y_{p,q},Y_{p-1,q}),
\end{align*}
\item $G=\GSp{n},\ n\geq 2\colon$ For $a,b\in\N_0$ with $a\geq b$,
\begin{align*}
\nu(V_{a,b},V_{a+1,b})&=2a\lambda(V_{a,b},V_{a+1,b}),\\
\nu(V_{a,b},V_{a,b-1})&=-(4n-2+2b)\lambda(V_{a,b},V_{a,b-1}),\\
\nu(V_{a,b},V_{a,b+1})&=2(b-1)\lambda(V_{a,b},V_{a,b+1}),\\
\nu(V_{a,b},V_{a-1,b})&=-(4n+2a)\lambda(V_{a,b},V_{a-1,b}),
\end{align*}
\item $G=\GF\colon$ For $m,\ell\in\N_0,\ \ell\leq m,\ m\equiv\ell\,\on{mod}\,2,$
\begin{align*}
\nu(V_{m,\ell},V_{m+1,\ell+1})&=(m+\ell)\lambda(V_{m,\ell},V_{m+1,\ell+1}),\\
\nu(V_{m,\ell},V_{m-1,\ell+1})&=-(14+m-\ell)\lambda(V_{m,\ell},V_{m-1,\ell+1}),\\
\nu(V_{m,\ell},V_{m+1,\ell-1})&=(m-\ell-6)\lambda(V_{m,\ell},V_{m+1,\ell-1}),\\
\nu(V_{m,\ell},V_{m-1,\ell-1})&=-(20+m+\ell)\lambda(V_{m,\ell},V_{m-1,\ell-1}).
\end{align*}
\end{enumerate}
\end{proposition}

\begin{proof}In view of Lemma \ref{la:nu_comp_series} it suffices to find a closed $G$-invariant subspace $U\leq H^\mu$, for some $\mu\in\mf a^*$, such that $\on{mult}_K(V,U)=0$ and $\on{mult}_K(Y,U)\neq0$. In this case we have $\nu(V,Y)=-(\mu+\rhoa)(H)\lambda(V,Y)$. The following table determines the Harish-Chandra module $U_K$ of $U$ in each case (see \cite[Thm.\@ 5.1]{JW} and \cite[Thm.\@ 5.2]{JW2}).
\begin{center}
\begin{tabular}{c|ccccc}
$G$&$V$& $Y$&$U_K$&$\mu(H)$&$(\mu+\rhoa)(H)$\\\hline
$\GSO{n}$&$Y_\ell$&$Y_{\ell+1}$&$\oplus_{j=\ell+1}^\infty Y_j$&$-\rhoa(H)-\ell$&$-\ell$\\
&$Y_\ell$&$Y_{\ell-1}$&$\oplus_{j=0}^{\ell-1} Y_j$&$\rhoa(H)+\ell-1$&$n+\ell-2$\\\hline
\GSU{n}&$Y_{p,q}$&$Y_{p+1,q}$&$\oplus_{p'\geq p+1,q'\geq 0}Y_{p',q'}$&$-2p-\rhoa(H)$&$-2p$\\
&$Y_{p,q}$&$Y_{p,q-1}$&$\oplus_{p'\geq 0,q'\leq q-1}Y_{p',q'}$&$\rhoa(H)+2(q-1)$&$2(n+q-1)$\\
&$Y_{p,q}$&$Y_{p,q+1}$&$\oplus_{p'\geq 0,q'\geq q+1}Y_{p',q'}$&$-2q-\rhoa(H)$&$-2q$\\
&$Y_{p,q}$&$Y_{p-1,q}$&$\oplus_{p'\leq p-1,q'\geq 0}Y_{p',q'}$&$\rhoa(H)+2(p-1)$&$2(n+p-1)$\\\hline
\GSp{n}&$V_{a,b}$&$V_{a+1,b}$&$\oplus_{a'\geq a+1, a'\geq b'}V_{a',b'}$&$-(\rhoa(H)+2a)$&$-2a$\\
&$V_{a,b}$&$V_{a,b-1}$&$\oplus_{b'\leq b-1, a'\geq b'}V_{a',b'}$&$\rhoa(H)+2b-4$&$4n+2(b-1)$\\
&$V_{a,b}$&$V_{a,b+1}$&$\oplus_{b'\geq b+1, a'\geq b'}V_{a',b'}$&$-(\rhoa(H)-2+2b)$&$-2(b-1)$\\
&$V_{a,b}$&$V_{a-1,b}$&$\oplus_{a'\leq a-1, a'\geq b'}V_{a',b'}$&$\rhoa(H)-2+2a$&$4n+2a$\\\hline
\GF&$V_{m,\ell}$&$V_{m+1,\ell+1}$&$\oplus_{m'+\ell'\geq m+\ell+2}V_{m',\ell'}$&$-(\rhoa(H)+m+\ell)$&$-(m+\ell)$\\
&$V_{m,\ell}$&$V_{m-1,\ell+1}$&$\oplus_{m'-\ell'\leq m-\ell-2}V_{m',\ell'}$&$\rhoa(H)+m-\ell-8$&$14+m-\ell$\\
&$V_{m,\ell}$&$V_{m+1,\ell-1}$&$\oplus_{m'-\ell'\geq m-\ell+2}V_{m',\ell'}$&$-(\rhoa(H)-6+m-\ell)$&$6-m+\ell$\\
&$V_{m,\ell}$&$V_{m-1,\ell-1}$&$\oplus_{m'+\ell'\leq m+\ell-2}V_{m',\ell'}$&$\rhoa(H)-2+m+\ell$&$20+m+\ell$
\end{tabular}
\end{center}
\end{proof}

\bibliographystyle{amsalpha}
\bibliography{literatur.bib}
\end{document}